\documentclass[]{amsart}

\usepackage{geometry}
\usepackage{amsmath}
\usepackage{amssymb}
\usepackage{amsthm}
\usepackage{enumitem}
\usepackage{url}
\usepackage{hyperref}
\usepackage{graphicx}
\usepackage{xcolor}
\usepackage{pdfpages}
\usepackage{verbatim}
\usepackage{nicematrix}
\usepackage{thm-restate}
\usepackage{mathtools}

\usepackage{tikz}
\usetikzlibrary{decorations.pathreplacing,calligraphy}

\usepackage{caption}
\usepackage{subcaption}

\usepackage{bbm}

\newtheorem{theorem}{Theorem}

\numberwithin{equation}{section}

\numberwithin{theorem}{section}

\theoremstyle{definition}
    \newtheorem{assumption}[theorem]{Assumption}

   \newtheorem{corollary}[theorem]{Corollary}
   \newtheorem{example}[theorem]{Example}
   \newtheorem{definition}[theorem]{Definition}
   \newtheorem{problem}[theorem]{Problem}
   \newtheorem{proposition}[theorem]{Proposition}
   \newtheorem{remark}[theorem]{Remark}
   \newtheorem{lemma}[theorem]{Lemma}
   \newtheorem{notation}[theorem]{Notation}

\newcommand{\conv}[1]{\mathrm{conv}\{{#1}\}}

\newcommand{\NN}{{\mathcal{N}}}
\newcommand{\R}{{\mathbb{R}}}
\newcommand{\relint}{{\mathrm{Relint}}}

\newcommand{\onetwo}{\raisebox{-4pt}{\includegraphics[width=3mm]{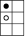}}}
\newcommand{\onefive}{\raisebox{-4pt}{\includegraphics[width=3mm]{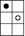}}}
\newcommand{\onethree}{\raisebox{-4pt}{\includegraphics[width=3mm]{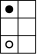}}}
\newcommand{\onesix}{\raisebox{-4pt}{\includegraphics[width=3mm]{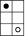}}}
\newcommand{\twotwo}{\raisebox{-4pt}{\includegraphics[width=3mm]{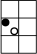}}}
\newcommand{\twothree}{\raisebox{-4pt}{\includegraphics[width=3mm]{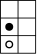}}}
\newcommand{\twosix}{\raisebox{-4pt}{\includegraphics[width=3mm]{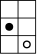}}}
\newcommand{\fourtwo}{\raisebox{-4pt}{\includegraphics[width=3mm]{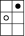}}}
\newcommand{\fourfive}{\raisebox{-4pt}{\includegraphics[width=3mm]{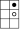}}}
\newcommand{\fourthree}{\raisebox{-4pt}{\includegraphics[width=3mm]{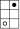}}}
\newcommand{\foursix}{\raisebox{-4pt}{\includegraphics[width=3mm]{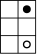}}}
\newcommand{\fivefive}{\raisebox{-4pt}{\includegraphics[width=3mm]{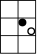}}}
\newcommand{\fivethree}{\raisebox{-4pt}{\includegraphics[width=3mm]{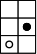}}}
\newcommand{\fivesix}{\raisebox{-4pt}{\includegraphics[width=3mm]{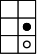}}}

\newcommand{\bigonetwo}{\raisebox{-4pt}{\includegraphics[width=4mm]{12bw.png}}}
\newcommand{\bigonefive}{\raisebox{-4pt}{\includegraphics[width=4mm]{15bw.png}}}
\newcommand{\bigonethree}{\raisebox{-4pt}{\includegraphics[width=4mm]{13bw.png}}}
\newcommand{\bigonesix}{\raisebox{-4pt}{\includegraphics[width=4mm]{16bw.png}}}
\newcommand{\bigtwotwo}{\raisebox{-4pt}{\includegraphics[width=4mm]{22bw.png}}}
\newcommand{\bigtwothree}{\raisebox{-4pt}{\includegraphics[width=4mm]{23bw.png}}}
\newcommand{\bigtwosix}{\raisebox{-4pt}{\includegraphics[width=4mm]{26bw.png}}}
\newcommand{\bigfourtwo}{\raisebox{-4pt}{\includegraphics[width=4mm]{42bw.png}}}
\newcommand{\bigfourfive}{\raisebox{-4pt}{\includegraphics[width=4mm]{45bw.png}}}
\newcommand{\bigfourthree}{\raisebox{-4pt}{\includegraphics[width=4mm]{43bw.png}}}
\newcommand{\bigfoursix}{\raisebox{-4pt}{\includegraphics[width=4mm]{46bw.png}}}
\newcommand{\bigfivefive}{\raisebox{-4pt}{\includegraphics[width=4mm]{55bw.png}}}
\newcommand{\bigfivethree}{\raisebox{-4pt}{\includegraphics[width=4mm]{53bw.png}}}
\newcommand{\bigfivesix}{\raisebox{-4pt}{\includegraphics[width=4mm]{56bw.png}}}

\graphicspath{{./Figures/}}

\title[
Enumeration of max-pooling responses with generalized permutohedra
]{Enumeration of max-pooling responses \\with generalized permutohedra}

\author[Escobar]{Laura Escobar}
\address{Department of Mathematics and Statistics\\ Washington University in St.\ Louis \\ 1 Brookings Drive\\ 
St. Louis, MO  63130}
\email{laurae@wustl.edu}

\author[Gallardo]{Patricio Gallardo}
\address{Department of Mathematics, 
University of California, Riverside\\ 
900 University Avenue\\ 
Riverside, CA 92521}
\email{pgallard@ucr.edu} 

\author[Gonz\'alez Anaya]{Javier Gonz\'alez Anaya} 
\address{Department of Mathematics, 
University of California, Riverside\\ 
900 University Avenue\\ 
Riverside, CA 92521 and Department of Mathematics, 
Harvey Mudd College, 
Claremont, CA 91711}
\email{javiergo@hmc.edu}

\author[Gonz\'alez]{Jos\'e L.\ Gonz\'alez} 
\address{Department of Mathematics, 
University of California, Riverside\\ 
900 University Avenue\\ 
Riverside, CA 92521}
\email{jose.gonzalez@ucr.edu} 

\author[Mont\'{u}far]{Guido Mont\'{u}far} 
\address{Departments of Mathematics and Statistics, University of California, Los Angeles, CA 90095\\ and Max Planck Institute for Mathematics in the Sciences, 04103 Leipzig} 
\email{montufar@math.ucla.edu}

\author[Morales]{Alejandro H.\ Morales}
\address{D\'epartement de Math\'ematiques, 
Universit\'e du Qu\'ebec \`a Montr\'eal\\
201 Av. du Président-Kennedy\\
Montréal}  
\email{morales\_borrero.alejandro@uqam.ca}

\begin{document}

\maketitle

\begin{abstract}
We investigate the combinatorics of max-pooling layers, which are functions that downsample input arrays by taking the maximum over shifted windows of input coordinates, and which are commonly used in convolutional neural networks. 
We obtain results on the number of linearity regions of these functions by equivalently counting the number of vertices of certain Minkowski sums of simplices. We characterize the faces of such polytopes and  obtain generating functions and closed formulas for the number of vertices and facets in a 1D max-pooling layer depending on the size of the pooling windows and stride, and for the number of vertices in a special case of 2D max-pooling. 
\end{abstract}

\section{Introduction} 

Convolutional neural networks are central tools in audio, image, and text processing 
that can identify complex data features through a hierarchy of computations 
\cite{NIPS2012_c399862d,Szegedy_2015_CVPR}. 
Part of the success of these architectures comes from including pooling layers, which downsample intermediate feature representations of the data and introduce invariance to local translations. 
We investigate the combinatorial complexity of \emph{max-pooling}, which is one of the most commonly used forms of pooling. A max-pooling function is a piecewise-linear function that takes an array as input and returns an array collecting the maximum values over different windows of input coordinates. 
The combinatorial analysis of max-pooling functions is interesting in its own right, as it involves distinctive classes of Minkowski sums of simplices and generalized permutohedra, and it contributes to a more complete theoretical understanding of convolutional networks. 



For piecewise linear functions, we may regard the number of linear regions as a complexity measure. The number of linear regions of the functions represented by neural networks with piecewise linear activation functions offers a combinatorial perspective to compare the representational power of different network architectures. In particular, this has been used to establish differences between deep and shallow network architectures \cite{pascanu2013number, NIPS2014_109d2dd3, pmlr-v49-telgarsky16}. 
The problem of enumerating the linear regions of the functions represented by different network architectures has received significant attention in recent years, 
with several advances for deep fully-connected networks \cite{arora2018understanding,
Serra-2018-bounding, 
Hinz-2019-framework}, 
convolutional networks \cite{pmlr-v119-xiong20a}, 
graph neural networks \cite{pmlr-v139-bodnar21a}, 
as well as the development of connections between neural networks and polyhedral theory \cite{huchette2023deep}, 
power diagrams \cite{NIPS2019_9712}, 
tropical geometry and polytopes \cite{pmlr-v80-zhang18i, 
Charisopoulos2018ATA, 
maragos2021tropical, 
NEURIPS2021_1b9812b9, 
montufar2021sharp}. 
Max-pooling layers have a distinctive combination of properties that is not well covered by previous works. A discussion of the topic appeared in the blog post \cite{TM}. 
Concretely, in contrast to other components of neural networks, the fixed weights and restricted connectivity of max-pooling layers implies that one cannot resort to genericity arguments to simplify the enumeration problem. This makes it difficult to obtain precise estimates beyond certain upper and lower bounds. 
     
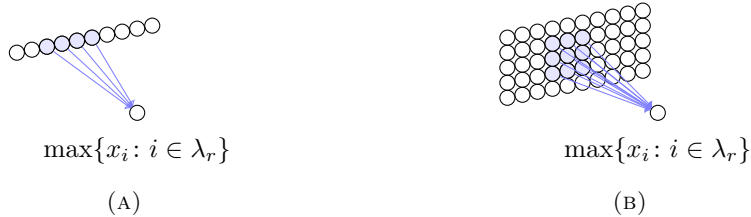
\begin{figure}
      \centering
  \begin{subfigure}[b]{0.45\textwidth}
    \centering
    \begin{tikzpicture}[x=.3cm, y=.3cm]
    
\foreach \o in {0,2,4,6}{\node [circle,draw, 
minimum size=.2cm, inner sep=0pt](o\o) at ({8-2+\o}, {-1 -.4 +.2*\o}) {}; 
\foreach \x in {1,...,4}{
     \foreach \y in {3}{ 
     \pgfmathsetmacro\mytemp{\y +.2*(\x-1+\o )}; 
     \pgfmathsetmacro\mytempd{\x-1 + \o}; 
     \draw [color=gray!50, ->, thin] (\x-1 + \o, \mytemp) -- (o\o); 
} } }

\foreach \x in {0,...,9}{
	\foreach \y in {3,...,3}{
		\pgfmathsetmacro\mytemp{\y +.2*\x}
		\node [circle,draw, fill=white,,minimum size=.2cm, inner sep=0pt](n\x\y) at (\x , \mytemp) {}; 
}}

\node [circle,draw, fill=blue!20, minimum size=.2cm, inner sep=0pt](o) at (8, -1) {}; 
	\foreach \x in {2,...,5}{\foreach \y in {3,...,3}{
		\pgfmathsetmacro\mytemp{\y +.2*\x}
		\node [circle, draw, minimum size=.2cm, inner sep=0pt, fill =blue!20] at (\x , \mytemp) {};
        \draw [color=blue!50,->, thick] (n\x\y) -- (o);   
}}

 \draw [decorate,  decoration = {brace,raise=20pt}] (0,3) --  (9,3+1.8) node[pos=0.5,above=22pt,black]{\hspace{-.5cm}$\Lambda$}; 

 \draw [decorate, decoration = {brace,raise=5pt}] (2,3.4) --  (5,4) node[pos=0.5,above=7pt,black]{\small $\lambda_j$}; 
 
\node at (19.5,-.7) {$\left(\max\{x_i\colon i\in \lambda_j\}\right)_{j\in\Lambda'}$};
 
 \draw [decorate, decoration = {brace,raise=8pt}] (12,-.2) -- (6,-1.4)  node[pos=0.5,below=10pt,black]{\hspace{.5cm}$\Lambda'$}; 

\end{tikzpicture}
    \caption{}
\end{subfigure}
\begin{subfigure}[b]{0.45\textwidth}
  \centering
  \begin{tikzpicture}[x=.3cm, y=.3cm]
 \draw [rounded corners, thin, fill = gray!10] (2.5,2) -- (2.5,5) -- (5.5,5.6) -- (5.5,2.6) -- cycle; 

\node [inner sep=.1pt] (la) at (5.2,7.6) {$\lambda_j$}; 
\draw (la) to [bend right=15] (4.5,5.6); 

\node [inner sep=.1pt] at (0,6.7) {$\Lambda$}; 
 
\foreach \x in {0,...,9}{
	\foreach \y in {1,...,5}{
		\pgfmathsetmacro\mytemp{\y +.2*\x}
		\node [circle,draw,minimum size=.2cm, inner sep=0pt](n\x\y) at (\x , \mytemp) {}; 
}}
\node [circle,draw=white,minimum size=.2cm, inner sep=0pt](o) at (8, -1) {}; 
	\foreach \x in {3,...,5}{\foreach \y in {2,...,4}{
		\pgfmathsetmacro\mytemp{\y +.2*\x}
		\node [circle,draw,minimum size=.2cm, inner sep=0pt, fill =blue!20] at (\x , \mytemp) {};
}}
\node [circle, fill=blue!20, draw, minimum size=.2cm, inner sep=0pt](o) at (10, 0) {}; 

\node[node distance=.5cm, below of = o] {$\max\{x_i\colon i\in \lambda_j\}$};
	\foreach \x in {3,...,5}{\foreach \y in {2,...,4}{
		\draw [color=blue!50,->, thin] (n\x\y) -- (o);}}

\end{tikzpicture}
  \caption{}
\end{subfigure}
   \caption{Illustration of max-pooling layer in (A) one dimension with a $1\times 4$ window $\lambda_j$ (blue nodes) and 
   (B) two dimensions with a $3\times 3$ window $\lambda_j$ (blue nodes).}
   \label{fig: 1d and 2d layers}
\end{figure}

%

A max-pooling function is a piecewise-linear function $f\colon \mathbb{R}^{\Lambda} \longrightarrow\mathbb{R}^{\Lambda'}$ such that $(f(x))_{j}=\max\{x_i\colon i\in\lambda_j\}$, where $\{\lambda_j\}$ is a fixed collection of pooling windows within $\Lambda$. 
See Figure~\ref{fig: 1d and 2d layers} for an illustration and Definition~\ref{def:SettingProblem} for details. 
We will study the number of linear regions of max-pooling functions by considering the equivalent problem of enumerating the faces of their Newton polytopes, which are Minkowski sums of standard simplices, as detailed in Section~\ref{sec:background}. 
Therefore, we are concerned with the following problem. 

\begin{problem} 
Let $\lambda_j\subseteq \Lambda$, $j\in\Lambda'$. 
For each $j\in \Lambda'$ consider the simplex $\Delta_{\lambda_
j} = \operatorname{conv}\{e_i\colon i\in \lambda_j\}$, where $e_i$ is the $i$th canonical vector in $\mathbb{R}^{\Lambda}$. 
What is the number of faces of each dimension, in particular vertices, 
of the Minkowski sum $P = \sum_{j\in \Lambda'} \Delta_{\lambda_
j}$? 
\end{problem} 

The general problem of counting the faces of Minkowski sums of polytopes is a subject of significant interest; see, e.g., \cite{Gritzmann-Sturmfels-1993,Fukuda-Weibel-2007,Sanyal-2009,Weibel-2012,
Adiprasito-Sanyal-2016,Karavelas-etal-2016,montufar2021sharp}. 
In particular, the upper bound theorem for Minkowski sums \cite{Adiprasito-Sanyal-2016} states that among all sums of polytopes with a given number of vertices, the sum of a Minkowski neighbourly family attains the maximum number of faces. An explicit formula for the maximum number of vertices was obtained in \cite[Thm.\ 3.7]{montufar2021sharp}. 
The more specific case of faces of sums of simplices has been studied in \cite{10.1093/imrn/rnn153,PRW,Agnarsson2009,Agnarsson2013,BBM}, obtaining combinatorial models for the faces. However, even with these combinatorial models, counting the faces of such polytopes is a  computationally hard problem (see Remark~\ref{rem: complexity faces}). 

In this paper we consider a class of polytopes obtained by endowing the $\lambda_j$'s with a certain structure that arises from the structure of max-pooling layers, see Definition~\ref{def:SettingProblem}. 
Since sums of standard simplices are generalized permutohedra, we are able to relate the faces of our polytopes with certain acyclic graphs. 
This is a new combinatorial model developed using \cite{PRW} and related to \cite{BBM}. 
In the case of vertices we use the structure of the $\lambda_j$'s to put these acyclic graphs in correspondence with walks in a directed graph. 
We then use the \emph{transfer-matrix method} to give generating functions for the number of vertices.
We obtain explicit closed formulas for the generating functions, linear recurrences for the number of vertices, and describe the asymptotics as the number of simplices tends to infinity. 
We now describe our contributions in more detail. 

\subsection{Faces of Minkowski sums of simplices}

Every face $F$ of $P$ can be uniquely written as a Minkowski sum $F = \sum_{r\in \Lambda'} F_{r}$ of faces $F_r\subseteq\Delta_{\lambda_r}$ of the summand polytopes. 
However, not every such sum 
is a face of $P$. 
In Proposition~\ref{prop: char vertices minkowski sum of simplices} of Section~\ref{sec_faces_dag} we give a criterion to determine whether a sum $F = \sum_{r\in \Lambda'} F_{r}$ of faces $F_r\subseteq\Delta_{\lambda_r}$ is a face of $P$. Moreover, if $F$ is a face of $P$ we describe its corresponding cone in the normal fan $\NN(P)$ of $P$ and its dimension. Our method consists
of constructing a directed graph that is acyclic if and only if $F$ is a face. The proofs of all the following results use this criterion. 

\subsection{One-dimensional input layers} 
Here we consider the case of one-dimensional input arrays. 
Given positive integers $n,k,s$, let $P_{n,k,s}$ denote the polytope given by the Minkowski sum of the $n$ simplices 
$\Delta_{\{si,si+1,\ldots,si+k-1\}}$ for $i=0,\ldots,n-1$. Let $b_n^{(k,s)}$ denote the number of vertices of the polytope $P_{n,k,s}$, which is also equal to the number of linearity regions of a max-pooling function over a $1 \times (s(n-1)+k)$ input with pooling windows of size $1\times k$ and stride $s$ (see Figure~\ref{fig: one layer}).

  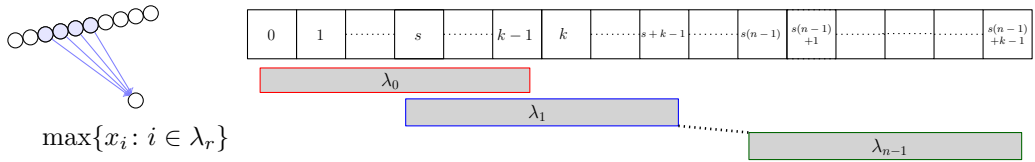
\begin{figure}[h]
    \centering
    \begin{tikzpicture}[x=.3cm, y=.3cm]
\foreach \x in {0,...,9}{
	\foreach \y in {3,...,3}{
		\pgfmathsetmacro\mytemp{\y +.2*\x}
		\node [circle,draw,minimum size=.2cm, inner sep=0pt](n\x\y) at (\x , \mytemp) {}; 
	}
}
\node [circle,draw=white,minimum size=.2cm, inner sep=0pt](o) at (8, -1) {}; 

	\foreach \x in {2,...,5}{\foreach \y in {3,...,3}{
		\pgfmathsetmacro\mytemp{\y +.2*\x}
		\node [circle,draw,minimum size=.2cm, inner sep=0pt, fill =blue!20] at (\x , \mytemp) {};}}
\node [circle,draw,minimum size=.2cm, inner sep=0pt, fill=blue!20](o) at (8, -1) {}; 
\node[node distance=.5cm, below left of = o] {$\max\{x_i\colon i\in \lambda_r\}$};
\foreach \x in {2,...,5}{\foreach \y in {3,...,3}{
			\draw [color=blue!50,->, thin] (n\x\y) -- (o);}}
\end{tikzpicture}
\setlength{\fboxrule}{0pt}
\scalebox{.7}{ 
\begin{tikzpicture}
\foreach \x in {0,...,15}{ 
\draw[thick, gray] (\x,0) -- (\x,1) -- (\x+1,1) -- (\x+1,0) -- cycle; 
}
\node[shift=({.5,.5})] at (0,0) {$0$};
\node[shift=({.5,.5})] at (1,0) {$1$};
\node[shift=({.5,.5})] at (2,0) {$\cdots$};
\node[shift=({.5,.5})] at (3,0) {$s$};
\node[shift=({.5,.5})] at (4,0) {$\cdots$};
\node[shift=({.5,.5})] at (5,0) {$k-1$};
\node[shift=({.5,.5})] at (6,0) {$k$};
\node[shift=({.5,.5})] at (7,0) {$\cdots$};
\node[shift=({0,0}),anchor=south west] at (8,0) {\rotatebox[]{45}{\fcolorbox{white}{white}{$s+k-1$}}};
\node[shift=({.5,.5})] at (9,0) {$\cdots$};
\node[shift=({0,0}),anchor=south west] at (10,0) {\rotatebox[]{45}{\fcolorbox{white}{white}{$s(n-1)$}}};
\node[shift=({0,0}),anchor=south west] at (11,0) {\rotatebox[origin=lB]{45}{\fcolorbox{white}{white}{$s(n-1)+1$}}};
\node[shift=({.5,.5})] at (12,0) {$\cdots$};
\node[shift=({.5,.5})] at (13,0) {$\cdots$};
\node[shift=({.5,.5})] at (14,0) {$\cdots$};
\node[shift=({0,0}),anchor=south west] at (15,0) {\rotatebox[]{45}{\fcolorbox{white}{white}{$s(n-1)+(k-1)$}}};

\draw[red,shift=({0,-1.125}),fill=gray!20, thick,rounded corners] (0,.125) -- (0,.875) -- (6,.875) -- (6,.125) -- cycle; 
\node[shift=({3,-1.125+.5})] {$\lambda_0$}; 

\draw[blue,shift=({3,-2}),fill=gray!20, thick,rounded corners] (0,.125) -- (0,.875) -- (6,.875) -- (6,.125) -- cycle; 
\node[shift=({6,-2+.5})] {$\lambda_1$}; 

\node[shift=({9.5,-2})] {$\cdots$}; 

\draw[green!80!black,shift=({10,-3}),fill=gray!20, thick, rounded corners] (0,.125) -- (0,.875) -- (6,.875) -- (6,.125) -- cycle; 
\node[shift=({13,-3+.5})] {$\lambda_{n-1}$}; 
\end{tikzpicture}
}

    \caption{The $1\times (s(n-1)+k)$ input with its $n$ pooling windows $\lambda_0,\ldots,\lambda_{n-1}$ of size $ 1 \times k$ and stride $s$. 
    }
    \label{fig: one layer}
\end{figure}

We prove in Theorem~\ref{lemma: gf vertices 1 dim} that the generating functions of the sequences $(b_n^{(k,s)})_{n\geq 1}$ are {\em rational} by using a characterization of faces of $P_{n,k,s}$ from Section~\ref{sec:acyclic.graph} and the {\em transfer-matrix method} (see Section~\ref{sec: transfer matrix}) from enumerative combinatorics. Moreover, we have closed forms for these generating functions for the cases of {\em large strides}, i.e.\ $\lceil k/2\rceil \leq s \leq k-2$, and {\em proportional strides}, i.e.\ $k=s(r+1)$ for a nonnegative integer $r$. 

\begin{theorem}[Theorems \ref{thm:LargeStrides} and \ref{thm:ProporStrides}]
Fix positive integers $k$ and $s$. 
\begin{enumerate}
    \item If $s \in \{ \lceil k/2\rceil,\ldots,k \}$, then the generating function of $(b_n^{(k,s)})_{n\geq 1}$  is given by
\[
1+\sum_{n\geq 1} b^{(k, s)}_{n} x^n
= \frac{1}{1-kx+(k-s)(k-s-1)x^2}.
\]

\item If $k=s(r+1)$ for a nonnegative integer $r$, then the generating function of $(b_n^{(k,s)})_{n\geq 1}$ is given by
\begin{align*}
1+\sum_{n\geq 1} b_{n}^{(k,s)} x^n = 
\frac{
1+(rs-s-2)x-(rs-1)x^{2}+sx^{r+1}
}
{
1 - 2(s+1) x + (s +1)^2x^{2} + s x^{r + 1} - s^{2}(r + 1) x^{r + 2} + s(r s -1)x^{r + 3} 
}.
\end{align*}
In particular, if $s=1$, we obtain
 \begin{align*}
1+\sum_{n\geq 1} b_{n}^{(k,1)} x^n = 
\frac{1+(k-4)x-(k-2)x^2 + x^k}
{     1 - 4x+ 4x^{2} + x^{k} - kx^{k+1} + (k-2)x^{k+2}}.
\end{align*}
\end{enumerate}
\end{theorem}

Note that when $s\geq k-1$, there is no overlap between the windows or just an overlap of one vertex and  so we have that $b^{(k,s)}_n=k^{n}$ (see Remark~\ref{rem: case r=0}). We also give asymptotics for the number of vertices for the first case above (Corollary~\ref{cor: asympt LargeStrides}), and for the general case (Corollary~\ref{cor: asymptotics gf vertices 1 dim}) using the Perron--Frobenius theorem. 

In the following result, we also calculate the number of facets of $P_{n,k,s}$ and in Corollary~\ref{cor: H representation of P_n,k,s} give the inequality description of this polytope. 

\begin{theorem}[Theorem~\ref{thm_facets_1d}]
Let $s,k$ be positive integers. If 
$k>s+1$, then the number of facets of $P_{n,k,s}$ is $(s+2)(n-1)+k$. 
If 
$1<k\le s+1$, then the number of facets of $P_{n,k,s}$ is $kn$.
\end{theorem}

\subsection{Two-dimensional input layers} 
Here we consider two-dimensional input arrays, focusing on a special setting. 
In the Euclidean space $\mathbb{R}^{3 \times n} \cong \mathbb{R}^{3n}$ with basis $\{e_{i,j} \ | \  0 \leq i \leq 2,\ 0 \leq j \leq n-1 \}$, 
let $Q_n$ be the Minkowski sum of the $2(n-1)$ simplices 
$\Delta_{i,j}=\conv{e_{i,j},e_{i,j+1},e_{i+1,j},e_{i+1,j+1}}$ for all 
$0\leq i \leq 1$ and $0 \leq j \leq n-2$. 
Let $V_n$ be the number of vertices of $Q_n$, which is also equal to the number of linearity regions of a max-pooling function with a $3 \times n$ input and pooling windows of size $2\times 2$ and stride one (see Figure~\ref{fig: board case 3xn}).

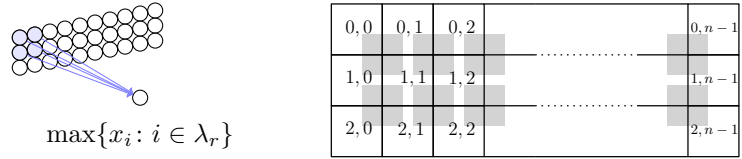
\begin{figure}[h]
  \centering
   \begin{tikzpicture}[x=.3cm, y=.3cm]
\foreach \x in {0,...,9}{
	\foreach \y in {1,...,3}{
		\pgfmathsetmacro\mytemp{\y +.2*\x}
		\node [circle,draw,minimum size=.2cm, inner sep=0pt](n\x\y) at (\x , \mytemp) {}; 
	}
}
\node [circle,draw=white,minimum size=.2cm, inner sep=0pt, fill=blue!20](o) at (8, -1) {}; 

	\foreach \x in {0,...,1}{\foreach \y in {2,...,3}{
		\pgfmathsetmacro\mytemp{\y +.2*\x}
		\node [circle,draw,minimum size=.2cm, inner sep=0pt, fill =blue!20] at (\x , \mytemp) {};}}
\node [circle,draw,minimum size=.2cm, inner sep=0pt](o) at (8, -1) {}; 
\node[node distance=.5cm, below of = o] {$\max\{x_i\colon i\in \lambda_r\}$};
\foreach \x in {0,...,1}{\foreach \y in {2,...,3}{
			\draw [color=blue!50,->, thin] (n\x\y) -- (o);}}
\end{tikzpicture}
\qquad \quad
  \includegraphics[scale=0.6]{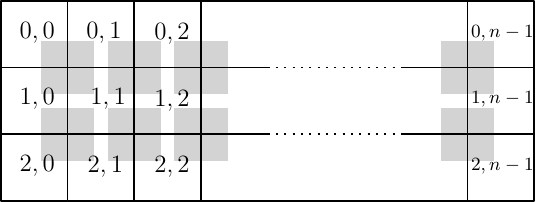}
  \caption{The $3\times n$ input with the $2n-2$ pooling windows of size $2\times 2$.}
  \label{fig: board case 3xn}
\end{figure}

\begin{theorem}[Theorem~\ref{recurrence.2D}]  \label{theorem.2D.3xn.intro}
The number $V_n$ of vertices of the polytope $Q_n$ 
is given by the generating function
\begin{equation*}
x \ + \ \sum_{n\geq 2} V_n \, x^n \  = \ x + 14x^2 + 150x^3+1536x^4+ \cdots = \frac{x+x^2-x^3}{1-13x+31x^2-20x^3+4x^4}. 
\end{equation*}
\end{theorem}

The case of a $2\times n$ input with $2 \times 2$ pooling windows is covered by the one-dimensional analysis from Section~\ref{sec:enumeration} and is discussed in Example~\ref{example.2D.2xn}.


\subsection{Article outline}
In Section~\ref{sec:background} we give background on transfer matrices, generating functions, and generalized permutohedra. 
In Section~\ref{sec:acyclic.graph} we relate the faces of Minkowski sums with directed acyclic graphs. In Section~\ref{sec:enumeration} we study the one-dimensional case and obtain generating functions and recurrences for the number of vertices as the number of windows increase, depending on the window size and stride. 
In Section~\ref{sec:2dim} we consider the two-dimensional case and a particular case in detail. 
We offer a few final remarks in Section~\ref{sec:conclusions}.  Code for the calculations in this article is available at \cite{code}.

\section{Background and notation}
\label{sec:background}

In this section, we describe the main objects we will study and the main technical tools used in our work.

\subsection{Polytopes and generalized permutohedra} \label{subssec: gen perm}

Let $P\subset \R^d$ be a polytope.
Given a linear functional $w: \R^d \rightarrow \R$, denote the $w$-maximal face of $P$ by $P_w = \{x \in P \, : \, w(x) = \max_{y \in P} w(y)\}$. 
The (outer) normal fan $\NN(P)$ of $P$ is the complete fan in $\R^d$ whose cones are
\[
\NN(P)_F := \{w \in \R^d \, : \, P_w \supseteq F\}
\]
for each nonempty face  $F$ of $P$. 
The relative interior $\relint(\sigma)$ of a cone $\sigma$ is the interior of $\sigma$ inside its affine span. In particular, 
\[
\relint(\NN(P)_F) := \{w \in \R^d \, : \, P_w =F\}.
\]
The face poset of $\NN(P)$ is isomorphic to the reverse of the face poset of $P$.
 
Throughout we denote the set $\{1,2,\ldots,d\}$ as $[d]$. The \emph{permutohedron} is the polytope 
\[
\Pi_d = \conv{x\in\R^d\ | \ \text{$x_1,\ldots,x_d$ is a permutation of $[d]$} } \subseteq \left\{x\in\R^d \ \left| \  x_1+\cdots+x_d=\binom{d+1}{2}\right.\right\} \subseteq\mathbb{R}^d
.\]
Since $\Pi_d$ is $(d-1)$-dimensional, it is common to take its normal fan in $\R^d/\R(1,\ldots,1)$ rather than in $\R^d$ and we do so in this paper.
Thus, the normal fan of $\Pi_d$ is the complete simplicial fan in $\R^d/\R(1,\ldots,1)$ with a cone for each ordered partition of $[d]$, as follows.
Given an ordered partition $\mathcal{S}=(S_1,\ldots, S_k)$ of $[d]$, the associated cone is cut out by 
    \begin{equation}\label{eq-cones-perm}
    \begin{cases}
    x_a=x_b, & a,b\in S_i  \text{ for some } i, \\
    x_a\ge x_b, & a\in S_i,\ b\in S_j  \text{ for some } i<j . 
    \end{cases}
    \end{equation}
    
A {\em generalized permutohedron} is a polytope $P$ whose normal fan  $\mathcal{N}(P)$ in $\R^d/\R(1,\ldots,1)$ is a coarsening of the normal fan of $\Pi_d$. See  \cite{10.1093/imrn/rnn153} and \cite{PRW} for more details on generalized permutohedra.

As shown in \cite[Prop. 6.3]{10.1093/imrn/rnn153}, generalized permutohedra include Minkowski sums of standard simplices $\sum_j \operatorname{conv}\{e_i \colon i\in \lambda_j\}$, where $\sum_j$ denotes Minkowski sum and $\operatorname{conv}\{e_i \colon i\in \lambda_j\}$ denotes the simplex with vertices $e_i$ for $i\in \lambda_j$.

\subsection{Max-pooling and vertices of Minkowski sums} 
\label{sec:newton-polytope}

The max-pooling functions that we consider in this paper are defined as follows. The input data has the format of a box $\Lambda$ and the pooling windows $\lambda_j$ are shifts of a smaller box across positions inside of $\Lambda$. 

\begin{definition} 
\label{def:SettingProblem}
Let $\Lambda = \{0,1,\ldots, K_1-1\}\times \cdots\times \{ 0,1,\ldots, K_\nu-1
\}$ and 
$\lambda = \{0,\ldots, k_1-1\}\times\cdots\times\{0,\ldots,k_\nu-1\}$ for some $\nu\in\mathbb{Z}^{+}$ and $K_l,k_l\in\mathbb{Z}^{+}$ with $K_l\geq k_l$. 
Further, let $s\in\mathbb{Z}^{+}$ and $\lambda_r  = \lambda + s r$ for any $r$ in $\Lambda' = \{ r \in \Lambda \colon \lambda + s r  \subseteq \Lambda\}$. 
A max-pooling layer with inputs of format $\Lambda$, \emph{pooling windows} or \emph{receptive fields} of format $\lambda$, and \emph{stride} $s$ is a function $f\colon \mathbb{R}^\Lambda \to \mathbb{R}^{\Lambda'};\; (x_i)_{i \in \Lambda} \mapsto (\max\{x_i\colon i\in \lambda_r\})_{r\in \Lambda'}$. See Figure~\ref{fig: 1d and 2d layers}.
\end{definition}

The connection between max-pooling functions 
and polytopes is described by the following result, which is closely related to well-known results within tropical geometry (see \cite[Thm.\ 1.13]{JoswigBook}) and their discussion in the context of neural networks \cite{pmlr-v80-zhang18i,montufar2021sharp}. 
We recall that a convex piecewise-linear function $f\colon \mathbb{R}^{\Lambda}\longrightarrow \mathbb{R}^{\Lambda'}$ defines a polyhedral complex within its domain by considering its linearity regions together with their intersections. 
The following result relates such polyhedral complex with the faces of a suitable polytope for the particular case of max-pooling layers. The result follows from \cite[Thm.\ 1.13]{JoswigBook} and the discussion of Newton polytopes of max-out networks in \cite[Sec.\ 2.4]{montufar2021sharp}. 

\begin{proposition}     \label{connection.between.linearity.regions.and.face.counting}     
Consider the max-pooling function $f\colon \mathbb{R}^{\Lambda} \longrightarrow \mathbb{R}^{\Lambda'}$, $(f(x))_j=\max\{x_i\colon i\in\lambda_j\}$. Then, there is an inclusion-reversing bijection between the faces of the polyhedral complex of $f$ and the cones in the normal fan of the polytope $\sum_j \operatorname{conv}\{e_i \colon i\in \lambda_j\}$.  
\end{proposition} 

It follows that the number of linearity regions of $f$ is equal to the number of vertices of the polytope $\sum_j \operatorname{conv}\{e_i \colon i\in \lambda_j\}$. 

\begin{remark} \label{rem: complexity faces}
Counting faces of Minkowski sums of simplices is a computationally hard problem. Indeed, when the simplices are line segments, i.e.\ $|\lambda_i|=2$, the corresponding polytope is called a {\em zonotope}, and the number of vertices correspond to counting {\em acyclic orientations} in the graph with edges given by the pairs $\lambda_i$ (see \cite[\S 8.6]{10.1093/imrn/rnn153}). 
Counting acyclic orientations of  graphs is {\em $\#P$-complete} \cite{vertigan_welsh_1992}. 
\end{remark}

\subsection{Transfer matrix method} \label{sec: transfer matrix}

We review the main tools for computing
generating functions and collect a few standard results (see, e.g.\ \cite[Sec.\ 4.7]{EC1}) 
that we will use to describe the number of vertices of some polytopes in Section~\ref{sec:enumeration} and Section~\ref{sec:2dim}. 

Given a directed graph $D=(V,E)$ with $V=\{v_1,\ldots,v_p\}$ and edge weights $w:E\to \mathbb{Q}$, a length-$n$ walk $\Gamma$ in $D$ is a sequence $s_1\cdots s_n$ of $n$ directed edges $s_i$ in $E$ respecting edge directions, i.e. the target of $s_i$ equals the source of $s_{i+1}$ for all $i =1, \ldots,n-1$. The {\em weight} of a walk $\Gamma$ is $w(\Gamma)=w(s_1)\cdots w(s_n)$.

Let $A_{i,j}(n):=\sum_{\Gamma} w(\Gamma)$ where the sum is over all length $n$ walks $\Gamma$ in $D$ from $v_i$ to $v_j$. In the case that all $w(e)=1$ then $A_{ij}(n)$ is just the number of length-$n$ walks in $D$ from $v_i$ to $v_j$. Let $A$ be the $p\times p$ matrix with $(i,j)$-th entry 
$A_{i,j} :=
A_{i,j}(1)=\sum_{e} w(e)$, where the sum is over all edges $e$ from $v_i$ to $v_j$. This is the \emph{adjacency matrix} of $D$. 
The following standard result relates $A_{i,j}(n)$ to entries of powers of matrix $A$.

\begin{theorem}[{e.g.~\cite[Thm.\ 4.7.1]{EC1}}]
\label{thm:An}
Let $D$ be a digraph as above with adjacency matrix $A$ and $n \in \mathbb{N}$. Then $A_{i,j}(n)$ equals the $(i,j)$th entry of $A^n$. 
\end{theorem}
 
Next, we recall that given a sequence of numbers 
$\{b_n\}_{n=0}^{\infty }$, its generating function is given by $\sum_{n \geq 0}b_{n}x^n$ (sometimes we shift the index for notational convenience). 
In our particular case, we use 
the transfer-matrix method  to evaluate the generating function for the number $A_{i,j}(n)$ of walks on a digraph.

\begin{theorem}[{e.g.~\cite[Thm.\ 4.7.2]{EC1}}] \label{thm: transfer matrix}
Let $D$ be a digraph as above. Fix $i,j$, $1\leq i,j \leq p$, and let $F_{i,j}(D,x) = \sum_{n\geq 0} A_{i,j}(n) x^n$ be the generating function for walks in $D$ from $v_i$ to $v_j$. Then 
\[
F_{i,j}(D,x) = \frac{(-1)^{i+j}\det(I-xA 
; j,i)}{\det(I - x A)},
\]
where $(B; j,i)$ denotes the matrix $B$ with the $j$th row and $i$th column removed. 
\end{theorem}

\begin{remark} \label{rem: rel to charpoly and deg of denom}
Note that the polynomial $Q(x)=\det(I-xA)$ in the denominator above is related to the characteristic polynomial $P(x)=\det(xI-A)$ of $A$ by $Q(x)=x^p P(1/x)$. Thus, the degree of $Q(x)$  is $p-m_0$ where $m_0$ is the multiplicity of the eigenvalue $0$ in $A$.
\end{remark}

\begin{example} \label{ex: k=3 and s=1}
Consider the directed graph in Figure~\ref{fig: example digraph}. The walks in $D$ of length $n$ correspond to words $w_1\cdots w_{n+1}$ where $w_i \in \{0,1,2\}$ (the vertices of $D$), and there is  no appearance of $11$ nor $20$. 
By Theorem~\ref{thm:An} the number $a_{n+1}$ of such words of size $n+1$  is equal to $\sum_{i,j=1}^3 A_{i,j}(n)$ and by Theorem~\ref{thm: transfer matrix} we have that 
\[
\sum_{n\geq 0} a_{n+1} x^n = \frac{\sum_{i,j=1}^3 (-1)^{i+j} Q_{i,j}(x)}{Q(x)},
\]
where $Q(x)=\det(I - x A)$ and $Q_{i,j}(x)$ is the determinant of the submatrix $(Q(x);j,i)$. 
By direct calculation (this is the same as \cite[Ex.\ 4.7.6]{EC1}) we find 
\begin{equation} \label{eq: gs ex k=3 s=1}
\sum_{n\geq 0} a_{n+1} x^n = 3+7x+16x^2+36x^3+81x^4+\cdots = \frac{3+x-x^2}{1-2x-x^2+x^3}. 
\end{equation}

\begin{figure}
    \centering
       \raisebox{-30pt}{ \includegraphics[scale=0.8]{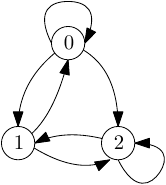}}
       \qquad     
       $A = \begin{bmatrix}
 1 & 1 & 1 \\
 1 & \textcolor{blue}{0} & 1 \\
 \textcolor{red}{0} & 1 & 1 
 \end{bmatrix}$
    \caption{A directed graph $D$ and its adjacency matrix $A$ from Example~\ref{ex: k=3 and s=1}.}
    \label{fig: example digraph}
\end{figure}

\end{example}

\subsection{Generating functions and  asymptotics}
The following results are the main tool to find the explicit form of the generating functions coming from Theorem \ref{thm: transfer matrix}. 
\begin{theorem}[{\cite[Thm.\ 4.1.1]{EC1}}]
\label{thm:GeneratingFunction}
Let $c_1,\ldots, c_d$ be a sequence of complex numbers, $c_d\neq 0$. The following are equivalent for a function $f:\mathbb{N}\to \mathbb{C}$.
\begin{itemize}
    \item[(i)]
    The formal power series of function $f(n)$ has the closed  form
    \[
    \sum_{n\geq 0} f(n)x^n = \frac{P(x)}{Q(x)}, 
    \]
    where $Q(x)=1+c_1x+\cdots + c_dx^d$ and $P(x)$ is a polynomial of degree less than $d$. 
    \item[(ii)] For all $n\geq 0$, $f(n+d)+ c_1f(n+d-1)+\cdots + c_d f(n) =0$. 
    \item[(iii)] For all $n\geq 0$, 
    \[
    f(n) = \sum_{i=1}^k P_i(n)\gamma_i^n,
    \]
    where $Q(x)=\prod_{i=1}^k (1-\gamma_ix)^{d_i}$, the $\gamma_i$ are distinct and nonzero, 
    and $P_i(n)$ is a polynomial of degree less than $d_i$. 
    \end{itemize}
\end{theorem}

For the asymptotic behavior of the number of vertices in Section \ref{sec:asymptotics}, we use the next standard result that follows from the result above. See \cite[Thm.\ 7.10]{B} and  \cite[Sec.\  V.5]{FS_book} for more details. 

\begin{theorem}[{see, e.g.\ \cite[Thm.\ 7.10]{B}}]  \label{thm:asymptotics}
Let $f(x) = \sum_{n \geq 0} s_n x^n =  P(x)/Q(x)$ be a rational function with $Q(0)\neq 0$, 
assume that $P(x)$ and $Q(x)$ do not have roots in common, and $Q(x)$ has a unique root $r_1$ of smallest modulus. 
Then the exponential growth rate of $s_n$ is $|z_1|$, where $z_1 = 1/r_1$. 
\end{theorem}

We recall that given an analytic function $F(x)$, a {\em dominant singularity} is a singularity of minimum modulus. For generating functions coming from the transfer-matrix method (Theorem~\ref{thm: transfer matrix}), the denominator is of the form $\det(I-xA)$, thus the dominant singularity is $1/\lambda$ where $\lambda$ is the eigenvalue of maximum modulus of the associated adjacency matrix $A$.

Given a $p\times p$ nonnegative matrix $A$, its {\em dependence graph} is the directed graph with vertices $\{1,2,\ldots,p\}$ and edges $i\to j$ if $A_{ij}\neq 0$. A square nonnegative matrix $A$ is said to be {\em irreducible} if its dependence graph is strongly connected, i.e. each vertex has a directed path to every other vertex.

 The generating functions from the transfer-matrix method (see Theorem~\ref{thm: transfer matrix}) are of the following form. Let $A$ be $p\times p$  nonnegative matrix (e.g. the adjacency matrix of a directed graph $D$ with positive edge weights) which is also irreducible.  For $1\leq i,j\leq p$, let
\begin{equation} \label{eq: def matrix series}
F^{\langle i,j\rangle}(x) := \sum_{n\geq 0} (A^n)_{i,j} x^n  \,=\,\left( (I-xA)^{-1}\right)_{i,j}   \,=\,  \frac{(-1)^{i+j}\det(I-xA 
; j,i)}{\det(I - x A)},
\end{equation}
where $(B; j,i)$ denotes the matrix $B$ with the $j$th row and $i$th column removed. For more details on matrices over formal power series, like $(I-xA)^{-1}$, see \cite[Sec.\ 1.1.10]{CE}. The following result is a consequence of the famous Perron--Frobenius theorem and guarantees a unique and simple dominant singularity of the generating functions $F^{\langle i,j\rangle}(x)$. In particular, this determines the asymptotics of the coefficients of the generating functions from the transfer-matrix method.

\begin{theorem}[{\cite[Thm.\ V.7]{FS_book}}] \label{thm: main asymptotic theorem}
Let $A$ be a square nonnegative irreducible matrix and let $F^{\langle i,j\rangle}(x)$ be defined as in \eqref{eq: def matrix series}. Then all entries $F^{\langle i,j\rangle}(x)$ have the same radius of convergence $\rho=\lambda^{-1}_1$, where $\lambda_1$ is the largest positive eigenvalue of $A$ (equivalently, $1/\lambda_1$ is the smallest positive root of $\det(I-xA)$). 
Moreover, if 
$F^{\langle i,j \rangle}(x)=:\sum_{n\geq 0} c_nx^n$, 
then 
\[
\lim_{n\to \infty} \frac{1}{n} \ln\left(
c_n
\right) = \ln \lambda_1.
\]
\end{theorem}

\begin{example} \label{ex: asymptotics k=3 and s=1}
Continuing with Example~\ref{ex: k=3 and s=1}, the $3\times 3$ matrix $A$ is irreducible since the dependence graph $D$ is strongly connected. The smallest root of the polynomial $\det(I-xA)=1-2x-x^2+x^3$ is $\rho=1/\lambda_1 \approx 0.44504$. Then by Theorem~\ref{thm: main asymptotic theorem} we have that the coefficient $a_{n+1}$ in \eqref{eq: gs ex k=3 s=1} satisfies
\[
\lim_{n\to \infty} \frac{1}{n} \ln a_{n+1} =  \ln \lambda_1 \approx 0.8096. 
\]

\end{example}

\section{Faces of Minkowski sums of simplices and directed acyclic graphs} 
\label{sec:acyclic.graph}

Consider a family of nonempty subsets $\lambda_{0},\dots,\lambda_{n-1}\subseteq [d]$ 
and the polytopes
\[
  \Delta_{\lambda_{i}}=\conv{e_{j}:j\in\lambda_{i}}\subseteq\R^{d}\text{ for }i=0,\dots, n-1.
\]
Throughout this section we consider the Minkowski sum $P=\sum_i \Delta_{\lambda_{i}}$. Our goal is to study the Minkowski sum $F=F_{0}+\cdots+F_{n-1}$ of faces $F_{i}\subseteq\Delta_{\lambda_{i}}$. In Proposition~\ref{prop: char vertices minkowski sum of simplices} we give a criterion to determine whether $F$ is a face of $P$, describe its corresponding cone in $\NN(P)$ and, consequently, its dimension. 
Our method consists on constructing a directed graph that is acyclic if and only if $F$ is a face. 

Every face $F\subseteq P$ can be written 
as the Minkowski sum $F= F_0+\cdots +F_{n-1}$ of some faces $F_i$ of $\Delta_{\lambda_{i}}$. 
However, not every choice of faces $F_{i}\subseteq \Delta_{\lambda_{i}}$ adds up to a face of $P$. 
In fact, $F\subseteq P$ is a face of $P$ if and only if $F= F_0+\cdots +F_{n-1}$ for some faces $F_i$ of $\Delta_{\lambda_{i}}$ such that there exists a linear function, independent of $i$, whose set of maximizers over $\Delta_{\lambda_{i}}$ is $F_i$ for all $i$. Furthermore, the decomposition $F = \sum_i F_i$ of any nonempty face $F$ is unique; see \cite[Lem.\ 2.1.4]{Gritzmann-Sturmfels-1993} and \cite[Prop.\ 2.1]{FUKUDA20041261}. 

\subsection{Digraph associated  to the Minkowski sum of faces}\label{sec_faces_dag}

Let $\Pi=(F_0,\ldots,F_{n-1})$ be a list where each $F_i$ is a face of $\Delta_{\lambda_i}$.
We now define a graph $G_\Pi$ that determines whether $\sum_i F_i$ is a face of $P$.
For each $i=0,\dots,n-1$, denote the set of indices of vertices of $F_i$ by
\[
    V(F_i)=\{a\in[d] \mid e_a \text{ is a vertex of } F_i\} \subseteq \lambda_i. 
\] 
First, let $\sim_\Pi$ denote the equivalence relation on $[d]$ obtained as the transitive closure of the relation
\[
  \{(a,b)\mid a,b\in V(F_i) \text{ for some } i\}.
\]
We denote the equivalence class of any $a\in [d]$ by $\bar a\in [d]/\sim_\Pi$.

\begin{definition}
    Let $\Pi$ be a list of faces as above. Then, $G_{\Pi}$ is the digraph with vertex set $V(G_{\Pi})=[d]/\sim_\Pi$, and such that there is an edge $\bar b\to\bar a$ between $\bar a, \bar b\in V(G_{\Pi})$ if and only if $\bar b\cap V(F_i)\neq \varnothing$ and $\bar a \cap ({\lambda_i}\setminus V(F_i))\neq \varnothing $ for some $i=0,\ldots,n-1$. 
    Note that loops $\bar a\to\bar a$ are allowed.
\end{definition}
With this we are ready to introduce the main result of the this section. 

\begin{proposition}\label{prop: char vertices minkowski sum of simplices}
Consider the sum $F=F_0+\cdots+F_{n-1}$ with $F_i$ a face of $\Delta_{\lambda_i}$ and let $\Pi=(F_0,\dots,F_{n-1})$. 
Then, $F$ is a face of $P$ if and only if $G_{\Pi}$ is acyclic. 
Moreover, if $F$ is a face, then $\dim(F)=d-|[d]/\sim_{\Pi}|$. 
\end{proposition}

\begin{figure}
    \centering
    \includegraphics[scale=0.9]{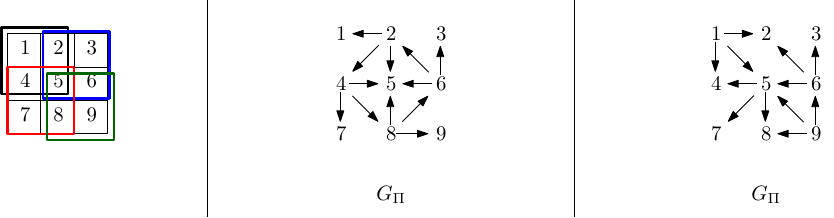}
    \caption{Illustration of Proposition~\ref{prop: char vertices minkowski sum of simplices} on a $3\times 3$ input with a $2\times 2$ pooling layer depicted on the left. 
    On the middle we have  $\Pi=(e_2,e_6,e_8,e_4)$, giving rise to a cycle in $G_\Pi$. On the right $\Pi=(e_1,e_6,e_9,e_5)$ and $G_\Pi$ is acyclic. 
    }
    \label{fig:2by2-example}
\end{figure}

\begin{example}
Figure~\ref{fig:2by2-example} shows two examples in the case where $\lambda_0=\{1,2,4,5\},\, \lambda_1=\{2,3,5,6\},\,\lambda_2=\{5,6,8,9\}$ and $\lambda_3=\{4,5,7,8\}\subseteq[9]$. 
In the example on the center panel, 
we consider the faces with $V(F_0)=\{2\}$, $V(F_1)=\{6\}$, $V(F_2)=\{8\}$, $V(F_3)=\{4\}$. 
The sum of faces $F_0+F_1+F_2+F_3$ is not a face of $P$ because $G_\Pi$ has a cycle. 
In the example on the right panel, 
we consider $V(F_0)=\{1\}$, $V(F_1)=\{6\}$, $V(F_2)=\{9\}$, $V(F_3)=\{5\}$. In this case $G_\Pi$ is acyclic and hence the corresponding sum is a face of $P$. 
\end{example}

\begin{remark}  \label{remark.subgraph.removing.sinks}
For a given list of faces $\Pi$, let $\Gamma_{\Pi}\subseteq G_{\Pi}$ be the subgraph obtained by removing from $G_{\Pi}$ all vertices that are sinks. Then, we remark that $G_{\Pi}$ is acyclic if and only if $\Gamma_{\Pi}$ is acyclic. This subgraph will be used later on in Section~\ref{sec:2dim}.
\end{remark}

We present the proof of Proposition~\ref{prop: char vertices minkowski sum of simplices} in parts. First let us prove that the graphs arising from faces must be acyclic.

\begin{lemma}
Let $\Pi=(F_0,\dots,F_{n-1})$ and $F$ be as in Proposition~\ref{prop: char vertices minkowski sum of simplices}. 
If $F$ is a face of $P$, then $G_\Pi$ is acyclic.
\end{lemma} 

\begin{proof}  
    Since $F$ is a face of $P$, every linear functional $h$ such that $P_h=F$ is also such that $(\Delta_{\lambda_i})_h=F_i$ for all $i$, by \cite[Prop.\ 2.1]{FUKUDA20041261}. 
    In particular, $h$ is constant for every vertex in a fixed equivalence class $\bar a\in V(G_\Pi)$.  Similarly, if there is an edge $\bar b\to\bar a$ between two different vertices of $G_{\Pi}$, then $h(b)>h(a)$ for every $a\in\bar a$ and $b\in \bar b$, by definition.  

    Assume $G_\Pi$ is not acyclic and contains a cycle ($k>1$) or loop ($k=1$) of the form 
    \[
    \overline{a_1}\to\overline{a_2}\to\cdots\to\overline{a_k}\to\overline{a_1}.
    \]
   Then, this would imply $h(a_{1})>h(a_{2})>\cdots >h(a_{k})>h(a_{1})$ for $a_i\in \bar a_i$, $i=1,\ldots, k$, which is clearly impossible. 
\end{proof}

To prove the converse to the previous lemma and the dimension statement in Proposition~\ref{prop: char vertices minkowski sum of simplices}, we will leverage the fact that $P$ is a generalized permutohedron \cite{10.1093/imrn/rnn153}, whose normal fan is a coarsening of the normal fan of the permutohedron, i.e., the braid arrangement \cite[Prop.\ 3.2]{PRW}. 
In \cite[Sec.\ 3]{PRW} it was shown that any cone in a normal fan of a generalized permutohedron corresponds in a natural way to a preposet.

A \emph{preposet} on $[d]$ is an equivalence relation $\sim$ together with a partial order $\preceq$ on $[d]/\sim$. 
Given a list $\Pi=(F_0,\ldots,F_{n-1})$, where each $F_i$ is a face of $\Delta_{\lambda_i}$, the condensation $Q_{\Pi}$ of $G_{\Pi}$ defines a preposet $([d],\sim_{Q_{\Pi}},\preceq_{Q_{\Pi}})$ on $[d]$. Explicitly, the vertices of the condensation naturally correspond to the elements of $[d]/\sim_{Q_{\Pi}}$, where
\begin{align*}
    a\sim_{Q_\Pi} b &\Leftrightarrow \bar{a}\text{ and }\bar{b}\text{ lie in the same strongly connected component of }G_{\Pi}\\ 
                    &\Leftrightarrow \bar{a}=\bar{b}\text{ or $\bar a$ and $\bar b$ lie in the same directed cycle of $G_\Pi$}.
\end{align*}
Let us denote the equivalence class of $a\in [d]$ under $\sim_{Q_\Pi}$ by $\hat{a}$. It follows from the definition above that $\overline{a}\subseteq\hat{a}$, with equality if and only if $G_\Pi$ is acyclic.

Next, the partial order $\preceq_{Q_\Pi}$ is defined to be the transitive closure of	 
\[
\hat a\preceq_{Q_\Pi} \hat b
\  \text{ if } \  
\bar b = \bar a \text{ or } \bar b\to\bar a,
\]
which corresponds precisely to the edges of the condensation.
 
In what follows, we will abuse notation and use $Q_{\Pi}$ interchangeably to denote either the condensation of $G_{\Pi}$ or the its induced poset $([d],\sim_{Q_{\Pi}},\preceq_{Q_{\Pi}})$. Note that any loop of $G_\Pi$ disappears upon taking its condensation.

Clearly, different $G_{\Pi}$ may give rise to the same preposet $Q_{\Pi}$.

\begin{example}
  Consider the simplices $\Delta_{\lambda_{i}}\subseteq\R^{5}$, where $\lambda_i:=\{i,i+1,i+2\}$ for $i=0,1,2$. Then, the two lists of faces
  \[
    \Pi=\left(\Delta_{\{01\}},\Delta_{\{123\}},\Delta_{\{23\}}\right)
    \quad\text{ and }\quad
    \Pi'=\left(\Delta_{\{012\}},\Delta_{\{123\}},\Delta_{\{23\}}\right)
  \]
  give rise to digraphs such that
  \[
    V(G_{\Pi})=V(G_{\Pi'})=\{\{4\},\{0,1,2,3\}\}
  \]
  and have an edge $\bar 0\to \bar 4$. 
  However, the face $\Delta_{\{01\}}$ produces two loops $\bar 0 \to \bar 2$ and $\bar 1\to \bar 2$
  at the vertex $\{0,1,2,3\}$ of $G_{\Pi}$, while these loops are not present in $G_{\Pi'}$. 
  On the other hand, $Q_{\Pi}=Q_{\Pi'}$ is the preposet $\{4\}\preceq \{0,1,2,3\}$.
  In particular, Proposition~\ref{prop: char vertices minkowski sum of simplices} implies that the Minkowski sum corresponding to $\Pi$ is not a face, but the one of $\Pi'$ is. The dimension of the corresponding face is $5-2=3$.
\end{example}

Given a preposet $Q=([d],\sim,\preceq)$, let $\sigma_Q\subseteq \R^d/(1,\ldots,1)\R$ be the cone cut out by $x_a\le x_b$ for all $a,b\in[d]$ such that $\hat a\preceq \hat b$. 
To simplify notation, for $\Pi$ as in Proposition~\ref{prop: char vertices minkowski sum of simplices} we use the notation $\sigma_{\Pi}:=\sigma_{Q_\Pi}$.
      
\begin{lemma}
Let $F$ and $\Pi=(F_0,\dots,F_{n-1})$ be as in Proposition~\ref{prop: char vertices minkowski sum of simplices}.
If $G_\Pi$ is acyclic, then $F$ is a face of $P$, $\NN(P)_F=\sigma_{\Pi}$, and $\dim(F)=d-|[d]/\sim_{\Pi}|$.
\end{lemma}

\begin{proof}
First note that, since $G_\Pi$ is acyclic, the relations $\sim_{G_\Pi}$ and $\sim_{Q_\Pi}$ are the same and $\bar{a}=\hat{a}$ for all $a\in[d]$. Consider a linear extension $L$ of $Q_\Pi$ (this is, a total order compatible with the partial order) and let $h\in \relint(\sigma_L)$. 
We claim that $F$ is the $h$-maximal face of $P$, i.e., $P_h = F$. To show that  $F\supseteq P_h$ it is enough to check that for each $F_i$ with $i=1,\dots, n-1$ we have:

\begin{enumerate}[label=(\roman*)]
    \item\label{minkowski proof item 1} $h$ is constant along $F_i$, and 
    \item $h(v) > h(w)$ for all possible pairs $v\in F_i$ and $w\in \Delta_{\lambda_{i}}\setminus F_i$.
\end{enumerate}

Indeed, this would guarantee that $F=F_1+\cdots+F_{n-1}$ is an $h$-maximal set. Note that it is possible to check both of these claims on the vertices of the $\Delta_{\lambda_{i}}$ because they are convex polytopes. 

Let $v=e_a$ and $w=e_b$ for some $a,b\in[d]$. Then, if $v,w\in F_i$ for some $i$ this would imply that $\hat{a}=\hat{b}$, so $h(v)=h(w)$ by definition; this proves \ref{minkowski proof item 1}. Similarly, if $v\in F_i$ and $w\in \Delta_{\lambda_{i}}\setminus F_i$ for some $i$, then $\hat{b}\preceq_{Q_\Pi} \hat{a}$, so $h(w)\leq h(v)$. However, equality is ruled out because $G_\Pi$ is acyclic. Indeed, $h(v)=h(w)$ if and only if $\hat{a}=\hat{b}$, which means that $\bar{a}=\bar{b}$. On the other hand, having $v\in F_i$ and $w\in \Delta_{\lambda_{i}}\setminus F_i$ means that there is an edge $\bar{a}\to\bar{b}$ in $G_\Pi$. These two observations would imply the existence of a loop $\bar{a}\to\bar{a}$, which contradicts the acyclicity of $G_\Pi$. 

Since $F\supseteq P_h$, then $F$ contains all the $h$-maximal points of $P$. 
Thus, to show $F=P_h$ it suffices to show that $h$ is constant along $F$.
Since $F=F_1+\cdots+F_m$, this can be further reduced to showing that $h$ is constant along each $F_i$, which is precisely what we showed for \ref{minkowski proof item 1} above.

The preceding argument also shows that $\sigma_L\subseteq \NN(P)_F$ for all linear extensions $L$ of $Q_\Pi$.
By the remarks after the proof of \cite[Prop.\ 3.5]{PRW} we have 
    $$\sigma_\Pi
    =\bigcup_{L} \sigma_L,
    $$
where the union is over all linear extensions of $Q_\Pi$.
It follows that $\NN(P)_F\supseteq \sigma_{\Pi}$.

To show $\NN(P)_F\subseteq \sigma_{\Pi}$, suppose that $h\in \NN(P)_F\setminus \sigma_\Pi$.
By definition of $\sigma_\Pi$ there exist $a,b$ such that $\hat a\succeq \hat b$ and $h(e_a)<h(e_b)$.
We deduce that there is a directed path from $\bar a$ to $\bar b$ in $G_\Pi$. 
This implies that, possibly after reordering indices, there exist $a=a_0,\ldots,a_k=b\in [d]$ such that $a_i\in V(F_i)$ and $a_{i+1}\in \lambda_i\setminus V(F_i)$, for each $1 \leq i \leq k-1$. 
Define the vectors $x=e_{a_0}+\cdots+e_{a_{k-1}}$ and $y=e_{a_1}+\cdots+e_{a_k}$. 
Now $x\in F_0+\cdots+F_{k-1}$, so it should be an $h$-maximal element of $\Delta_{\lambda_0}+\cdots+\Delta_{\lambda_{k-1}}$, but $y\in \Delta_{\lambda_0}+\cdots+\Delta_{\lambda_{k-1}} $ and $h(x)<h(y)$ which is a contradiction.

Lastly, by  \cite[Prop.\ 3.5]{PRW}, the minimal inequalities defining $\sigma_\Pi$ are
	$$
	\begin{cases}
	x_a\le x_b, & b\to a 
	\\
	x_a=x_b, & a\sim_\Pi b
	\end{cases}
	.
	$$
Since the cones of $\NN(P)$ lie in $\R^d/(1,\ldots,1)\R$, it follows that
    $$
    \dim(F)=d-1-\dim(\NN(P)_F)=d-1-\dim(\sigma_\Pi)=d-|[d]/\sim_{\Pi}|.
    $$
This concludes the proof. 
\end{proof}

\begin{remark}
The results in this section remain valid if we replace $e_1,\ldots, e_d$ by an arbitrary set of $d$ linearly independent vectors. In particular, the results remain valid to describe the linearity regions of a max-pooling layer that is pre-composed with an affine map of rank $d$, which could be for instance a linear convolutional layer. 
\end{remark}

\section{One-dimensional input layers} \label{sec:enumeration}

Given positive integers $n,k,s$, recall that  $P_{n,k,s}=\sum_{i=0}^{n-1}\Delta_{\lambda_i}$, where $\lambda_{i}=\{si,{si+1},\ldots,{si+k-1}\}$ for $0\leq i \leq n-1$. 
In this section, we use the characterizations of faces from the previous section to count the number of vertices and facets of the polytopes $P_{n,k,s}$.  In particular, we describe the sequence $(b_n^{(k,s)})_{n\geq 1}$, where $b_n^{(k,s)}$ is the number of vertices of the polytope $P_{n,k,s}$.

\begin{assumption}\label{asumption}
Without loss of generality in this section we assume that $\Lambda=\lambda_0\cup\cdots\cup\lambda_{n-1}$, i.e. $K:=K_1=s(n-1)+k$. 
Note that if $\Lambda\supsetneq\lambda_0\cup\cdots\cup\lambda_{n-1}$, we can replace $\Lambda$ by $\lambda_0\cup\cdots\cup\lambda_{n-1}$ without changing $P_{n,k,s}$. 
\end{assumption}

Our first result considers the case where the stride is at least half of the pooling window size.

\begin{theorem}[Large strides $s\geq \lceil k/2\rceil$]
\label{thm:LargeStrides}
Fix positive integers $k$ and $s$ such that 
$s \in \{ \lceil k/2\rceil,\ldots,k \}$. Then, the generating function of $(b_n^{(k,s)})_{n\geq 1}$ is given by

\begin{align}\label{eq:GenLargeStrides}
1+\sum_{n\geq 1} b^{(k, s)}_{n} x^n
= \frac{1}{1-kx+(k-s)(k-s-1)x^2}.    
\end{align}
The sequence satisfies the recurrence $b_{n+2}^{(k,s)}-kb_{n+1}^{(k,s)}+(k-s)(k-s-1)b_n^{(k,s)}=0$ for $n\geq 2$ with initial values $b_1^{(k,s)}=k$ and $b_2^{(k,s)}=k^2-(k-s)(k-s-1)$.
Moreover, the sequence can be written explicitly as
\begin{align} \label{eq:closed form large strides}
b_n^{(k,s)} = 
\frac{1}{ 2^{n + 1}} 
\left(
w_{-}^n + 
w_{+}^n + 
\frac{2k}{(w_{+} - w_{-})}
\left( w_+^n - w_{-}^n \right) 
\right),
\end{align}
with 
\begin{align*}
w_+ = k + \sqrt{k^2 -4(k-s)(k-s-1)}
&&
w_{-} = k -  \sqrt{k^2 -4(k-s)(k-s-1)}. 
\end{align*}
\end{theorem}
 
Our second result considers the case where the window is proportional to the stride. 
\begin{theorem}[Proportional strides $s\mid k$]
\label{thm:ProporStrides}
Fix a positive integer $s$ and a nonnegative integer $r$ and let $k=s(r+1)$. Then the generating function of $(b_n^{(k,s)})_{n\geq 1}$ is given by
\begin{align} \label{eq:GF proportional strides}
1+\sum_{n\geq 1} b_{n}^{(k,s)} x^n = 
\frac{
1+(rs-s-2)x-(rs-1)x^{2}+sx^{r+1}
}
{
1 - 2(s+1) x + (s +1)^2x^{2} + s x^{r + 1} - s^{2}(r + 1) x^{r + 2} + s(r s -1)x^{r + 3} 
}
\end{align}
In particular, if $s=1$, we obtain
 \begin{align} \label{eq:GF s=1}
1+\sum_{n\geq 1} b_{n}^{(k,1)} x^n = 
\frac{1+(k-4)x-(k-2)x^2 + x^k}
{     1 - 4x+ 4x^{2} + x^{k} - kx^{k+1} + (k-2)x^{k+2}}. 
\end{align}
\end{theorem}
The general tools to prove the above theorems are given in Section~\ref{sec:SetUp} and the proofs are given in Sections~\ref{sec:ProofLargeStrides} and \ref{sec_pf_propor_strides}. In Section~\ref{sec:asymptotics} we discuss the asymptotics. Last, we give a count for the number of facets and a hyperplane description (H-description) in Section~\ref{sec_num_facets}.
\begin{remark}\label{rem_disjoint}
Let us discuss the case $1<k\le s+1$.
First, if $1 < k\leq s$ note that $\lambda_i\cap\lambda_j=\varnothing$ for all $i\neq j$ and thus
    $$
    P_{n,k,s}=\sum_{i=0}^{n-1}\Delta_{\lambda_i}
    =\prod_{i=0}^{n-1}\Delta_{\lambda_i}.
    $$
Then, the count for the number of faces of each dimension is immediate.
For example, $b^{(k,s)}_n = k^n$ if $1 < k\leq s$.
Now, let us discuss the case $k=s+1$. 
In this case $\lambda_i\cap\lambda_j$ has cardinality 1 when $|i-j|=1$ and it is empty otherwise. One can see that the associated graph $G_{\Pi}$ to any list of vertices $\Pi$ from the $\lambda_i $ is necessarily acyclic. It follows that also $b^{(k,s)}_n = k^n$ when $s=k-1$.
 \end{remark}

\begin{remark} \label{rem: case r=0}
The case $r=0$ ($k=s$) in Theorem~\ref{thm:ProporStrides} means that there is no overlap between the windows and so $b_n^{(k,k)}=k^n$. Indeed, formula \eqref{eq:GF proportional strides} gives in this case 
\[
1 + \sum_{n\geq 1}
 b_n^{(s,s)} = \frac{1-2x+x^2}{1-(k+2)x+(2k-1)x^2-kx^3}= \frac{(1-x)^2}{(1-x)^2(1-kx)}=\frac{1}{1-kx},\]
 as expected. 
Similarly, in the case $s=k-1,k$ of \eqref{eq:GenLargeStrides} we also obtain the generating function $1/(1-kx)$. Note that in these cases, the graphs $G_{\Pi}$ from Section~\ref{sec:acyclic.graph} have no edges  and so are acyclic, thus $b_n^{(k,s)}=k^n$. 
\end{remark}

\subsection{Setup and preliminary tools}\label{sec:SetUp}

For positive integers $k,s$, with $k > s$, 
let $D_{k,s}$ be the digraph with vertices $\{0,\ldots,k-1\}$ and all possible arcs except $(s+i',j')$, where $s+i' \neq j'$, $i'=0,...,k-1-s$, and $j'=0,...,k-1-s$. We denote the adjacency matrix of the digraph $D_{k,s}$ by $A=A_{k,s}=(a_{ij})_{i,j=1}^k$ with the convention $a_{ij}=1$ if $(i-1,j-1)$ is an arc in $D_{k,s}$, and $a_{ij}=0$, otherwise. See Figure~\ref{fig: graph k=3s=1} for an example.

\begin{theorem} \label{lemma: gf vertices 1 dim}
Given positive integers $k,s$, let $A=A_{k,s}$ be the adjacency matrix of the digraph $D_{k,s}$.
Then, the generating function of $(b_n^{(k,s)})_{n\geq 1}$ is given by
\[
\sum_{n\geq 0} b_{n+1}^{(k,s)} x^n  = \sum_{n \geq 0} \Bigl( 
\sum_{i,j=1}^k (A^n)_{ij}
\Bigr)x^n = \frac{\sum_{i,j=1}^k (-1)^{i+j} Q_{i,j}(x)}{Q(x)}, 
\]
where $Q(x)=\det(I - x A)$ and $Q_{i,j}(x)$ is the determinant of the submatrix $(I-xA;j,i)$. 
\end{theorem}

\begin{example}
    See Table~\ref{table:vertices one dim} for values of $b_{n}^{(k,s)}$ for several choices of $k$ and $s$ computed using the generating function in Theorem~\ref{lemma: gf vertices 1 dim}. 
\end{example}

\begin{table}[h]
$$\begin{array}{clrrrrrrrrrrrrrrr}
\hline
k &s \,\,\,\,\,\, \backslash n& 1&  2& 3& 4& 5& 6& 7& 8&9&10 \\ \hline
2& 1,2 &  2 & 4 & 8 & 16& 32 & 64 & 128 & 256& 512&1024\\
3&1& 3& 7& 16& 36& 81& 182& 409& 919& 2065& 4640\\
& 2,3& 3& 9& 27& 81& 243& 729& 2187& 6561& 19683& 59049\\
4&1&4& 10& 24& 56& 128& 292& 664& 1508& 3424& 7772 	\\
& 2&4& 14& 48& 164& 560& 1912& 6528& 22288& 76096& 259808 \\
& 3,4 & 4& 16& 64& 256& 1024& 4096& 16384& 65536& 262144& 1048576 \\
5&1&5& 13& 32& 76& 176& 400& 905& 2038& 4577& 10264	\\
& 2&5& 19& 69& 247& 881& 3139& 11181& 39823& 141833& 505147 \\
& 3 & 5&
 23&
 105&
 479&
 2185&
 9967&
 45465&
 207391&
 946025&
 4315343 \\
& 4,5& 5& 25& 125& 625& 3125& 15625& 78125& 390625& 1953125& 9765625\\
 \hline
\end{array}
$$
\caption{Initial terms for the number $b_n^{(k,s)}$ of vertices of the polytopes $P_{n,k,s}$ for some values of $k$ and $s$.}
\label{table:vertices one dim}
\end{table}

In order to prove this result we need some notation and a lemma. By the dimension formula in Proposition~\ref{prop: char vertices minkowski sum of simplices}, the candidate vertices of $P_{n,k,s}$ are of the form $\Pi=(e_{i_0},e_{i_1},\ldots,e_{i_{n-1}})$ where $i_j\in \lambda_{j}$. We view such $\Pi$ as a  word $w(v):=(i_0,i_1,\ldots,i_{n-1})$ corresponding to the point $v=e_{i_0}+e_{i_1}+\cdots + e_{i_{n-1}}$. By abuse of notation we use $G_{w(v)}$ to also denote the graph $G_{\Pi}$.

The graph $G_{w(v)}$ has vertices $\cup_{i=0}^{n-1} \lambda_i$ and arcs $i_j \to b$ for $b\in \lambda_j\setminus i_j$. By Proposition~\ref{prop: char vertices minkowski sum of simplices}, $v$ is a vertex whenever $G_{w(v)}$ is a acyclic. The next result characterizes the cycles of $G_{w(v)}$. 

\begin{lemma}
\label{char: Gv acyclic}
Let $w(v)=(i_0,i_1,\ldots,i_{n-1})$ where $i_j \in \lambda_j$. Then $G_{w(v)}$ has a cycle if and only if $G_{w(v)}$ has a $2$-cycle between consecutive letters of $w(v)$.
\end{lemma}

\begin{proof}
First note that the elements $a$ in $(\cup_{i=0}^{n-1} \lambda_i) \setminus \{i_0,i_1,\ldots,i_{n-1}\}$ are sinks in $G_{w(v)}$ so they are not involved in cycles of the graph. Next, by the overlaps of the consecutive $\lambda_i$, the graph $G_{w(v)}$ satisfies the following properties: for indices $0\leq j<k<\ell \leq n-1$,

\begin{itemize}
    \item if $i_j \to i_{\ell}$ and $i_k\neq i_{\ell}$ then $i_k \to i_{\ell}$. Since if $i_{\ell} \in \lambda_{\ell} \cap (\lambda_{j}\setminus i_j)$ then $i_{\ell} \in \lambda_k\setminus i_k$. See Figure~\ref{fig: prop graph G_wvR}, 
    \item if $i_{j} \gets i_{\ell}$ and $i_k\neq i_{j}$  then $i_j \gets i_{k}$. Since if $i_{j} \in \lambda_{j} \cap (\lambda_{\ell}\setminus i_\ell)$ then $i_{j} \in \lambda_k\setminus i_k$. See Figure~\ref{fig: prop graph G_wvL}. 
\end{itemize}

These properties imply that if $G_{w(v)}$ has a cycle involving $i_j$ then it has a smaller such cycle. Thus if $G_{w(v)}$ has a cycle involving $i_j$, we can assume it is a 2-cycle between two indices, say $i_j$ and $i_{\ell}$ with $j<\ell$. Assume $\ell-j$ is minimal. If $\ell-j=1$ we are done. If $\ell-j>1$, pick an index $k$ in between $j<k<\ell$. There are three possibilities for the value $i_k$:
\[
i_k< \min(i_j,i_{\ell}) \quad \text{ or } \quad i_k \in [\min(i_j,i_{\ell}),\max(i_j,i_{\ell})],\quad \text{ or } \quad i_k> \max(i_j,i_{\ell}) 
\]
Each of the three cases would imply there is a 2-cycle between $i_j \leftrightarrow i_k$ or $i_k\leftrightarrow i_{\ell}$. See Figures~\ref{fig: prop graph G_w cycles I},\ref{fig: prop graph G_w cycles II} for illustrations of the first two cases. This contradicts the minimality of $\ell-j$. Thus the result follows. \end{proof}

\begin{figure}
    \centering
      \centering
  \begin{subfigure}[b]{0.22\textwidth}
    \centering
     {\includegraphics[scale=0.6]{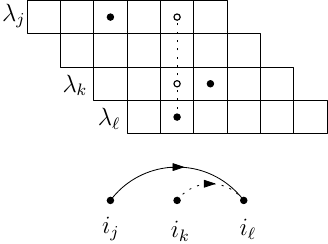}}
    \caption{}
      \label{fig: prop graph G_wvR}
\end{subfigure}
\begin{subfigure}[b]{0.22\textwidth}
  \centering
  {\includegraphics[scale=0.6]{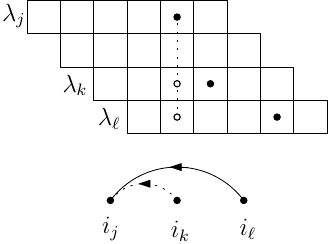}}
  \caption{}
  \label{fig: prop graph G_wvL}
\end{subfigure}
 \begin{subfigure}[b]{0.22\textwidth}
 \includegraphics[scale=0.6]{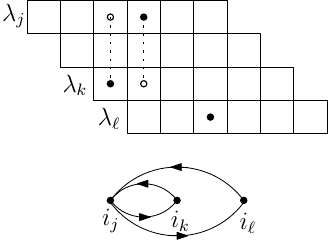}
 \caption{}
  \label{fig: prop graph G_w cycles I}
 \end{subfigure}
 \begin{subfigure}[b]{0.22\textwidth}
 \includegraphics[scale=0.6]{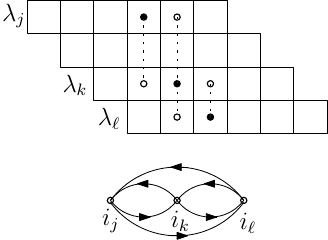}
 \caption{}
  \label{fig: prop graph G_w cycles II}
 \end{subfigure}
    \caption{Properties of the graph $G_{w(v)}$. The black dot in each row represents the corresponding entry in $w(v)$.}
\end{figure}

\begin{proof}[Proof of Theorem~\ref{lemma: gf vertices 1 dim}]
Consider a candidate vertex $w(v)=(i_0,i_1,\ldots,i_{n-1})$ of $P_{n,k,s}$ where $i_j \in \lambda_j$. By Lemma~\ref{char: Gv acyclic},
$G_{w(v)}$ is acyclic if and only if it has no 2-cycles between consecutive letters of $w(v)$. By the definition of $G_{w(v)}$, this graph has edges $i_j \to i_{j+1}$ and $i_{j+1}\to i_j$ precisely when $i_{j}\neq i_{j+1}$ and they are both in the intersection $\lambda_j \cap \lambda_{j+1}$.  
See Figure~\ref{fig: bipartite graph}. Next, we translate this condition to a condition on walks in the digraph $D_{k,s}$.

Given the word $w(v)$ let $w'(v)$ be the {\em standardization}:
\[
(i_0,i_1,\ldots,i_{n-1}) \mapsto (i_0,i_1-s,i_2-2s,\ldots,i_{n-1}-(n-1)s),
\]
so that each letter is in $\{0,1,\ldots,k-1\}$. See Figure~\ref{fig: standardization bipartite graph}. The above acyclicity condition of $G_{w(v)}$ translates to the condition that in $w'(v)=(i'_0,i'_1,\ldots,i'_{n-1})$ no pairs of consecutive elements are of the form $(s+i,j)$, where $s+i \neq j$, $i=0,...,k-1-s$, and $j=0,...,k-1-s$.
 These are exactly the missing arcs in the digraph $D_{k,s}$ (see Figure~\ref{fig: graph k=3s=1}). Thus, the number $b_n^{(k,s)}$ of such words $w'(v)$ (i.e.\ the number of vertices of $P_{n,k,s}$) is equal to the number of walks in the digraph $D_{k,s}$ of length $n-1$. We count such walks via the transfer-matrix method (Theorem~\ref{thm: transfer matrix}) to obtain the desired generating function. 
\end{proof}

\begin{example} \label{ex: case k=3,s=1}
For $k=3$ and $s=1$, the digraph $D_{3,1}$ is exactly the graph from Examples~\ref{ex: k=3 and s=1} and \ref{ex: asymptotics k=3 and s=1}. 
Thus by Theorem~\ref{lemma: gf vertices 1 dim}, the numbers $b_n^{(3,1)}$ of vertices of the polytopes $P_{n,3,1}$ satisfy
\[
\sum_{n\geq 0} b_{n+1}^{(3,1)} x^n = \frac{3+x-x^2}{1-2x-x^2+x^3}.
\]

\begin{figure}
    \centering
      \centering
  \begin{subfigure}[b]{0.2\textwidth}
    \centering
     \raisebox{20pt}{\includegraphics{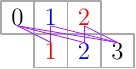}}
    \caption{}
    \label{fig: bipartite graph}
\end{subfigure}
  \begin{subfigure}[b]{0.3\textwidth}
    \centering
     \raisebox{20pt}{\includegraphics{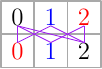}}
    \caption{}
    \label{fig: standardization bipartite graph}
\end{subfigure}
\begin{subfigure}[b]{0.4\textwidth}
  \centering
  \raisebox{-30pt}{\includegraphics{Figures/digraphk3.pdf}}
   $A = \begin{bmatrix}
 1 & 1 & 1 \\
 1 & \textcolor{blue}{0} & 1 \\
 \textcolor{red}{0} & 1 & 1 
 \end{bmatrix}$
  \caption{}
  \label{fig: graph k=3s=1}
\end{subfigure}
    \caption{For the polytope $P_{2,3,1}=\Delta_{\{012\}}+\Delta_{\{123\}}$ we illustrate (A) a bipartite graph whose \textcolor{purple}{edges} $(i,j)$ correspond to the vertices $e_i+e_j$ of $P_{2,3,1}$ and (B) the standardization/relabelling of the graph whose edges are the arcs of (C) the digraph $D_{3,1}$ and its adjacency matrix. The walks in  $D_{3,1}$ encode the vertices of $P_{2,3,1}$.}
\end{figure}
\end{example}

\subsection{Proof Theorem \ref{thm:LargeStrides}}\label{sec:ProofLargeStrides}
 As in the setup of the proof of Theorem~\ref{lemma: gf vertices 1 dim}, for $n\geq 2$, we encode a vertex $v=e_{i_0}+e_{i_1}+\cdots + e_{i_{n+1}}$ of the polytope $P_{n+2,k,s}$ as a word $w(v)=(i_0,i_1,\ldots,i_{n+1})$. 
 Two consecutive simplices $\Delta_{\lambda_i}$ and $\Delta_{\lambda_{i+1}}$ have an overlap of $k-s$ elements. Since $s>\lfloor k/2\rfloor$, non-consecutive simplices $\Delta_{\lambda_i}$ and $\Delta_{\lambda_j}$ with $|i-j|\geq 2$ have no overlap of elements. 
 Thus in the graph $G_{w(v)}$ the edges $i_j\to i_k$ are only between consecutive letters of $w(v)$. That is, the constraints on the words $w(v)$ are on consecutive letters. We form these from the words $(i_0,i_1,\ldots,i_{n})$ corresponding to vertices of $P_{n+1,k,s}$ by adding $k$ possible values $i_{n+1}$. This is an overcount since we have to exclude the invalid cases due to pairs $(i_n,i_{n+1})$ where $i_{n+1} \neq i_{n}$ and $i_{n+1} \in \lambda_{n}\cap \lambda_{n+1}$.  
 and where the rest of the word $(i_0,i_1,\ldots,i_{n-1})$ corresponds to a vertex of $P_{n,k,s}$. There are  $(k-s)(k-s-1)$ such invalid pairs $(i_{n},i_{n+1})$ and $b_n^{(k,s)}$ such words $(i_0,i_1,\ldots,i_{n-1})$. Thus
\begin{align*}
b^{(k,s)}_{n+2} \,=\, k\cdot b^{(k,s)}_{n+1}-(k-s)(k-s-1)b^{(k,s)}_n; 
\end{align*}
that is 
\begin{align*}
b^{(k,s)}_{n+2} - k\cdot b^{(k,s)}_{n+1} + (k-s)(k-s-1)b^{(k,s)}_n
= 0.  
\end{align*}
The initial cases, $b^{(k,s)}_1=k$ and $b^{(k,s)}_2=k^2-(k-s)(k-s-1)$, follow by the same argument.

Next, we find the closed formula for the generating function $G(x):=1+\sum_{n\geq 1} b_n^{(k,s)} x^n=P(x)/Q(x)$. The above recurrence implies that  the denominator $Q(x)$ is the quadratic polynomial in the statement of the theorem.  
By Theorem \ref{thm:GeneratingFunction}, the numerator $P(x)$ is a linear polynomial. 
To find it, we notice that 

\begin{align*}
G(x) &= 1+b_1^{(k,s)}x +b_2^{(k,s)}x^2 +  \cdots
\\
-kx\cdot G(x) &= \phantom{0} \;\ -kx \; \;-kb_1^{(k,s)}x^2  - \cdots
\\
 (k-s)(k-s-1)x^2\cdot G(x) &= 
\phantom{0} \;\;  \phantom{+ 0x} \;\;
 (k-s)(k-s-1)x^2 +\cdots
\end{align*}
which implies the numerator of the generating sequence is equal to
\begin{align*}
(1-kx + (k-s)(k-s-1)x^2)G(x) &= 
1 +  (b_1^{(k,s)} -k)x + (b_2^{(k,s)}-kb_1^{(k,s)}+(k-s)(k-s-1))x^2 + \cdots  
\\
&  \;\; + \left( b_{n+2}^{(k,s)}   -kb_{n+1}^{(k,s)} + (k-s)(k-s-1)b_{n}^{(k,s)} \right)
x^{n+2} + \cdots
\end{align*}
From the initial values for $b^{(k,s)}_1$ and $b^{(k,s)}_2$ we obtain 
$P(x)=(1-kx + (k-s)(k-s-1)x^2)G(x) = 1$, giving the desired formula \eqref{eq:GenLargeStrides}.

Next, solving the recurrence for $b_n^{(k,s)}$ appearing after \eqref{eq:GenLargeStrides} yields
\begin{align*}
b_n^{(k,s)} = 
\frac{1}{ 2^{n + 1}} 
\left(
(k - \sqrt{k_a})^n + 
(k + \sqrt{k_a} )^n + 
\frac{k}{\sqrt{k_a}}
\left( (k + \sqrt{k_a})^n - (k - \sqrt{k_a} )^n \right) 
\right)
\end{align*}
with $k_a = k^2 -4(k-s)(k-s-1)$ (see Theorem~\ref{thm:GeneratingFunction}). The closed formula~\eqref{eq:closed form large strides} is obtained by setting  $w_{+} = k + \sqrt{k_a}$ and $w_{-} = k - \sqrt{k_a}$.

\subsection{Proof of Theorem \ref{thm:ProporStrides}}\label{sec_pf_propor_strides}

When $r=0$ then $s=k$ and the result holds by Remark~\ref{rem: case r=0}. From now on, we assume that $r$ is a positive integer. 

We find a closed formula for the auxiliary generating function $F(x):=\sum_{n\geq 0} b_{n+1}^{(k,s)} x^n$ and then compute $G(x)=1+xF(x)$ to obtain the desired generating function. 

By Theorem~\ref{lemma: gf vertices 1 dim} we have that $F(x)$ is a rational function of the form $P(x)/Q(x)$ where  
\begin{align}
Q(x) &= \det(I-xA_{k,s}), \label{eq:denomGF prop srides} \\
P(x) &= \sum_{i,j=1}^k (-1)^{i+j} \det (I-xA_{k,s};j,i). \label{eq:numeratorGF prop srides}
\end{align}

First, in Lemmas~\ref{lemma:charPoly} and \ref{lemma:denominators}, we find the denominator $Q(x)$. Second, in Lemmas \ref{lem:CoeffNmerator} and \ref{lem:TowardsNums1}, we determine the numerator $P(x)$.
For the latter, we do not use formula \eqref{eq:numeratorGF prop srides} to find $P(x)$, but rather we determine the first terms $b_n^{(k,s)}$ of the series $F(x)$, and use this information to deduce the numerator (see Section~\ref{rem:numerator}).

\begin{lemma}
\label{lemma:charPoly}
Fix positive integers $s$ and $r$ and let $k=s(r+1)$. The characteristic polynomial $C(s(r+1),s)=\det(x\mathbf{I}_k - A_{k,s})$ of the matrix $A_{s(r+1),s}$ satisfies

\begin{align}\label{eq:characteristic}
\frac{C( s(r+1),s)}
{x^{(s-1)(r+1 ) }}   
&=
s^2
\frac{1}{1-x}\left(r-\frac{x-x^{r+1}}{1-x}\right)
-s(1+x)\frac{1-x^r}{1-x}
-(s-x)x^{r}. 
\end{align}

\end{lemma}

\begin{proof}
Let $\mathbf{1}_s$ be $(s \times s)$-matrix with all entries equal to one, and let $\mathbf{I}_s$ be the $(s \times s)$-identity matrix. In our case $k=s(r+1)$, we first write the matrix 
$A_{k,s}-x\mathbf{I}_k$ as the following block matrix:
\begin{align*}
A_{k,s}-x\mathbf{I}_k
=
\begin{pNiceArray}{ccc|c}
\mathbf{1}_s -x\mathbf{I}_s & \cdots & \mathbf{1}_s &\mathbf{1}_s   
\\ \hline 
&&     & \mathbf{1}_s
    \\
 & \mathbf{I}_{k-s} - xN &    & \vdots
    \\
&&     & \mathbf{1}_s  -x\mathbf{I}_s     
\end{pNiceArray}
&&
N:=
\begin{pNiceArray}{cc|cc }
& \mathbf{0} &  & \mathbf{I}_{k-2s} \\ 
& & & \\ \hline
& \mathbf{0}_{s} &  & \mathbf{0} \\ 
\end{pNiceArray}.
\end{align*}
Next, we perform elementary row operations so that the block matrix $\mathbf{I}_{k-s} - xN$ becomes the identity. This will yield  a  new block matrix of the form 
\begin{align*}
\begin{pNiceArray}{ccc|c}
\mathbf{1}_s -x\mathbf{I}_s & \mathbf{1}_s &\cdots  &\mathbf{1}_s   
\\ \hline 
&&     & \mathbf{K}_m
    \\
 & \mathbf{I}_{k-s}   &    & \vdots
    \\
&&     & \mathbf{K}_0      
\end{pNiceArray}
&& 
\mathbf{K}_i := 
\begin{pmatrix}
b_i & g_i    & \cdots & g_i   \\
g_i & b_i    & \cdots & g_i   \\
\vdots    &  &  \ddots      & \vdots\\
g_i & g_i    & \cdots & b_i
\end{pmatrix}
&& 
m:=\frac{k}{s}-2, 
\end{align*}
where $g_i= \sum_{k=0}^i x^k=(1-x^{i+1})/(1-x)$ and $b_i = g_i - x^{i+1}$, and $\mathbf{K}_i$ is a $(s \times s)$-matrix. Next, we use more row operations to simplify the blocks at the top of our matrix from
\(\begin{pmatrix}
 \mathbf{1}_s - x\mathbf{I}_s & \mathbf{1}_s &\cdots & \mathbf{1}_s    
 \end{pmatrix}\)   
to
\(
  \begin{pmatrix}
 \mathbf{H}_s & \mathbf{D}_s &\cdots&  \mathbf{D}_s    
 \end{pmatrix}   
\)
where $\mathbf{H}_s$ and $\mathbf{D}_s$ are 
$(s \times s)$-matrices:
\begin{align*}
\mathbf{H}_s:=
\begin{pNiceArray}{ccccc}
-x & 0 & \cdots& 0 & x   \\
0 & -x & \cdots& 0 & x   \\
\vdots   & &\ddots & &  \vdots \\
0 & 0 & \cdots& -x & x   \\
1 & 1 & \cdots& 1 & 1-x  \\
\end{pNiceArray}
&& 
\mathbf{D}_s:=
\begin{pNiceArray}{ccccc}
0 & 0 & \cdots& 0 & 0   \\
0 & 0 & \cdots& 0 & 0   \\
\vdots   & & & &  \vdots \\
0 & 0 & \cdots& 0 & 0   \\
1 & 1 & \cdots& 1 & 1  \\
\end{pNiceArray}
\end{align*}
The resulting new block matrix $B_{k,s}(x)$ is 
\begin{align*}
B_{k,s}(x)=\begin{pNiceArray}{c cc |c}
\mathbf{H}_s & \mathbf{D}_s & \cdots & \mathbf{D}_s   
\\ \hline 
 & & & \\
&\mathbf{I}_{k-s}&         & \mathbf{K} \\
& & &
\end{pNiceArray}, \qquad 
\mathbf{K} = 
\begin{pNiceArray}{c}
\mathbf{K}_m \\
\vdots \\
\mathbf{K}_0
\end{pNiceArray}. 
\end{align*}
Since the matrix $B_{k,s}(x)$ was obtained from $A_{k,s}-x\mathbf{I}_k$ with row operations then both matrices have the same determinant. Thus the characteristic polynomial $C(k,s):=\det(x \mathbf{I}_k-A_{k,s})=(-1)^k\det B_{k,s}(x)$. Next we flip horizontally the block matrix $B_{k,s}(x)$ so that we can use the formula for the determinant of the {\em Schur complements of block matrices} to obtain
\begin{align*}
\det B_{k,s}(x)= (-1)^{s(k-s)}\det
\begin{pNiceArray}{c|c}
  \mathbf{I}_{k-s}       & \mathbf{K}
     \\ \hline
\mathbf{C}_s &  \mathbf{D}_s   
\end{pNiceArray}
=
(-1)^{s(k-s)}\det \left( 
\mathbf{D}_s - \mathbf{C}_s \cdot \mathbf{K} \right) 
&& 
\mathbf{C}_s :=
\begin{pNiceArray}{c:ccc}
\mathbf{H}_s & \mathbf{D}_s & \cdots & \mathbf{D}_s
\end{pNiceArray} . 
\end{align*}
Before evaluating the determinant on the RHS above, we calculate each matrix in the sum separately. That is,
\begin{align*}
\mathbf{D}_s - \mathbf{C}_s \cdot \mathbf{K}
=
\mathbf{D}_s - 
\left( 
\mathbf{H}_s \cdot \mathbf{K}_m  + 
\sum_{i=0}^{m-1} \mathbf{D}_s \cdot \mathbf{K}_i
\right),  
\end{align*}  
where
\begin{align}
\mathbf{H}_s \cdot \mathbf{K}_m = \begin{pNiceArray}{ccccc}
r_m & 0& \cdots & 0 & -r_m \\
0 & r_m & \cdots & 0 & -r_m \\
\vdots &  &\ddots  & & \vdots\\
0 &0  & &r_m & -r_m\\
a_m & a_m&\cdots & a_m & c_m
\end{pNiceArray}
&&
\mathbf{D}_s \cdot \mathbf{K}_i =  \begin{pNiceArray}{ccc}
0 & \cdots & 0 \\
\vdots &  & \vdots\\
0 & \cdots & 0 \\
(s-1)g_i+b_i & \cdots & (s-1)g_i+b_i
\end{pNiceArray},
\end{align}
for $r_m = -xb_m+xg_m=x^{m+2}$, $a_m = (s-1-x)g_m+b_m$, and $c_m=(s-1)g_m+(1-x)b_m$. This gives the following expression for the sum of the matrices
\begin{align*}
\mathbf{D}_s - \mathbf{C}_s \cdot \mathbf{K} \,=\, \begin{pNiceArray}{ccccc}
-r_m & 0& \cdots & 0 & r_m \\
0 & -r_m & \cdots & 0 & r_m \\
\vdots &  &\ddots  & & \vdots\\
0 &0  & &-r_m & r_m\\
h_m - a_m & h_m-a_m &\cdots & h_m-a_m & h_m-c_m
\end{pNiceArray}, 
\end{align*}
where $h_m=1-\sum_{i=0}^{m-1} ((s-1)g_i+b_i)$. In the matrix $\mathbf{D}_s - \mathbf{C}_s \cdot \mathbf{K}$, we add to the last column a copy of each other column to make the matrix lower triangular with diagonal $-r_m,\ldots,-r_m, s\cdot h_m-(s-1)a_m-c_m$. We are now ready to calculate the determinant
\begin{align*}
\det \left( 
\mathbf{D}_s - \mathbf{C}_s \cdot \mathbf{K} \right) &= (-1)^{s-1}r_m^{s-1}(s\cdot h_m - (s-1)a_m-c_m)\\
&= (-1)^{s-1}x^{(m+2)(s-1)}\left( s-s\sum_{i=0}^{m} (sg_i -x^{i+1}) + x(sg_m - x^{m+1})\right)\\
&= (-1)^{s}x^{(m+2)(s-1)}\left(s^2 \sum_{i=0}^m g_i - s\sum_{j=0}^m x^j - sx^{m+1} - sxg_m + x^{m+2}\right).
\end{align*}
We use  $\sum_{j=0}^m x_j = g_m$ and $\sum_{i=0}^m g_i = \dfrac{1}{1-x}\left(m+1-(x-x^{m+2})/(1-x)\right)$ to obtain that 
\begin{align*}
   (-1)^s\det \left( 
\mathbf{D}_s - \mathbf{C}_s \cdot \mathbf{K} \right) =  x^{(m+2)(s-1)}\left( s^2
\frac{1}{1-x}\left(r-\frac{x-x^{r+1}}{1-x}\right)
-s(1+x)\frac{1-x^r}{1-x}
-(s-x)x^{r}\right). 
\end{align*}
Putting everything together, we have that $C(k,s)=(-1)^{k+s(k-s)+s}(-1)^s\det \left( 
\mathbf{D}_s - \mathbf{C}_s \cdot \mathbf{K} \right)$. Lastly, in our case where $s$ divides $k$, we have that $k+s(k-s)+s=k(s+1)-s(s-1)$ is even and the result follows. 
\end{proof}

\begin{lemma}\label{lemma:denominators}
Fix positive integers $s$ and $r$ and let $k=s(r+1)$. The denominator $Q(x)$ of the generating function $F(x)$  satisfies the following equation 
\begin{align*}
(x-1)^2Q(x) &= 
s(r s -1)x^{r + 3}  - s^{2}(r + 1) x^{r + 2} + s x^{r + 1}
+ (s +1)^2x^{2}   - 2(s+1) x   + 1 . 
\end{align*}
\end{lemma}

\begin{proof}
Recall that $C(k,s)$ is the characteristic polynomial of the matrix $A_{k,s}$. Since $Q(x)=\det(I-xA_{k,s})$, by Remark~\ref{rem: rel to charpoly and deg of denom}  we have that 
$C(k,s)=x^k Q(1/x)$. The desired result follows by combining this relation with Lemma \ref{lemma:charPoly} and standard manipulations.
\end{proof}

We now turn to finding the numerator $P(x)$ of the generating function $F(x)$. In fact, we find $(x-1)^2P(x)$, which we will see has a more compact form than $P(x)$.

\begin{lemma}
\label{lem:CoeffNmerator}
Fix positive integers $s$ and $r$ and let $k=s(r+1)$. Let $P(x)$ be the numerator of the generating function $F(x)$, then coefficients of the polynomial $(x-1)^2P(x) = \sum_{m=0}^{r+2}c_{m+1}x^m$ satisfy the following relations:
\begin{align*}
c_1 & = k, \\  
c_2 & = 
(2r+1)s^2 + rs -2(s+1)k,
\\
c_{m}& =
b_{m}^{(k,s)} -2(s+1)b_{m-1}^{(k,s)} + (s+1)^2b_{m-2}^{(k,s)}
\qquad \text{ with $3 \leq  m \leq r+1$},
\\ 
c_{r+2} &=
sk + b_{r+2}^{(k,s)} -2(s+1)b_{r+1}^{(k,s)} +  (s+1)^2b_{r}^{(k,s)},
\\ 
c_{r+3} &= b_{r+3}^{(k,s)} - 2(s+1)b_{r+2}^{(k,s)} +  (s+1)^2b_{r+1}^{(k,s)}.
+rs^2(1-rs)
\end{align*}

\end{lemma}
\begin{proof}
We write the generating function as 
$$
F(x)= \frac{P(x)}{Q(x)} = 
\frac{(x-1)^2P(x)}{(x-1)^2Q(x)}
$$
with $P(x)$ and $Q(x)$ as in  Theorem  \ref{thm:GeneratingFunction}. In particular,  $\deg \left( (x-1)^2P(x) \right) \leq \deg \left( (x-1)^2Q(x) \right)-1 = r+2$.  
By Lemma \ref{lemma:denominators}, we have 
\begin{align*}
    (x-1)^2Q(x)F(x) 
&=
\left( s(rs -1)x^{r + 3}  - s^{2}(r + 1) x^{r + 2} + s x^{r + 1}
+ (s +1)^2x^{2}   - 2(s+1) x   + 1 \right)
\left( \sum_{n \geq 0}b_{n+1}^{(k,s)}x^n \right) . 
\end{align*}
Multiplying the expression $F(x)=\sum_{n \geq 0}b_{n+1}^{(k,s)}x^n$ by appropriate monomials and focusing on the coefficients of the monomials $x^{m}$ with $m \leq r+2$, we arrive to the following equalities, where $R_1(x)$ and $R_2(x)$ and represent sums of monomials of degree larger than or equal to $r+3$: 
\begin{align*}
-s^2(r+1)x^{r+2}F(x)
&=
-s^2(r+1)b_1^{(k,s)}x^{r+2} + R_1(x)
\\
sx^{r+1}F(x)
&=
sb_1^{(k,s)}x^{r+1} +  sb_2^{(k,s)}x^{r+2} + R_2(x)
\\
(s+1)^2x^2F(x)
&=
\sum_{n \geq 2}
(s+1)^2b_{n-1}^{(k,s)}x^n
\\
-2(s+1)xF(x) &=
-2(s+1)b_1^{(k,s)}x - 
\sum_{n \geq 2} 2(s+1)b_n^{(k,s)}x^n
\\
F(x) &= b_1^{(k,s)} + b_2^{(k,s)}x + \sum_{n \geq 2}b_{n+1}^{(k,s)}x^n. 
\end{align*}
By adding them up we obtain the following, where $R_3(x)$ represents a sum of monomials of degree larger than or equal to $r+3$:
\begin{align*}
    (x-1)^2Q(x)F(x) 
    &=
    b_1^{(k,s)} + 
    \left( b_2^{(k,s)} - 2(s+1)b_1^{(k,s)} \right)x + 
    \sum_{m=2}^{r}
    \left(
    b_{m+1}^{(k,s)} -2(s+1)b_{m}^{(k,s)} + (s+1)^2b_{m-1}^{(k,s)}
    \right)x^{m}
    \\ & \quad +
    \left( 
    b_{r+2}^{(k,s)} -2(s+1)b_{r+1}^{(k,s)} +  (s+1)^2b_{r}^{(k,s)}
    +sb_1^{(k,s)} 
    \right)x^{r+1}
    \\ &\quad +
    \left(
    b_{r+3}^{(k,s)} -2(s+1)b_{r+2}^{(k,s)} +  (s+1)^2b_{r+1}^{(k,s)}
    +sb_2^{(k,s)} -s^2(r+1)b_1^{(k,s)}    
    \right)x^{r+2}
    + R_3(x).
\end{align*}
The equalities 
$(x-1)^2P(x)  = c_1 +  c_2x + \cdots + c_{r+3}x^{r+2}$ 
and $(x-1)^2P(x) = (x-1)^2Q(x)F(x)$ now yield 
\begin{align*}
c_1 &  = b_1^{(k, s)} 
\\
c_2 &  = b_2^{(k, s)} - 2(s+1)b_1^{(k, s)} 
\\
c_{m+1}& =
b_{m+1}^{(k, s)} -2(s+1)b_{m}^{(k, s)} + (s+1)^2b_{m-1}^{(k, s)}
\qquad \qquad \text{ with $2 \leq  m \leq r$}
\\ 
c_{r+2} &=
b_{r+2}^{(k, s)} -2(s+1)b_{r+1}^{(k, s)} +  (s+1)^2b_{r}^{(k, s)} + sb_1^{(k,s)} 
\\ 
c_{r+3} &=
b_{r+3}^{(k,s)} -2(s+1)b_{r+2}^{(k,s)} +  (s+1)^2b_{r+1}^{(k,s)}
    +sb_2^{(k,s)} -s^2(r+1)b_1^{(k,s)} . 
\end{align*}
By a direct calculation with the matrices $A_{k,s}^0=I=I_k$ and $A_{k,s}$, we find that 
\begin{align*}
b_1^{(k,s)} = 
\sum_{i,j} I_{ij} = k,
&&
b_2^{(k,s)} = 
\sum_{i,j} (A_{k,s})_{ij}
=
(2r+1)s^2 + rs 
. 
\end{align*}
Therefore, our result follows by a direct substitution. 
\end{proof}

We will see in Lemma~\ref{lem:TowardsNums1} that most of the $c_i$s in the previous result vanish. Next, we give two general facts that we will need for the proof of Lemmas~\ref{lem:TowardsNums1}. For simplicity, we write $A:=A_{k,s}$. 

\begin{lemma}\label{lemma:prelim formula b}
Fix positive integers $s$ and $r$ and let $k=s(r+1)$. For $m\geq 1$ we have that
\begin{equation}
b_{m+2}^{(k,s)} \,=\, 
  (s+1)b_{m+1}^{(k,s)} + (k-s-1) \sum_{j=1}^k 
\sum_{v=k-s+1}^k(A^{m})_{vj}. 
    \label{eq:bm2}
\end{equation}
\end{lemma}

\begin{proof}
Given $m\geq 1$, a  direct calculation shows that

\begin{align}\label{eq:matrixA}
(A^{m+1})_{ij}
=
\begin{cases}
\sum_{v=1}^k (A^m)_{vj}
& 
\text{ if $i=1,\ldots, s$}
\\
(A^m)_{(i-s)j} + \sum_{v=k-s+1}^k (A^m)_{vj}
&
\text{ otherwise}. 
\end{cases}
\end{align}
Therefore, using \eqref{eq:matrixA} repeatedly, we obtain that 
\begin{align*}
b_{m+2}^{(k,s)} = \sum_{i,j}(A^{m+1})_{ij}
&=
\sum_{j=1}^k 
\left( \sum_{i=1}^s (A^{m+1})_{ij} 
+
\sum_{i=s+1}^k (A^{m+1})_{ij} \right)
\\
&=
\sum_{j=1}^k 
\left( \sum_{i=1}^s \sum_{v=1}^k (A^m)_{vj}
+
\sum_{i=s+1}^k 
\left(
(A^m)_{(i-s)j} + \sum_{v=k-s+1}^k (A^m)_{vj}
\right)
\right)\\
\intertext{Using $b_{m+1}^{(k,s)}=\sum_{v,j} (A^m)_{vj}$ and simplifying the summands that are independent of $i$, and reindexing the sum with terms $(A^m)_{(i-s)j}$ the above equation becomes}
&= s\cdot b_{m+1}^{(k,s)} + \sum_{j=1}^k \left(\sum_{i=1}^{k-s} (A^m)_{ij} + (k-s) 
\sum_{j=1}^k
\sum_{v=k-s+1}^k (A^m)_{vj}\right),\\
\intertext{Lastly, by splitting off one copy of $b_{m+1}$ from the sum above gives}
& = 
(s+1)b_{m+1}^{(k,s)}
+ (k-s-1)
\sum_{j=1}^k\sum_{v=k-s+1}^k (A^m)_{vj}
\end{align*}
which implies
\begin{multline}
b_{m+2}^{(k,s)} -2(s+1)b_{m+1}^{(k,s)} +(s+1)^2b_{m}^{(k,s)} 
 = \sum_{i,j} (A\cdot A^{m})_{ij} -2(s+1)b_{m+1}^{(k,s)} +(s+1)^2b_{m}^{(k,s)}\\
\begin{aligned}
&=
(s+1)b_{m+1}^{(k,s)} + 
(k-s-1)\sum_{j=1}^k\sum_{v=k-s+1}^k (A^m)_{vj} -2(s+1)b_{m+1}^{(k,s)} +(s+1)^2b_{m}^{(k,s)} 
 \\ &=
(s+1)^2b_{m}^{(k,s)} - (s+1)b_{m+1}^{(k,s)} + 
(k-s-1)\sum_{j=1}^k\sum_{v=k-s+1}^k (A^m)_{vj}.
\end{aligned}
\end{multline}
Cancelling the term of $b_m^{(k,s)}$ and collecting the terms of $b_{m+1}^{(k,s)}$ gives the desired result.
\end{proof}

\begin{lemma}
Fix positive integers $s$ and $r$ and let $k=s(r+1)$. For all $p=1,\ldots,r$, the matrix $A^p$ where $A:=A_{k,s}$ decomposes  as
\begin{align}\label{eq:decom}
A^p
=
\begin{pNiceArray}{cc}
B^{\top} & C \\ 
I_{k-sp} & B
\end{pNiceArray}    
&&
\text{ where }
&&
B=B_{(k-sp) \times sp}
=
\begin{pNiceArray}{cccc}
(s+1)^0 & (s+1)^1 & \cdots & (s+1)^{p-1} 
\end{pNiceArray} \otimes {\bf 1}_{(k-sp)\times s}
\end{align}
where $C$ is an $sp \times sp$ matrix that is  symmetric with respect to the anti-diagonal.
\end{lemma}

\begin{proof}
We show Equation~\eqref{eq:decom} by induction on the power $p$ of the matrix $A$. The base case $p=1$ follows from the definition of $A$. For the case $p+1$, we directly calculate $A^{p+1} =  A \cdot A^p$. Since both matrices $A$ and $A^p$ have a block structure, we separate our analysis into four cases for calculating $(A^{p+1})_{ij}$: 
(i) $1 \leq i \leq s$ and $1 \leq j \leq k-sp$; 
(ii) $1 \leq i \leq s$ and $k-sp +1 \leq j \leq k$; (iii) $s+1 \leq i \leq k$ and $1\leq j \leq k-sp$; and (iv) $s+1 \leq i \leq k$ and 
$k-sp+1 \leq j \leq k$. 

In the first case, where $1 \leq i \leq s$ and $1 \leq j \leq k-sp$, we have
\begin{align*}
(A^{p+1})_{ij} & = \sum_{l=1}^k A_{il}\cdot A^{p}_{lj}
 =  \sum_{l=1}^k 1\cdot A^{p}_{lj}
= 1 + \sum_{l=1}^{sp} (B^{\top})_{lj} 
= 1+ s \sum_{l=1}^{p-1} (s+1)^l
=  (s+1)^p
\end{align*}
For the second case, where $s+1 \leq i \leq k$ and $1 \leq j \leq k-sp$, we observe that the $i$-th row of $A$ is the concatenation of the standard basis vector $e_{i-s}$  of $\mathbb{R}^{k-s}$ with a vector of all ones.  On the other hand, the $j$-th column of $A^p$ is the concatenation of the $j$-th column of $(B^{\top})$ and the $j$-th column of the identity matrix $I_{k-sp}$. Therefore we derive the following expression:
\begin{align*}
(A^{p+1})_{ij} 
=
\begin{pmatrix}
\delta_{1(i-s)}, \cdots,  \delta_{(k-s)(i-s)} , 1, \cdots, 1    
\end{pmatrix}
\cdot 
\begin{pmatrix}
(B^{\top})_{1j} 
\\ \vdots \\
(B^{\top})_{(sp)j}
\\
\delta_{1j}
\\ \vdots \\
\delta_{(k-sp)j}
\end{pmatrix},
\end{align*}
where $\delta_{ij}$ denotes the Kronecker delta.
We observe that $sp \leq (k-s)$ because by hypothesis we have $s(p+1) \leq k=s(r+1)$. 
We find that the above dot product becomes:
\begin{align}\label{eq:Ap1}
    (A^{p+1})_{ij}  
    =  
    \sum_{l=1}^{sp} \delta_{l(i-s)}(B^{\top})_{lj}
    + 
    \sum_{l=sp+1}^{l=k-s} 
    \delta_{l(i-s)}\delta_{(l-sp)j}
    +
    \sum_{l=1}^s
    1 \cdot \delta_{(k-sp-s+l)j}.
\end{align}
We calculate the last term in Equation \eqref{eq:Ap1}. Since, we are in the induction step, we have $p < r$. This implies 
$k-sp-s= s(r+1)-sp-s =s(r-p) \geq s$. Given this bound, we conclude that  
$\sum_{l=1}^s 1 \cdot \delta_{(k-sp-s+l)j} = 0$.

The first two terms of Equation \eqref{eq:Ap1} are a bit more subtle. We observe that for a fixed value of $i$ that only one of them can be non-zero due to the presence of the Kronecker deltas. The possible sub-cases are:
\begin{itemize}
    \item When $s+1 \leq i \leq s(p+1)$, we find that $(i-s) \leq sp$. Therefore, we have:
    \begin{align*}
    \sum_{l=1}^{sp} \delta_{l(i-s)}(B^{\top})_{lj}
    =
    (B^{\top})_{(i-s)j}, 
    &&
    \sum_{l=sp+1}^{l=k-s} 
    \delta_{l(i-s)}\delta_{(l-sp)j}
    =0
    && 
    \end{align*}
    \item When
    $s(p+1) +1 \leq i \leq k$ and 
    $1 \leq j \leq k-s(p+1)$,
    we find that 
    $sp  \leq (i - s) \leq (k-s)$.
    Therefore, we have: 
    \begin{align*}
    \sum_{l=1}^{sp} \delta_{l(i-s)}(B^{\top})_{lj}
    =
    0
    &&
    \sum_{l=sp+1}^{l=k-s} 
    \delta_{l(i-s)}\delta_{(l-sp)j}
    =
    \delta_{(i-s(p+1))j}
    && 
    \end{align*}
\end{itemize}
Therefore, in the second case we have the following possible values that $(A^{p+1})_{ij}$ can take:
   \begin{align*}
    (A^{p+1})_{ij}  
    =
    \begin{cases}
    (B^{\top})_{(i-s)j} 
    &\;\text{ if }\;  (s+1) \leq i \leq s(p+1)
    \; \text { and }\;
    \\
    \delta_{(i-s(p+1))j}
    &\;\text{ if }\;
    s(p+1) + 1\leq i \leq k
    \; \text { and }\;
    1 \leq j \leq k-s(p+1). 
    \end{cases}.
    \end{align*}
The other two cases are very similar, so they are left to the reader. 
\end{proof}

Next, we show that many of the coefficients $c_i$ from Lemma 
\ref{lem:CoeffNmerator} are equal to zero.

\begin{lemma}\label{lem:TowardsNums1}
Fix positive integers $s$ and $r$ and let  $k=s(r+1)$.  For  $2 \leq m \leq r+1$, the number of vertices 
$
b_{m+1}^{(k,s)} := \sum_{i,j} (A_{k,s}^m)_{ij}
$ satisfies the recurrence
\begin{equation} \label{eq:recurrence first vals b}
b_{m+1}^{(k,s)} = 2(s+1)b_{m}^{(k,s)} -(s+1)^2b_{m-1}^{(k,s)}, 
\end{equation}
and has the closed form
\begin{align}\label{eq:closed form for first bs}
b_{m+1}^{(k,s)} = (s+1)^{m-1}( (m+1)s(k-s-1)+k+s(s+1)).     
\end{align}
\end{lemma}
\begin{proof}
By calculating directly using 
 $b_{l+1}^{(k,s)} = \sum_{i,j} (A^{l})_{ij}$, we obtain that
\begin{align}\label{eq: bdry conditions}
b_1^{(k,s)} &= k,  \notag \\  
b_2^{(k,s)} &= 2s(k-s-1)+k+s(s+1),\\ 
b_3^{(k,s)} &= (s+1)(3s(k-s-1)+k+s(s+1)).  
\notag
\end{align}
These initial values satisfy
$$
b_{3}^{(k,s)} -2(s+1)b_2^{(k,s)} +(s+1)^2b_{1}^{(k,s)}
= 
0.
$$
Therefore, \eqref{eq:recurrence first vals b} holds for $m=2$ and \eqref{eq:closed form for first bs} holds for $m=1,2$. 
Notice that for any $3 \leq m \leq r+1$, the case $m$ of \eqref{eq:closed form for first bs} follows from the two previous cases $m-1$ and $m-2$ of \eqref{eq:closed form for first bs} and the case $m$ of the closed formula \eqref{eq:closed form for first bs}.
Then, to complete the proof it is enough to show \eqref{eq:recurrence first vals b}  by induction on $2 \leq m\leq r+1$. The case $m=2$ holds, so we proceed with the inductive step. By inductive hypotheses, we suppose that for some $m \geq 2$, the number $b_{j+1}^{(k,s)}$ of vertices satisfies \eqref{eq:recurrence first vals b} for all $2 \leq j\leq m<r+1$. And therefore $b_{j+1}^ {(k,s)}$ satisfies the closed formula \eqref{eq:closed form for first bs} for all $2 \leq j \leq m<r+1$. 
We now verify the case $j=m+1$. From the above closed formulas for $b_m^{(k,s)}$ and $b_{m+1}^{(k,s)}$, we obtain
\begin{align*}
(s+1)^2b_{m}^{(k,s)} &= 
(s+1)^{m}(ms(k-s-1)+k+s(s+1)) 
\\
- (s+1)b_{m+1}^{(k,s)} &= 
- (s+1)^m((m+1)s(k-s-1)+k+s(s+1)), 
\end{align*}
which implies
\begin{align}\label{eq:sumb_m}
(s+1)^2b_{m}^{(k,s)} - (s+1)b_{m+1}^{(k,s)}
=
s(-k + s + 1)(s + 1)^m . 
\end{align}
By the description of the matrices in Equation~\eqref{eq:decom}, in particular $B$, we have that
\begin{align}\label{eq:doubleSum}
\sum_{v=k-s+1}^k \sum_{j=1}^k  (A^m)_{vj}  
&= \sum_{v=k-s+1}^k(1 + s\sum_{i=0}^{m-1}(s+1)^i )
\\
&= \sum_{v=k-s+1}^k(s+1)^m = s(s+1)^m.
\notag
\end{align}
We use \eqref{eq:sumb_m} and \eqref{eq:doubleSum}, and \eqref{eq:bm2} (which is valid for all $m$) to obtain
\begin{align}\label{eq:coeffB}
b_{m+2}^{(k,s)} -2(s+1)b_{m+1}^{(k,s)} +(s+1)^2b_{m}^{(k,s)}
&=
(s+1)^2b_m^{(k,s)} - (s+1)b_{m+1}^{(k,s)} +  (k-s-1) \sum_{j=1}^k \sum_{v=k-s+1}^k(A^{m})_{vj} . 
\notag
\\
& = 
s(-k + s + 1)(s + 1)^m
+
 (k-s-1) \sum_{j=1}^k \sum_{v=k-s+1}^k(A^{m})_{vj} 
 \\
 &= 
 s(-k + s + 1)(s + 1)^m + (k-s-1)s(s+1)^m
 \notag
 \\
 &= 0
 \notag
\end{align}
This verifies the recurrence \eqref{eq:recurrence first vals b} for $j=m+1$. 
\end{proof}

The previous result gave the closed formula \eqref{eq:closed form for first bs} for $b_{m}^{(k,s)}$ for $m=3,\ldots,r+2$. A direct computation shows \eqref{eq:closed form for first bs} also holds for $m=1,2$. The next result gives a closed formula for the next case $m=r+3$.

\begin{lemma}\label{lemma: br+3}
Fix positive integers $s$ and $r$ and let $k=s(r+1)$. We have that \[
b_{r+3}^{(k,s)}
=
(s+1)^{r+1}\bigl((r+3)s(sr-1)+s(r+1)+s(s+1)\bigr)+s(rs-1)^2.
\]
\end{lemma}
\begin{proof}
By Lemma \ref{lemma:prelim formula b} 
with $m=r+1$,  we have that 
\begin{align*}
b_{r+3}^{(k,s)}
& = (s+1)b_{r+2}^{(k,s)} + 
  (k - s -1)\sum_{j=1}^k \sum_{v= k-s+1}^k (A^{r+1})_{vj}
\end{align*}
We use the closed form  in Equation~\eqref{eq:closed form for first bs} when $m=r+1$ for $b_{r+2}^{(k,s)}$ to obtain that 
\begin{align} \label{eq: intermediate expression for bs}
b_{r+3}^{(k,s)} & =
(s+1)^{r+1}\left( 
 s(r+2)(k-s-1) + k+ s(s+1)
\right)
+ 
(k - s -1)\sum_{j=1}^k \sum_{v= k-s+1}^k (A^{r+1})_{vj}
\end{align}
Therefore, our problem is reduced to calculating certain terms of the matrix $A^{r+1}$. 
For that purpose, 
we recall that 
we have a partial description 
of $A^r$ given by
\begin{align*}
A^{r}
=
\begin{pNiceArray}{cc}
B^{\top} & C \\ 
I_{s} & B
\end{pNiceArray}    
&&
\text{ where }
&&
B_{s \times sr}
=
\begin{pNiceArray}{cccc}
(s+1)^0 & (s+1)^1 & \cdots & (s+1)^{r} 
\end{pNiceArray} \otimes {\bf 1}_{sr \times s}
\end{align*}
The matrices $A$, $A^r$ and $A^{r+1}$ are symmetric with respect to the anti-diagonal and therefore $\sum_{v=k-s+1}^k(A^{r+1})_{vj} = \sum_{v=1}^s (A^{r+1})_{jv}$. We prefer to calculate the right hand side 
via $A\cdot A^{r}$, by only using the matrices $B^{\top}$ and $I_s$. Using this symmetry   \eqref{eq:closed form for first bs} becomes
\begin{equation}
\label{eq:closed form for first bs antidiagonal}
    b_{r+3}^{(k,s)}
 = (s+1)^{r+1}\left( 
 s(r+2)(k-s-1) + k+ s(s+1)
\right)
+ 
(k - s -1)\sum_{j=1}^k \sum_{v=1}^s (A^{r+1})_{jv}
\end{equation}
  We use Equation~\eqref{eq:matrixA} to calculate the required entries $(A^{r+1})_{jv}$ for all $j$.  For $1 \leq j \leq s$, we have that  
\begin{align} \label{eq: sum of B transpose}
(AA^r)_{j1} &= \sum_{v=1}^k 1 \cdot (A^r)_{v1}   
=
1+ \sum_{i=1}^{k-s} (B^{\top})_{i1} 
= 1+ s\sum_{i=0}^{r-1} (s+1)^i = (s+1)^r
\end{align}
The fact that the first $s$ columns of $B^{\top}$ are the same implies that 
\begin{equation}
    \label{A r+1 jv entry first s}
(A^{r+1})_{jv} = (A^{r+1})_{j1} =   (s+1)^r,
\;\; \text{ for }
1 \leq j \leq s \;
\text{ and }
1 \leq v \leq s.
\end{equation}
We recall that $A$ has an identify matrix of size $(k-s)$ which is precisely the number of rows of $B^{\top}$. 
For $(s+1) \leq j \leq k$, the $j$-th row of $A$ is of the form 
\[
(e_{j-s} , 1 \ldots 1)
\]
A direct calculation shows that
\begin{align*}
(AA^r)_{j1} &= 
B^{\top}_{(j-s)1} + 1.
\end{align*}
The fact that the first $s$ columns of $B^{\top}$ are the same implies that 
\begin{equation}
\label{A r+1 jv entry last k-s}
(A^{r+1})_{jv} = (A^{r+1})_{j1} =  B^{\top}_{(j-s)1} + 1,
\;\; \text{ for }
(s+1) \leq j \leq k \;
\text{ and }
1 \leq v \leq s.
\end{equation}
We use the above expressions \eqref{A r+1 jv entry first s},\eqref{A r+1 jv entry last k-s}, and \eqref{eq: sum of B transpose} to calculate $\sum_{j=1}^k (A^{r+1})_{j1}$:
\begin{align*}
\sum_{j=1}^k (A^{r+1})_{j1} &=     
\sum_{j=1}^s (A^{r+1})_{j1} + 
\sum_{j=s+1}^k (A^{r+1})_{j1}
\\
& = s(s+1)^r
+ 
\sum_{i=1}^{k-s} ((B^{\top})_{i1} +1)
\\
&= 
s(s+1)^r + 
(k-s)
+
\sum_{i=1}^{k-s} (B^{\top})_{i1}
\\
& =
s(s+1)^r + 
(k-s)
+
(s+1)^r-1
=
(s + 1)^{r + 1} + k  - s - 1,
\end{align*}
which implies
\begin{align*}
\sum_{j=1}^k \sum_{v=1}^s (A^{r+1})_{jv}
=
s \sum_{j=1}^k (A^{r+1})_{j1}
 =
s\left( (s + 1)^{r + 1} + k  - s - 1 \right).
\end{align*}
We substitute this in \eqref{eq:closed form for first bs antidiagonal} and do standard manipulations to obtain the desired result.
\end{proof}

\begin{proof}[Proof Theorem \ref{thm:ProporStrides}]
Let $F(x) = \sum_{n\geq 0} b^{(k,s)}_{n+1} x^n$. By Theorem~\ref{lemma: gf vertices 1 dim} we have that $F(x)$ is a rational function of the form $P(x)/Q(x)$ where $Q(x)$ is given by \eqref{eq:denomGF prop srides} and $P(x)$ is given by \eqref{eq:numeratorGF prop srides}. We calculate the fraction 
$$
F(x) = \frac{(x-1)^2P(x)}{(x-1)^2Q(x)}.
$$
The denominator $(x-1)^2Q(x)$ is obtained in 
 Lemma \ref{lemma:denominators}. 
By Lemma \ref{lem:CoeffNmerator}, we have that
$(x-1)^2P(x) = \sum_{n \geq 0}^{r+2}c_{i+1}x^i$
with the coefficients $c_1,\ldots,c_{r+3}$ satisfying certain relationships on the variables $r,s,k$ and $b_j^{(k,s)}$. 
These expressions are simplified with Lemma~\ref{lem:TowardsNums1} to yield
\begin{align*}
 c_3=c_4=\cdots =c_{r+1}=0
 &&\text{and}& &
 c_{r+2} = 0 + sk.
\end{align*}
Next, we compute $c_{r+3}$. By the formula for $b_{r+3}^{(k,s)}$ in Lemma~\ref{lemma: br+3} and the closed formula in Lemma~\ref{lemma:prelim formula b} for $b_{r+1}^{(k,s)}$ and $b_{r+2}^{(k,s)}$, we obtain the following relation similar to \eqref{eq:recurrence first vals b}
\begin{equation*} \label{eq: recurr br+3}
b_{r+3}^{(k,s)} = 2(s+1)b_{r+2}^{(k,s)} -(s+1)^2b_{r+1}^{(k,s)} + (s+1)(rs-1).
\end{equation*}
This relation gives the following simple expression for $c_{r+3}$
 \[
 c_{r+3} = b_{r+3}^{(k,s)} -2(s+1)b_{r+2}^{(k,s)} +  (s+1)^2b_{r+1}^{(k,s)}
    +sb_2^{(k,s)} -s^2(r+1)b_1^{(k,s)} = -s(rs-1).
 \]
This gives the closed following closed formula for $F(x)$.
\[
F(x)  = \sum_{n\geq 0} b_{n+1}^{(k,s)} x^n = 
\frac{
s(r+1) -s(r+s+2)x +  s^2(r+1)x^{r+1} -  s(rs-1)x^{r+2}
}
{
1 - 2(s+1) x + (s +1)^2x^{2} + s x^{r + 1} - s^{2}(r + 1) x^{r + 2} + s(r s -1)x^{r + 3}.
}
\]
Since $G(x) = 1 + \sum_{n\geq 1} b_n^{(k,s)} x^n = 1+xF(x)$, substituting the expression above for $F(x)$ yields the desired closed formula for $G(x)$.
\end{proof}

\subsection{Asymptotics of number of vertices}
\label{sec:asymptotics}

We use the Perron--Frobenius formula (see Theorem~\ref{thm: main asymptotic theorem}) to obtain general asymptotics for the number  of vertices of the polytopes $P_{n,k,s}$. 

\begin{corollary} \label{cor: asymptotics gf vertices 1 dim}
Given positive integers $n,k,s$, let $b^{(k,s)}_n$ be the number of vertices of the polytope $P_{n,k,s}$ and $\lambda_1$ be the largest positive eigenvalue of $A_{k,s}$ then
\begin{equation} \label{eq: asymptotics b}
\lim_{n\to \infty} \frac1n \ln b^{(k,s)}_n = \ln \lambda_1.
\end{equation}
\end{corollary}

\begin{proof}
The matrix $A_{k,s}$ is irreducible since the companion directed graph is strongly connected. 
Indeed, the vertex $1$ has a directed edge to every vertex, and every vertex $m>1$ has a directed edge to some vertex $m'$ with $m'<m$, so there is necessarily a directed path from $m$ to $1$. 
This implies that each vertex has a directed path to every vertex.  
The result now follows from Theorem~\ref{lemma: gf vertices 1 dim} and the Perron--Frobenius theorem (Theorem~\ref{thm: main asymptotic theorem}). 
\end{proof} 

\begin{example} \label{ex: case k=3,s=1,asymptotics}
Continuing with Example~\ref{ex: case k=3,s=1} (see also  Examples~\ref{ex: k=3 and s=1} and \ref{ex: asymptotics k=3 and s=1}), the number $b_n^{(3,1)}$ of vertices of the polytope $P_{n,3,1}$ satisfies
\[
\lim_{n\to \infty} \frac1n \ln b_n^{(3,1)} = \ln \lambda_1 \approx 0.8096.
\]
See Figure~\ref{fig: graph lambda_1} for a matrix plot of $\lambda_1$ for $k=200$ and several values of $s\leq 200$.
\end{example}

\begin{figure}
 \begin{subfigure}[b]{0.44\textwidth}
  \centering
  \includegraphics[scale=0.5]{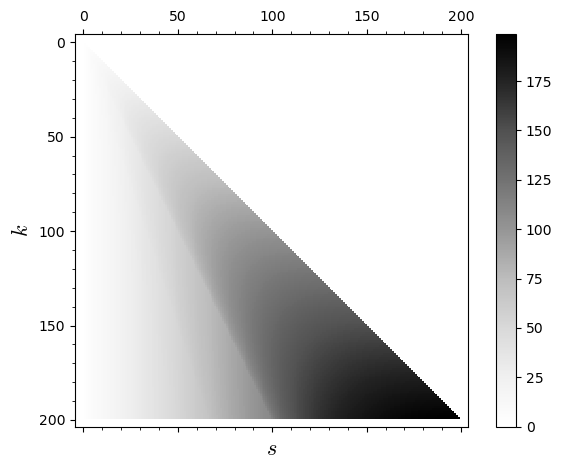}
  \caption{}
  \label{}
\end{subfigure}
\qquad 
\begin{subfigure}[b]{0.45\textwidth}
    \centering
    \raisebox{5pt}{
    \includegraphics[width=7cm]{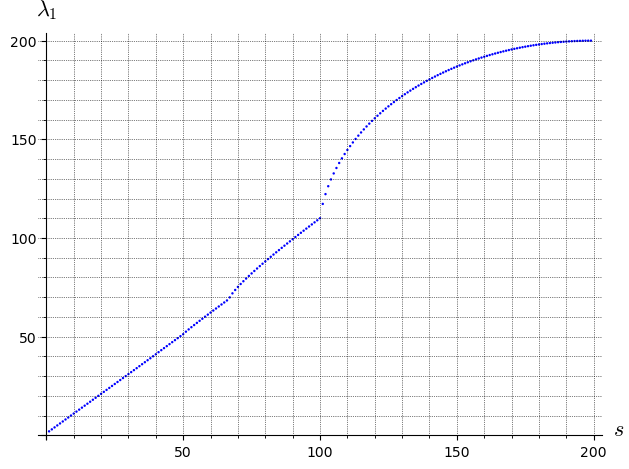}}
    \caption{}
      \label{}
\end{subfigure}
\caption{(A) matrix plot of the largest positive eigenvalue $\lambda_1$ of the matrix $A_{k,s}$ for stride $1\leq s\leq k \leq 200$, (B) plot for the case $1\leq s \leq k$ for $k=200$. This eigenvalue determines the asymptotics for the number $b_n^{(k,s)}$ of vertices of the polytope $P_{n,k,s}$ (see Equation~\ref{eq: asymptotics b}). There are some noticeable changes of regime when $\lfloor k/s\rfloor =3,2$.} 
\label{fig: graph lambda_1}
\end{figure}



We can say a bit more for the special case of large strides  from the explicit forms of the generating functions in Theorem~\ref{thm:LargeStrides}.

\begin{corollary}[Large strides $s\geq \lfloor k/2\rfloor$]
\label{cor: asympt LargeStrides}
Fix a positive integer $k$ and $s$ such that $s \in \{ \lfloor k/2\rfloor,\ldots,k-2 \}$. Then the number $b_n^{(k,s)}$ of vertices of the polytope $P_{n,k,s}$ satisfies
\[
\lim_{n\to \infty} \frac1n\ln b^{(k,s)}_n = \ln\left(\frac{2(k-s)(k-s-1)}{k-\sqrt{k^2-4(k-s)(k-s-1)}}\right). 
\]
\end{corollary}

\begin{proof}
Since in the rational generating function in Theorem~\ref{thm:LargeStrides}, the numerator and denominator do not have common roots, then we can apply Theorem~\ref{thm:asymptotics}. The result follows by using this result on the generating function, and using the quadratic formula  to find the smallest root of the denominator.
\end{proof}

Although we have an explicit generating function for the case of proportional strides in Theorem~\ref{thm:ProporStrides}, the degree of the denominator $Q(x)$ is $k$, and thus we do not expect a nice formula for the smallest root. See Example~\ref{ex: case k=3,s=1} for the case $k=3$, $s=1$.

\subsection{Number of facets}\label{sec_num_facets}

 In this section we give a count for the number of facets of $P_{n,k,s}$. 

\begin{theorem}\label{thm_facets_1d}
Let $s,k$ be positive integers. If 
$k>s+1$, then the number of facets of $P_{n,k,s}$ is $(s+2)(n-1)+k$. 
If 
$1<k\le s+1$, then the number of facets of $P_{n,k,s}$ is $kn$.
\end{theorem}

First, note that by Remark~\ref{rem_disjoint} if $1<k\le s$, then $P_{n,k,s}$ has $kn$ facets. 
For the remaining of this section we assume that $k\ge s+1$.
We now develop the tools we will use to prove this result under this assumption. Recall that  $K=|\Lambda|=s(n-1)+k$ (see Assumption~\ref{asumption}).

\begin{lemma}\label{lem_graphs_facets}
If $k\ge s+1$, then $\dim(P_{n,k,s})=K-1$ and the facets of $P_{n,k,s}$ correspond to graphs of the form $B\to A$ where $A,B$ form a partition of $\{0,1,\ldots,K-1\}$.
\end{lemma}

\begin{proof}
Proposition~\ref{prop: char vertices minkowski sum of simplices} implies that $P_{n,k,s}$ as a face of itself corresponds to the graph $G_\Pi$ with the least possible number of vertices.
Since $\lambda_i\cap\lambda_{i+1}\neq\varnothing$ for all $i$, then for $\Pi=(\Delta_{\lambda_0},\ldots,\Delta_{\lambda_{n-1}})$ we have that the graph $G_\Pi$ has exactly one vertex, which yields that the least possible number of vertices is one. 
All claims in the lemma now follow by using Proposition~\ref{prop: char vertices minkowski sum of simplices} again.
\end{proof}

We quickly remark that this lemma does not hold for $k<s+1$ since the facets of $P_{n,k,s}$ would correspond to graphs with $n+1$ vertices.

The preceding lemma implies that to characterize the facets under our running assumption we need to understand the structure of the graphs $G_\Pi =B\to A$.
The following lemma, which holds in the generality of Section~\ref{sec_faces_dag}, is a first step.

\begin{lemma}\label{lem_facet_obs}
Let $\Pi=(F_0,\ldots,F_{n-1})$ be such that $G_\Pi$ is the graph $B\to A$. 
For any $i$ we have that either $\lambda_i\cap B=\varnothing$ or $V(F_i)\subseteq B$. (Equivalently, either $\lambda_i\subseteq A$ or $A\cap V(F_i)=\varnothing$.)
\end{lemma}

\begin{proof}
Suppose $V(F_i)\not\subseteq B$ and let $y\in V(F_i)\setminus B$. For $x\in\lambda_i$ we either have $x\sim_\Pi y$ or $\bar y\to \bar x$.
However, the latter is not possible since $\bar y=A$.
It follows that $\lambda_i\subseteq A$, i.e.\ $\lambda_i\cap B=\varnothing$. 
\end{proof}

Now, we focus on how the structure of the windows that define $P_{n,k,s}$ leads to a very concrete description of its facets. 
In this section we will repeatedly use the following observations: 
\begin{enumerate}
    \item If $i\le j\le k$, then $\lambda_i\cap\lambda_k\subseteq \lambda_j$.
    \item  If $i\le j\le k$ and $\lambda_i\cap\lambda_k\neq\varnothing$, then $\lambda_j\subseteq \lambda_i\cup\lambda_k$.
\end{enumerate}

 \begin{lemma}\label{lem_vp_facets}
    Let $s,k$ be positive integers such that $k\ge s+1$.
    If $\Pi=(F_0,\ldots,F_{n-1})$ is such that $G_\Pi$ is the graph $B\to A$ with $|A|>1$ and $B\neq\varnothing$, then there is $r\in[0,n-1]$ such that either
	$$
	A=	\lambda_0\cup\cdots\cup\lambda_r
	\qquad
	\text{ or }
	\qquad
	A=	\lambda_r\cup\cdots\cup\lambda_{n-1}
	.$$
 \end{lemma}

\begin{proof}
Let $x,y\in A$ with $x\neq y$.
Since $x\sim_\Pi y$, there is some $F_i$ such that $A\cap V(F_i)\neq \varnothing$.
Lemma~\ref{lem_facet_obs} implies $\lambda_i\subseteq A$ so we can consider
	$$
	m:=\min\{i\in \{0,1,\ldots,n-1\} \mid \lambda_i\subseteq A\}.
	$$
Suppose $m>0$ and let $b\in \lambda_{m-1}\cap B$ and $a\in \lambda_{j}\setminus\lambda_m$ where $j>m$.
If $b\sim_\Pi a$, then there exist $x_1,\ldots,x_\ell$ such that $b=x_1$, $a=x_\ell$ and for all $i<\ell$ there is some $F_{r_i}$ such that $x_i,x_{i+1}\in V(F_{r_i})$.
Note that in particular,  $x_{i}\in \lambda_{r_{i-1}}\cap \lambda_{r_i}$ for $i=2,\ldots, \ell-1$.
Moreover, since $x_1\in \lambda_{m-1}$ and $x_{\ell}\in\lambda_j$, there is an $i$ such that $r_{i-i}\leq m \leq r_i$ from which it follows that 
$
x_{i}\in\lambda_{r_{i-1}}\cap\lambda_{r_i}\subseteq\lambda_m.
$
However, since $x_i\sim_\Pi b$ we have that $x_i\in B\cap\lambda_m$
contradicting $\lambda_m\subseteq A$.
Therefore $\lambda_j\subseteq A$ for all $j \geq m$.
It follows that $A=	\lambda_m \cup\cdots\cup\lambda_{n-1}$.
Now, suppose that $m=0$, let
	$$M:=\max\{r\in \{0,1,\ldots,n-1\} \mid \lambda_r\subseteq A\}
	$$
and suppose $M<n-1$.
A minor tweak to the argument above leads to $\lambda_j\subseteq A$ for all $j \leq M$, and from this it follows that $A=	\lambda_0\cup\cdots\cup\lambda_M$.

Last, we show that $m=0$ and $M=n-1$ is not possible. 
Let $\epsilon$ be the smallest index $i$ such that $\lambda_i\cap B\neq \varnothing$.
Note that Lemma~\ref{lem_facet_obs} implies $V(F_\epsilon)\subseteq B$.
Since $0,K-1\in A$, then $0\sim_\Pi K-1$ and there exist $x_1,\ldots,x_\ell\in A$ such that $x_1=0$, $x_\ell=K-1$
and for all $i<\ell$ there is some $F_{r_i}$ such that $x_i,x_{i+1}\in V(F_{r_i})$.
Since $x_1\in\lambda_0$ and $x_\ell\in\lambda_{n-1}$, there is a $j$ such that $r_j<\epsilon\le r_{j+1}$ and
	$$
    \lambda_{r_j}\cap \lambda_{r_{j+1}}\neq\varnothing
    \quad\Longrightarrow\quad
	V(F_\epsilon)\subseteq
	\lambda_\epsilon\subseteq \lambda_{r_j}\cup \lambda_{r_{j+1}}.
	$$
Now, suppose there exists $b\in V(F_\epsilon)\cap \lambda_{r_j}$.
However, since $V(F_\epsilon)\subseteq B$, Lemma~\ref{lem_facet_obs} would imply $V(F_{r_j})\subseteq B$ contradicting $x_j\in A$.
It follows that $V(F_\epsilon)\subseteq \lambda_{r_{j+1}}$, which, together with $V(F_\epsilon)\subseteq B$, imply that $V(F_{r_{j+1}})\subseteq B$.
However, this would imply $x_{j+1}\in B$, a contradiction. 
The lemma follows.
\end{proof}

We are now ready to count the number of facets of $P$.
    
    \begin{proof}[Proof of Theorem~\ref{thm_facets_1d}]
By Lemma~\ref{lem_graphs_facets}, we need to count the number of acyclic graphs $G_\Pi$ of the form $B\to A$
where $A,B$ form a partition of $\{0,1,\ldots,K-1\}$.

By Lemma~\ref{lem_vp_facets}, the number of such graphs with $|A|>1$ is at most $2(n-1)$.
To prove equality, let us show that for any graph as in the lemma, there is a choice of $\Pi=(F_0,\ldots,F_{n-1})$ such that $G_\Pi$ agrees with the graph.
Concretely, let $B\to A$ with $A=\lambda_0\cup\cdots\cup\lambda_r$ for some $r\in\{0,1,\ldots,n-2\}$.
For each $i$, let 
	$$
	F_i=
	\begin{cases}
	\conv{e_j\mid i\in \lambda_i}, & i\le r\\
	\conv{e_j\mid j\in\lambda_i\setminus A}, & i> r
	\end{cases}
	,$$
set $\Pi=(F_0,\ldots,F_{n-1})$ and note that $G_\Pi$ is the graph $B\to A$. The argument is analogous if $A=\lambda_r\cup\cdots\cup\lambda_{n-1}$ for some $r\in \{1,2,\ldots,n-1\}$.

Now we need to count the number of graphs $B\to A$ with $|A|=1$.
This count depends on whether $k>s+1$ or $k=s+1$.
If $k>s+1$, we claim that given $a\in\{0,\ldots,K-1\}$ there exists $\Pi$ such that $G_\Pi$ is the graph $\{0,\ldots,K-1\}\setminus \{a\}\to \{a\}$.
Indeed, for each $i$ let
	\begin{equation}\label{eq_singleton_poset}
	F_i=
	\begin{cases}
	\conv{e_j\mid j\in\lambda_i\setminus\{a\}}, & a\in\lambda_i\\
	\conv{e_j\mid j\in\lambda_i}, & a\notin\lambda_i
	\end{cases}
	\end{equation}
and consider the graph $G_\Pi$ associated to $\Pi=(F_0,\ldots,F_{n-1})$.
First, note that since $a\notin V(F_i)$ for all $i$, then $|\bar a|=1$.
Next, let us show that $G_\Pi$ has exactly two vertices. Since $k>s+1$, then $|(\lambda_i\cap\lambda_{i+1})\setminus\{a\}|\ge 1$ for all $i$. It follows that $x\sim_\Pi b\sim_\Pi y$ for all $b\in(\lambda_i\cap\lambda_{i+1})\setminus\{a\}$, $x\in\lambda_i\setminus\{a\}$, and $y\in\lambda_{i+1}\setminus\{a\}$. Repeating this for all $i$ we have that the vertices of $G_\Pi$ are $A=\{a\}$ and $B=\{0,\ldots,K-1\}\setminus \{a\}$. It is immediate that $G_\Pi$ has the edge $B\to A$.
Last, we show that $G_\Pi$ is acyclic.
If not, then by definition of this graph there would be an $i$ such that $B\cap V(F_i)\neq \varnothing$ and $B\cap(\lambda_i\setminus V(F_i))\neq\varnothing$.
However, the second condition would force both $V(F_i)=\lambda_i\setminus \{a\}$ and $a\in B$, a contradiction. 
Using Assumption~\ref{asumption} we conclude that if $k>s+1$, then $P_{n,k,s}$ has a total of
	$$
	2(n-1)+K=(s+2)(n-1)+k
	$$
facets, as desired. 

To finish the proof, suppose that $k=s+1$. 
We claim that there exists $\Pi$ such that 
    \begin{equation}\label{eq_some_facet}
    G_\Pi= \{0,1,\ldots,K-1\}\setminus\{a\}\to \{a\}
    \end{equation}
if and only if there is no $r\in[n-1]$ such that $a=r(k-1)+1$.
Note that assuming the claim holds as well as Assumption~\ref{asumption} leads to $P_{n,k,s}$ having
	$$
	2(n-1)+K-(n-1)=(n-1)+k+(n-1)s=kn
	$$
facets, as desired.
First,
suppose $a$ is such that there is no $r\in[n-1]$ such that $a=r(k-1)$ and let $\Pi=(F_0,\ldots,F_{n-1})$ be as in \eqref{eq_singleton_poset}.
Since $(\lambda_i\cap\lambda_{i+1})\setminus\{a\}\neq\varnothing$ for all $i$, we can repeat the argument in the preceding paragraph to prove that \eqref{eq_some_facet} holds.

Finally, let $a=r(k-1)+1$ with $r\in[n-1]$ and note that $\lambda_{r-1}\cap\lambda_r=\{a\}$. 
Suppose that $\Pi=(F_0,\ldots,F_{n-1})$ is such that $G_\Pi$ is as in \eqref{eq_some_facet}.
Since $|\bar a|=1$ and there is no edge with source $\bar a$, then there is no $i$ such that $a\in V(F_i)$.
Let $b\in\lambda_{r-1}\setminus\{a\}$ and $b'\in\lambda_{r}\setminus\{a\}$.
Since $b\sim_\Pi b'$, then there exist $x_1,\ldots,x_\ell$ such that $x_1=b$, $x_\ell=b'$ and for all $i<\ell$ there is some $F_{j_i}$ such that $x_i,x_{i+1}\in V(F_{j_i})$.
Now, $x_1\in \lambda_{r-1}$ and $x_{\ell}\in\lambda_r$ imply that there is an $i$ such that $j_{i-1}<r\le j_{i}$ which leads to 
    $$
    x_i\in V(F_{j_{i-1}})\cap V(F_{j_{i}}) \subseteq \lambda_{j_{i-1}}\cap \lambda_{j_{i}}\subseteq\{a\}
    ,$$
a contradiction.
Thus, $G_\Pi$ must have more than two vertices so it does not correspond to a facet.
    \end{proof}

We end this section by using our description of the rays of the normal fan of the polytope $P_{n,k,s}$ to give an inequality description of the polytope.  

\begin{corollary} \label{cor: H representation of P_n,k,s}
Suppose that $s,k$ are such that $k>s+1$.
The polytope $P_{n,k,s}$ is given by the following minimal inequality description:
\begin{eqnarray*}
\sum_{i\in[0,K-1]}x_i &=& n, 
\\
\sum_{i\notin \lambda_0\cup\cdots\cup\lambda_r} x_i &\le& n-1-r \qquad \text{for each $r\in[0,n-2]$},
\\
\sum_{i\notin \lambda_r\cup\cdots\cup\lambda_{n-1}} x_i &\le& r-1 \qquad \text{for each $r\in[n-1]$},
 \\
\sum_{i\neq a} x_i &\le& n \qquad \text{for each $a\in[0,K-1]$}.
\end{eqnarray*}
For $s,k$ with $k=s+1$, we have that $P_{n,k,s}$ is given by the same inequalities above, except that some of the ones of the fourth type are no longer used. Concretely, we only take:
\begin{eqnarray*}
\sum_{i\neq a} x_i &\le& n \qquad \text{for each $a$ for which there is no $r\in[n-1]$ with $a=r(k-1)+1$}.
\end{eqnarray*}
\end{corollary}

\begin{proof}
The first equation is the affine span of $P_{n,k,s}$ since we are adding $n$ simplices and the affine span of each simplex has as equation the sum of the coordinates equals $1$.
The minimal inequality description of $P_{n,k,s}$ is determined by the rays of its normal fan.
Throughout this proof we use the notation $e_B:=\sum_{i\in B} e_i\in\R^{[0,K-1]}$ for $B\subseteq[0,K-1]$.
The ray in $\R^{[0,K-1]}/\R(1,\ldots,1)$ corresponding to $B\to A$ is $e_B$. 
The corresponding linear functional is maximized by any vertex $v=e_{i_0}+\cdots+e_{i_{n-1}}$ such that $i_j\in B$ whenever $\lambda_j\cap B\neq \varnothing$ and 
    \begin{equation}\label{eq_facet_functional}
    e_B\cdot v=|\{j\in \{0,\dots,n-1\} \mid \lambda_j\cap B\neq \varnothing\}|.
    \end{equation}
The rest of the proof is a straightforward computation of \eqref{eq_facet_functional} for the facets appearing in the proof of Theorem~\ref{thm_facets_1d}.
\end{proof}


\section{Two-dimensional input layers}
\label{sec:2dim}

In this section we prove Theorem~\ref{recurrence.2D} counting the number of linearity regions $V_n$ of a max-pooling function on a $3\times n$ input using pooling windows of size $2\times 2$ (see Figure~\ref{fig: case 3xn}). 
By Proposition~\ref{connection.between.linearity.regions.and.face.counting} this number of linearity regions is equal to the number of vertices $V_n$ of the polytope 
\begin{equation}  \label{eqn.definition.qn}
Q_n = \sum_{\substack{0\leq i \leq 1 \\ 0 \leq j \leq n-2}} \Delta_{i,j},
\end{equation}
in the Euclidean space $\mathbb{R}^{3 \times n} \cong \mathbb{R}^{3n}$ with basis $\{e_{i,j} \ | \  0 \leq i \leq 2,\ 0 \leq j \leq n-1 \}$, 
where $\Delta_{i,j}$ denotes the polytope \[\Delta_{i,j}=\conv{e_{i,j},e_{i,j+1},e_{i+1,j},e_{i+1,j+1}}.\]

\begin{figure}[h]
    \centering
    \includegraphics[width=0.7\textwidth]{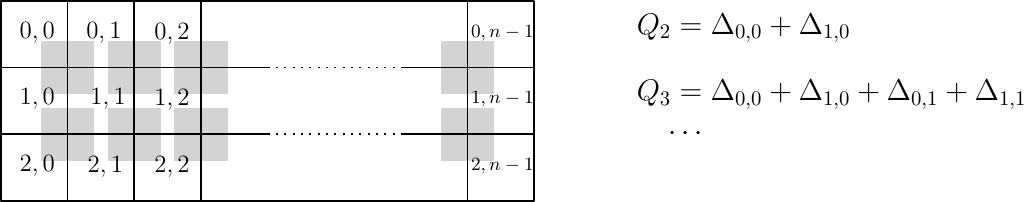}
    \caption{The $3\times n$ input with the $2n-2$ pooling windows of size $2\times 2$, and the Minkowski sums representing the polytopes $Q_2$ and  $Q_3$.}
    \label{fig: case 3xn}
\end{figure}

Before stating the main theorem of this Section, let us discuss the case in which the input array is $2\times n$ and the window size is $2\times 2$.

\begin{example}[Max-pooling on a $2\times n$ input with $2\times 2$ pooling windows] \label{example.2D.2xn}
By Proposition~\ref{connection.between.linearity.regions.and.face.counting} the number of linearity regions $V'_n$ of a max-pooling function on a $2\times n$ input using pooling windows of size $2\times 2$ is equal to the number of the vertices of the polytope 
$Q'_n=\sum_{0 \leq j \leq n-2} \Delta_{0,j}$. 
Up to a relabelling of the vertices, the polytope $Q'_n$ coincides with the polytope $P_{n-1,4,2}$ from Section~\ref{sec:enumeration}, whose number of vertices is given by both Theorem~\ref{thm:LargeStrides} and Theorem~\ref{thm:ProporStrides}. By Theorem~\ref{thm:LargeStrides}, the sequence $V'_n$ is given by 
the generating function  
\begin{equation} 
x + \sum_{n \geq 2} V'_n x^n 
= \frac{x}{1-4x+2x^2}
= x + 4 x^2 + 14 x^3 + 48 x^4 + 164 x^5 + \cdots 
\end{equation}
Equivalently, the sequence $V'_n$ is given by 
the recurrence $V'_{n+2}= 4V'_{n+1} - 2V'_{n}$, for $n \geq 2$, with initial values $V'_2=4$ and $V'_3=14$. 
\end{example}

We now state the main theorem of this section.

\begin{theorem} \label{recurrence.2D}
The number of vertices $V_n$ of the polytope $Q_n$ 
is given by the generating function
\begin{equation} \label{eq:gf Vn}
x + \sum_{n\geq 2} V_n x^n = \frac{x+x^2-x^3}{1-13x+31x^2-20x^3+4x^4} = x+ 14 x^2 + 150 x^3 + 1536 x^4 + 15594 x^5 + \cdots    
\end{equation}
Equivalently, $V_n$ 
is given by the recurrence relation
\begin{equation} \label{recurrence.2D.in.theorem}
    V_{n+4}= 13V_{n+3} - 31V_{n+2} + 20V_{n+1} - 4V_n,
\end{equation}
for all $n \geq 2$, with initial values $V_2=14$, $V_3=150$, $V_4=1536$ and $V_5=15594$. 
\end{theorem}

With the aim of proving this theorem, which we do in Section~\ref{sec_proof_2D}, we now discuss some notation and necessary results.

For the purposes of an inductive argument on $n$ we identify the vectors $e_{i,j} \in \mathbb{R}^{3 \times n}$ and $e_{i,j} \in \mathbb{R}^{3 \times (n+1)}$ with each other. 
This in particular allows us to realize $\mathbb{R}^{3 \times n}$ as a linear subspace of $\mathbb{R}^{3 \times (n+1)}$. For each $n \geq 2$ we define the left, middle and right polytopes as follows 
\begin{equation}  \label{eqn.definition.ln.mn.qn}
L_n=\Delta_{0,0}+\Delta_{1,0},   \qquad 
M_n = \sum_{\substack{0\leq i \leq 1 \\ 1 \leq j \leq n-3}} \Delta_{i,j},  \qquad
R_n=\Delta_{0,n-2}+\Delta_{1,n-2}. 
\end{equation}
Notice that $L_n$, $M_n$ and $R_n$ are respectively the Minkowski sums of the summands with $j=0$, with $1\leq j \leq n-3$ and with $j=n-2$
in the Minkowski sum in \eqref{eqn.definition.qn}, and hence $Q_n=L_n+M_n+R_n$.

\begin{notation}   \label{notation.vertices}
By Proposition~\ref{prop: char vertices minkowski sum of simplices} the polytope $Q_2$ in $\mathbb{R}^6$ has the following 14 vertices: 
$e_{0,0}+e_{1,0}$, $e_{0,0}+e_{1,1}$, $e_{0,0}+e_{2,0}$, $e_{0,0}+e_{2,1}$, 
$e_{0,1}+e_{1,0}$, $e_{0,0}+e_{1,1}$, $e_{0,1}+e_{2,0}$, $e_{0,1}+e_{2,1}$,
$e_{1,0}+e_{1,0}$,                    $e_{1,0}+e_{2,0}$, $e_{1,0}+e_{2,1}$, 
                   $e_{1,1}+e_{1,1}$, $e_{1,1}+e_{2,0}$, and $e_{1,1}+e_{2,1}$.
These will be respectively denoted by 
\onetwo, \onefive, \onethree, \onesix, 
\fourtwo, \fourfive, \fourthree, \foursix, 
\twotwo, \twothree, \twosix, 
\fivefive, \fivethree, and \fivesix \ (see Table~\ref{table.little.boards}).
We will fix this order for the rest of the section. 
\end{notation}

\begin{table}[h]
$$\begin{array}{|c|c|c|c|c|c|c|}
\hline
e_{0,0}+e_{1,0} & 
e_{0,0}+e_{1,1} & 
e_{0,0}+e_{2,0} & 
e_{0,0}+e_{2,1} &
e_{0,1}+e_{1,0} & 
e_{0,0}+e_{1,1} & 
e_{0,1}+e_{2,0} 
\\ \hline
& & & & & & \\[-5pt]
\bigonetwo & \bigonefive & \bigonethree & \bigonesix &
\bigfourtwo & \bigfourfive & \bigfourthree
\\[7pt] \hline
\end{array}$$
\\
$$\begin{array}{|c|c|c|c|c|c|c|}
\hline
e_{0,1}+e_{2,1} &
e_{1,0}+e_{1,0} &                  e_{1,0}+e_{2,0} & 
e_{1,0}+e_{2,1} &
e_{1,1}+e_{1,1} & 
e_{1,1}+e_{2,0} & 
e_{1,1}+e_{2,1} 
\\ \hline
& & & & & & \\[-5pt]
\bigfoursix & 
\bigtwotwo & \bigtwothree & \bigtwosix & 
\bigfivefive & \bigfivethree & \bigfivesix
\\[7pt] \hline
\end{array}
$$
\caption{Dictionary for the notation introduced in Notation~\ref{notation.vertices}.}
\label{table.little.boards}
\end{table}

\begin{notation}  \label{notation.vertices.n.dimensional}
The polytopes $L_n$ and $R_n$ in $\mathbb{R}^{3 \times n}$ can be identified with the polytope $Q_2$ in $\mathbb{R}^2$ via $e_{i,j} \mapsto e_{i,j}$ and $e_{i,j} \mapsto e_{i,j-n+2}$, respectively. 
We will use these identifications to refer to each of the 14 vertices of $L_n$ and each of the 14 vertices of $R_n$ by the name of the corresponding vertex of $Q_2$ as in Notation~\ref{notation.vertices}. 
For example, the vertex \onetwo \ of $L_n$ is $e_{0,0}+e_{1,0}$ and 
 the vertex \onetwo \ of $R_n$ is $e_{0,n-2}+e_{1,n-2}$.  
We get total orders on the vertices of $L_n$ and of $R_n$ induced by the order that we fixed on the vertices of $Q_2$ in Notation~\ref{notation.vertices}. 
\end{notation}

\begin{definition}  \label{definition.matrix.A}
Let $A$ be the matrix of size $14 \times 14$ such that for each $1 \leq i,j \leq 14$, its entry in position $(i,j)$ is equal to $1$ if the $j$-th vertex of $L_3$ plus the $i$-th vertex of $R_3$ is a vertex of $Q_3=L_3+R_3$ and it is equal to $0$ otherwise, where the order of the vertices of $L_3$ and $R_3$ is as in Notation~\ref{notation.vertices.n.dimensional}. Explicitly,   
\[A = \begin{bmatrix}
1 & 1 & 1 & 0 & 1 & 1 & 1 & 0 & 1 & 1 & 0 & 0 & 0 & 0 \\
1 & 1 & 1 & 1 & 1 & 1 & 1 & 1 & 1 & 1 & 0 & 0 & 0 & 0 \\
1 & 0 & 1 & 1 & 1 & 0 & 1 & 1 & 1 & 1 & 1 & 0 & 0 & 0 \\
1 & 1 & 1 & 1 & 1 & 1 & 1 & 1 & 1 & 1 & 1 & 0 & 0 & 0 \\
1 & 1 & 1 & 0 & 1 & 1 & 1 & 0 & 1 & 1 & 0 & 1 & 1 & 0 \\
1 & 1 & 1 & 1 & 1 & 1 & 1 & 1 & 1 & 1 & 1 & 1 & 1 & 1 \\
1 & 0 & 1 & 1 & 1 & 0 & 1 & 1 & 1 & 1 & 1 & 0 & 1 & 1 \\
1 & 1 & 1 & 1 & 1 & 1 & 1 & 1 & 1 & 1 & 1 & 1 & 1 & 1 \\
1 & 1 & 1 & 0 & 0 & 0 & 0 & 0 & 1 & 1 & 0 & 1 & 1 & 0 \\
1 & 0 & 1 & 1 & 0 & 0 & 0 & 0 & 1 & 1 & 1 & 0 & 1 & 1 \\
1 & 1 & 1 & 1 & 0 & 0 & 0 & 0 & 1 & 1 & 1 & 1 & 1 & 1 \\
1 & 1 & 1 & 1 & 1 & 1 & 1 & 1 & 1 & 1 & 1 & 1 & 1 & 1 \\
1 & 0 & 1 & 1 & 0 & 0 & 1 & 1 & 1 & 1 & 1 & 0 & 1 & 1 \\
1 & 1 & 1 & 1 & 1 & 1 & 1 & 1 & 1 & 1 & 1 & 1 & 1 & 1
\end{bmatrix}, 
\]
which can be obtained by using software or by computing its entries as in Example~\ref{example.matrix.A} below. 
\end{definition}

\begin{example} \label{example.matrix.A}
Given vertices 
$p$ of $L_3=\Delta_{0,0}+\Delta_{1,0}$ and $q$ of $R_3=\Delta_{0,1}+\Delta_{1,1}$, we can write $p$ uniquely as a sum of vertices of $\Delta_{0,0}$ and $\Delta_{1,0}$, and we can write $q$ uniquely as a sum of vertices of $\Delta_{0,1}$ and $\Delta_{1,1}$.
As in Section~\ref{sec:acyclic.graph}, we get a directed graph associated to this sum of four vertices of the four distinct summands in the Minkowski sum $Q_3=\Delta_{0,0}+\Delta_{1,0}+\Delta_{0,0}+\Delta_{1,0}$.  
This graph is acyclic as a directed graph if and only if $p+q$ is a vertex of $Q_3=L_3+R_3$. 
For example, if $p=\onethree$ as a vertex of $L_3$ and $q=\onetwo$ as a vertex $R_3$, one can see that the associated directed graph contains no cycles, and then the entry $(1,3)$ of $A$ is $1$.
On the other hand, taking $p=\twosix$ and $q=\onefive$ produces a directed graph with a cycle, then the entry $(2,11)$ of $A$ is $0$.
\end{example}

\begin{example}
The matrix $A$ has 150 of its 196 entries equal to $1$ and the remaining 46 entries are equal to $0$. 
By construction, there is a bijection between the vertices of $Q_3=L_3+R_3$ and the entries of $A$ that are equal to $1$. 
Hence, the polytope $Q_3$ has $150$ vertices.
\end{example}

Since $Q_n=L_n+M_n+R_n$, then each vertex of $Q_n$ is a sum of a vertex of $L_n$, a vertex of $M_n$ and a vertex of $R_n$. 
Moreover, those three vertices are uniquely determined. 
Then, for each vertex $p$ of $Q_n$ there exist a unique vertex of $Q_2$, which we denote by $\pi(p)$, such that $p$ is the sum of the vertex of $R_n$ corresponding to $\pi(p)$ (under the identification of $Q_2$ and $R_n$) and some vertices of $L_n$ and $M_n$. 

\begin{definition} \label{notation.subindex} 
For each vertex $q$ of $Q_2$ and each $n \geq 2$, let $q_n$ denote the number of vertices $p$ of $Q_n$ such that $\pi(p)=q$. 
\end{definition}

\begin{example}
To practice the notation introduced in Definition~\ref{notation.subindex}, 
notice that for each $n \geq 2$, $V_n$ is equal to the sum of the numbers $q_n$ as $q$ ranges over the 14 vertices of $Q_2$.  
In other words, 
\[V_n=\onetwo_n+\onefive_n+\onethree_n+\onesix_n+
\fourtwo_n+\fourfive_n+\fourthree_n+\foursix_n+
\twotwo_n+\twothree_n+\twosix_n+
\fivefive_n+\fivethree_n+\fivesix_n.
\]
\end{example}

\begin{lemma} \label{lemma.basic.relations.Vs}
For each $n \geq 2$, we have  

\begin{enumerate}[label=(\alph*)] \setlength{\itemsep}{3mm}
    \item \label{lemma.basic.relations.Vs.a}
$\onetwo_n = \twothree_n,  \quad 
\onefive_n = \fivethree_n,   \quad
\twosix_n = \fourtwo_n   \quad
\textnormal{and}   \quad
\onesix_n = \fourthree_n$. 
\item \label{lemma.basic.relations.Vs.b} 
$\fourfive_n = 
\foursix_n =  
\fivefive_n =
\fivesix_n$ and moreover if $n \geq 3$ their common value is $V_{n-1}$. 
\end{enumerate}
\end{lemma}

\begin{proof}
The equalities in part \ref{lemma.basic.relations.Vs.a} hold by symmetry.  
The case $n=2$ of part \ref{lemma.basic.relations.Vs.b} holds since $\fourfive_2 = 
\foursix_2 =  
\fivefive_2 =
\fivesix_2=1$. 
Then, to prove part \ref{lemma.basic.relations.Vs.b}, we fix $n \geq 3$ and show that each of the subsets of the vertices of $Q_n$ counted in $\fourfive_n$, $\foursix_n$, $\fivefive_n$  and $\fivesix_n$
is in bijective correspondence with the set of vertices of $Q_{n-1}$, which by definition has cardinality $V_{n-1}$.  
Each vertex of $Q_n=Q_{n-1}+R_n$ can be written in a unique way as the sum of a vertex of $Q_{n-1}$ and a vertex of $R_n$. 
We claim that reciprocally for each vertex $q$ of $Q_{n-1}$ and either one of the vertices 
$\fourfive,   \  
\foursix,    \ 
\fivefive,  \ 
\fivesix$ 
of $R_n$, which we denote by $r$, we have that $q+r$ is a vertex of $Q_n$. 
The desired bijections follow by proving this claim. 

To prove the claim, first notice that the vertex $q$ of $Q_{n-1}$ is a sum of vertices of the $2n-4$ polytopes $\Delta_{i,j}$ in the Minkowski sum \eqref{eqn.definition.qn} defining $Q_{n-1}$, one vertex from each summand. 
Let $\Pi_{n-1}$ be the list of these $2n-4$ vertices.  
Similarly, $r$ is a sum of vertices of the two polytopes $\Delta_{i,j}$ in the Minkowski sum \eqref{eqn.definition.ln.mn.qn} defining $R_n$, one vertex from each summand.  
Combining the expressions for $q$ and $r$, we get an expression for $q+r$ as a sum vertices of the $2n-2$ polytopes $\Delta_{i,j}$ in the Minkowski sum \eqref{eqn.definition.qn} defining $Q_n=Q_{n-1}+R_n$, one vertex from each summand. 
Let $\Pi_{n}$ be the list of these $2n-2$ vertices.

We get digraphs $G_{\Pi_{n-1}}$ and $G_{\Pi_n}$ associated to $\Pi_{n-1}$ and $\Pi_{n}$, as defined in Section~\ref{sec_faces_dag}.
Their respective sets of vertices are $\{e_{i,j} \ | \  0 \leq i \leq 2, 0 \leq j \leq n-2 \}$ and $\{e_{i,j} \ | \  0 \leq i \leq 2, 0 \leq j \leq n-1 \}$. 
We know that $G_{\Pi_{n-1}}$ is acyclic as a directed graph since $q$ is a vertex of $Q_{n-1}$. 
The proof will be complete if we show that $G_{\Pi_{n}}$ is acyclic as a directed graph as well.

The vertices of $G_{\Pi_{n}}$ are the disjoint union of the vertices of $G_{\Pi_{n-1}}$ and $\{e_{0,n-1},e_{1,n-1},e_{2,n-1}\}$. 
Since the vertex $r$ of $R_n$ is one of $\fourfive, \   
\foursix,    \ 
\fivefive,  \ 
\fivesix$, the graph $G_{\Pi_{n}}$ has the following properties: 
(i) the restriction of the directed graph $G_{\Pi_{n}}$ to the set of vertices of $G_{\Pi_{n-1}}$ is precisely $G_{\Pi_{n-1}}$; 
(ii) there are no directed edges from the vertices of $G_{\Pi_{n-1}}$ to the vertices 
$\{e_{0,n-1},e_{1,n-1},e_{2,n-1}\}$;
(iii) between pairs of vertices in $\{e_{0,n-1},e_{1,n-1},e_{2,n-1}\}$ there are exactly two directed edges and they do not form an oriented cycle. 
It follows that the directed graph $G_{\Pi_{n}}$ is acyclic. 
\end{proof}


\begin{lemma}  \label{lemma.minimal.cycles}
Let $n \geq 4$ and let $v_{i,j}$ be a vertex of $\Delta_{i,j}$ for each $0 \leq i \leq 1$ and $0 \leq j \leq n-2$. Let 
\begin{equation*} 
l_n= v_{0,0}+v_{1,0},   \qquad 
m_n = \sum_{\substack{0\leq i \leq 1 \\ 1 \leq j \leq n-3}} v_{i,j},  \qquad
r_n=v_{0,n-2}+v_{1,n-2}. 
\end{equation*}
Suppose that $l_n+m_n$ is a vertex of $L_n+M_n$ and $m_n+r_n$ is a vertex of $M_n+R_n$, but $l_n+m_n+r_n$ is not a vertex of $Q_n=L_n+M_n+R_n$. 
Then, one of the following two cases holds
\[
v_{i,j}=
\begin{cases}
  e_{i,j}  & \text{ for } i=0, \ 1\leq j \leq n-2 \\
  e_{i+1,j} & \text{ for } i=0, \ j=0 \\ 
  e_{i,j+1}  & \text{ for }  i=1, \ j=n-2 \\
  e_{i+1,j+1} & \text{ for }  i=1, \ 0\leq j \leq n-3
\end{cases}
\qquad \text{ or } \qquad
v_{i,j}=
\begin{cases}
  e_{i,j}  & \text{ for }  i=1, \ j=0 \\
  e_{i+1,j} & \text{ for } i=1, \ 1\leq j \leq n-2 \\ 
  e_{i,j+1}  & \text{ for }  i=0, \ 0\leq j \leq n-3 \\
  e_{i+1,j+1} & \text{ for }  i=0, \  j=n-2. 
\end{cases}
\]
Intuitively, these two cases correspond to the two possible orientations of the following cycle represented on a $3 \times n$ board with rows $0\leq i \leq 2$ and columns $0\leq j \leq n-1$:
\begin{center}
\includegraphics[scale=0.2]{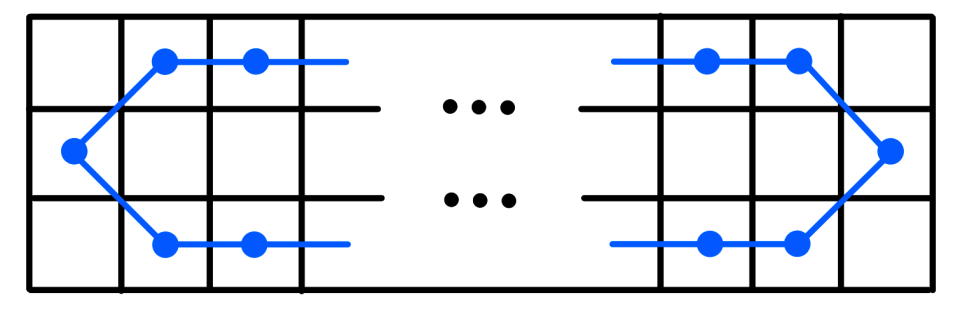}    
\end{center}
\end{lemma}

\begin{proof}
As in Section~\ref{sec_faces_dag}, we get a directed graph $G_{\Pi}$ associated to the list $\Pi$ of the vertices $v_{i,j}$ of $\Delta_{i,j}$ for $0 \leq i \leq 1$ and $0 \leq j \leq n-2$. 
The vertices in the list $\Pi$ add to the point $l_n+m_n+r_n$ in the Minkowski sum $Q_n$.  
By Proposition~\ref{prop: char vertices minkowski sum of simplices} the graph $G_{\Pi}$ contains a directed cycle, because $l_n+m_n+r_n$ is not a vertex of $Q_n$. 
As in Remark~\ref{remark.subgraph.removing.sinks}, we consider the directed subgraph $\Gamma=\Gamma_{\Pi}$ of $G_\Pi$ obtained by removing all vertices that are sinks from $G_\Pi$. 
Notice that $\Gamma$ must contain a directed cycle because $G_{\Pi}$ contains a directed cycle. 

Let us consider a fixed directed cycle in the graph $\Gamma$. 
The vertices of $\Gamma$ are precisely the $v_{i,j}$, for $0 \leq i \leq 1$ and $0 \leq j \leq n-2$. 
We observe that at least one of the vertices $v_{0,0}$ and $v_{1,0}$ is in the cycle because the graph obtained from $\Gamma$ by removing $v_{0,0}$ and $v_{1,0}$ and the edges through them is acyclic, since $l_n+m_n$ is a vertex of $L_n+M_n$. 
Similarly, at least one of the vertices $v_{0,n-2}$ and $v_{1,n-2}$ is in the cycle because $m_n+r_n$ is a vertex of $M_n+R_n$.
We also observe that the vertices $v_{0,1},v_{0,2},\ldots,v_{0,n-3}$ must all be in the cycle. 
Indeed, if $v_{0,j}$ is not in the cycle for some $1 \leq j \leq n-3$, then the cycle must pass twice through $v_{1,j}$ which allows us to get a shorter cycle which either does not pass through $v_{0,0}$ and $v_{1,0}$ or does not pass through $v_{0,n-2}$ and $v_{1,n-2}$, which is a contradiction.
Similarly, the vertices $v_{1,1},v_{1,2},\ldots,v_{1,n-3}$ must all be in the cycle. 

For each $1 \leq j \leq n-3$, the vertices $v_{0,j}$ and $v_{1,j}$ are not connected by an edge in $\Gamma$ because otherwise there is a shorter cycle in $\Gamma$ which either does not pass through $v_{0,0}$ and $v_{1,0}$ or does not pass through $v_{0,n-2}$ and $v_{1,n-2}$, which is a contradiction. 
Therefore none of the vertices $e_{1,1},e_{1,2},\ldots,e_{1,n-2}$ can occur among $v_{0,1},v_{0,2},\ldots,v_{0,n-3}$ or among  $v_{1,1},v_{1,2},\ldots,v_{1,n-3}$ because this would produce one such forbidden edge.  

We deduce that our cycle either contains the directed subgraphs \[v_{0,1} \rightarrow v_{0,2} \rightarrow \cdots \rightarrow v_{0,n-3}
\qquad \textnormal{and} \qquad 
v_{1,1} \leftarrow v_{1,2} \leftarrow \cdots \leftarrow v_{1,n-3}, 
\]
or the directed subgraphs
\[v_{0,1} \leftarrow v_{0,2} \leftarrow \cdots \leftarrow v_{0,n-3}
\qquad \textnormal{and} \qquad 
v_{1,1} \rightarrow v_{1,2} \rightarrow \cdots \rightarrow v_{1,n-3}. \]
By symmetry, we can assume that we are in the former case (since we can reflect and relabel if necessary). 

We know that $v_{0,j}$ is equal to $e_{0,j}$ or $e_{0,j+1}$ for each $1 \leq j \leq n-3$, as we have discarded any other possibilities for those $v_{0,j}$.  
But the existence of the directed subgraph $v_{0,1} \rightarrow v_{0,2} \rightarrow \cdots \rightarrow v_{0,n-3}$ implies that 
$v_{0,j}=e_{0,j}$ for each $1 \leq j \leq n-3$. 
Similarly, 
$v_{1,j}=e_{2,j+1}$ for each $1 \leq j \leq n-3$. 

Next we determine the values of $v_{0,0}$ and $v_{1,0}$. 
Since there is a cycle in the directed graph $\Gamma$ and $v_{0,1}=e_{0,1}$ must be in the cycle, some vertex of $\Gamma$ must have a directed edge to $v_{0,1}=e_{0,1}$. 
At this point the only possibility for the vertex with an edge to $v_{0,1}=e_{0,1}$ is $v_{0,0}$. 

The possible values of $v_{0,0}$ are $e_{0,0}$, $e_{0,1}$, $e_{1,0}$, and $e_{1,1}$.  
However, $v_{0,0}=e_{0,0}$ and $v_{0,0}=e_{0,1}$ are not possible because independent of the value of $v_{1,0}$ there would be no cycle in the directed graph $\Gamma$ containing either $v_{0,0}$ or $v_{1,0}$, which is a contradiction.  
Now, $v_{0,0}=e_{1,1}$ is also not possible because this would create the cycle $v_{0,0}=e_{1,1} \rightarrow v_{0,1}=e_{0,1} \rightarrow v_{0,0}=e_{1,1}$, which by Proposition~\ref{prop: char vertices minkowski sum of simplices} contradicts that $l_n+m_n$ is a vertex of $L_n+M_n$. We deduce that $v_{0,0}=e_{1,0}$. 

The possible values of $v_{1,0}$ are $e_{1,0}$, $e_{1,1}$, $e_{2,0}$, and $e_{2,1}$.  
However, $v_{1,0}=e_{1,0}$ and $v_{1,0}=e_{2,0}$ are not possible since $v_{0,0}=e_{1,0}$ and hence there would be no cycle in the directed graph $\Gamma$ containing either $v_{0,0}$ or $v_{1,0}$, which is a contradiction.  
Now, $v_{1,0}=e_{1,1}$ is also not possible because this would create the cycle $v_{0,0}=e_{1,0} \rightarrow v_{1,0}=e_{1,1} \rightarrow v_{0,0}=e_{1,0}$ which by Proposition~\ref{prop: char vertices minkowski sum of simplices} contradicts that $l_n+m_n$ is a vertex of $L_n+M_n$. We deduce that $v_{1,0}=e_{2,1}$. 

By symmetry (or alternatively by the same argument used to get the values of $v_{0,0}$ or $v_{1,0}$), we deduce that $v_{0,n-2}=e_{0,n-2}$ and $v_{1,n-2}=e_{1,n-1}$. This completes the proof. 
\end{proof}

\begin{definition}
For each integer $n \geq 2$, let $\overline{V_n}=[V_{n,i}]_{1\leq i\leq 14}$ be the column vector whose 14 entries are 
$V_{n,1}=\onetwo_n$, $V_{n,2}=\onefive_n$, $V_{n,3}=\onethree_n$, $V_{n,4}=\onesix_n$, 
$V_{n,5}=\fourtwo_n$, $V_{n,6}=\fourfive_n$, $V_{n,7}=\fourthree_n$, $V_{n,8}=\foursix_n$, 
$V_{n,9}=\twotwo_n$, $V_{n,10}=\twothree_n$, $V_{n,11}=\twosix_n$,  $V_{n,12}=\fivefive_n$, $V_{n,13}=\fivethree_n$, and $V_{n,14}=\fivesix_n$.
This follows the same order of the vertices that we fixed earlier. 
Let $\overline{W_{n+1}}$ be the column vector $A\overline{V_n}$.
We denote the $i$-th entry of $\overline{W_n}$ by $W_{n,i}$, for $1 \leq i \leq 14$.   
For notational convenience we define the initial value $\overline{W_2}=[W_{2,i}]_{1\leq i\leq 14}$ to be equal to $\overline{V_2}$.

\end{definition}

\begin{example} \label{example.comparing.V.and.W}
By definition $\overline{W_2}=\overline{V_2}$ 
and hence 
$W_{2,i}=V_{2,i}=1$ for all $1 \leq i \leq 14$.  
By the definition of the matrix $A$ we have that $\overline{V_3}=A\overline{V_2}=\overline{W_3}$, and hence 
$V_{3,i}=W_{3,i}$ for all $1 \leq i \leq 14$.  
Notice that the sum of the entries of $\overline{V_n}$ is the number $V_n$ of vertices of the polytope $Q_n$. 
\end{example}

\begin{remark}  \label{remark.interpretation.W.i}
Let us discuss the interpretation of the numbers $W_{n,i}$ for fixed $n \geq 3$ and $1 \leq i \leq 14$.  
Let $v_i$ denote the $i$-th vertex of $R_{n}$ in the order of Notation~\ref{notation.vertices.n.dimensional}. 
Let $w$ be any vertex of $Q_{n-1}$. 
We can write $v_i$ uniquely as a sum of vertices $v_{0,n-2}$ of $\Delta_{0,n-2}$ and $v_{1,n-2}$ of $\Delta_{1,n-2}$, and we can write $w$ uniquely as a sum of vertices $v_{l,j}$ of $\Delta_{l,j}$ for $0\leq l \leq 1$ and $0 \leq j \leq n-3$. 
As in Section~\ref{sec_faces_dag}, we get a directed graph $G_{\Pi}$ where $\Pi$ is the list of these $2(n-1)$ vertices of the $2(n-1)$ summands in the Minkowski sum $Q_n= \sum \Delta_{l,j}$ where $0\leq l \leq 1$ and $0 \leq j \leq n-2$.
As in Remark~\ref{remark.subgraph.removing.sinks}, by removing the sinks and the directed edges to the sinks we get a directed subgraph $\Gamma_{v_i,w}:=\Gamma_{\Pi}$. 
The vertices of $\Gamma_{v_i,w}$ are precisely the $v_{l,j}$ for $0\leq l \leq 1$ and $0 \leq j \leq n-2$ and its directed edges are the directed edges between them in $G_{\Pi}$. 
Since all vertices removed when passing from $G_{\Pi}$ to $\Gamma_{v_i,w}$ are sinks, then $G_{\Pi}$ is acyclic as a directed graph if and only if $\Gamma_{v_i,w}$ is acyclic as a directed graph. 
Hence, we know that $\Gamma_{v_i,w}$ is acyclic as a directed graph if and only if $v_i+w$ is a vertex of $Q_n=Q_{n-1}+R_n$. 
In our present case $\Gamma_{v_i,w}$ might have cycles. 
By the definition of the matrix $A$, the number $W_{n,i}$ is equal to the number of pairs $(v_i,w)$ as above (that is, with $v_i$ the $i$-th vertex of $R_n$ and $w$ any vertex of $Q_{n-1}$) such that the directed graph $\Gamma_{v_i,w}$ has no directed cycles whose vertices $\{v_{l,j}\}$ all satisfy $0 \leq j \leq n-3$ and has no directed cycles whose vertices $\{v_{l,j}\}$ all satisfy $n-3 \leq j \leq n-2$.  
\end{remark}

\begin{proposition}  \label{proposition.comparing.V.and.W}
\begin{enumerate}[label=(\alph*)]
    \item For each integer $n \geq 2$ and each integer $1 \leq i \leq 14$, such that $i \neq 2$ and $i \neq 13$, we have that $V_{n,i}=W_{n,i}$.  \label{proposition.comparing.V.and.W.part.a}
    \item For each integer $n \geq 2$, we have that 
$V_{n,2}=V_{n,13}=V_{n,5}=V_{n,11}$.  \label{proposition.comparing.V.and.W.part.b}
\end{enumerate}
\end{proposition}
\begin{proof}
Since $V_{2,i}=W_{2,i}=1$ and $V_{3,i}=W_{3,i}$ for all $1 \leq i \leq 14$, we can assume that $n \geq 4$. 
Let us fix integers $n \geq 4$ and $1 \leq i \leq 14$ and start by comparing $V_{n,i}$ and $W_{n,i}$.   

Let $v_i$ denote the $i$-th vertex of $R_n$ in the order of Notation~\ref{notation.vertices.n.dimensional}. 
Recall that $V_{n,i}$ is equal to the number of vertices $q$ of $Q_n$ with $\pi(q)=\pi(v_i)$.
Any such vertex $q$ of $Q_n$ with $\pi(q)=\pi(v_i)$ is equal to a sum $q=w+v_i$ where $w$ is some uniquely determined vertex of $Q_{n-1}$, but not all pairs $(v_i,w)$ where $w$ is a vertex of $Q_{n-1}$ satisfy that $w+v_i$ is a vertex of $Q_n$. 

For each vertex $w$ of $Q_{n-1}$, let $\Gamma_{v_i,w}$ be the directed graph $\Gamma_{v_i,w}$ defined in Remark~\ref{remark.interpretation.W.i}.  
We know that $V_{n,i}$ is the number of pairs $(v_i,w)$, where $w$ is a vertex of $Q_{n-1}$, such that $\Gamma_{v_i,w}$ is acyclic as a directed graph. 
By Remark~\ref{remark.interpretation.W.i}, we know that $W_{n,i}$ is the number of pairs $(v_i,w)$, where $w$ is a vertex of $Q_{n-1}$,  
such that $\Gamma_{v_i,w}$ has no directed cycles whose vertices $\{v_{l,j}\}$ all satisfy $0 \leq j \leq n-3$ and has no directed cycles whose vertices $\{v_{l,j}\}$ all satisfy $n-3 \leq j \leq n-2$.

Then, $W_{n,i}-V_{n,i}$ is equal to the number of pairs $(v_i,w)$ where $w$ is a vertex of $Q_{n-1}$ such that the graph $\Gamma_{(v_i,w)}$ contains a directed cycle, 
but it has no directed cycles whose vertices $\{v_{l,j}\}$ all satisfy $0 \leq j \leq n-3$ and has no directed cycles whose vertices $\{v_{l,j}\}$ all satisfy $n-3 \leq j \leq n-2$.  
Fix one such pair $(v_i,w)$ and let $0 \leq m \leq n-2$ be the largest integer such that there is a directed cycle whose vertices $\{v_{l,j}\}$ all satisfy $m \leq j \leq n-2$.  
By assumption such $m$ exists and satisfies $0 \leq m \leq n-4$. 
It follows that the graph $\Gamma_{v_i,w}$ has no directed cycles whose vertices $\{v_{l,j}\}$ all satisfy $m \leq j \leq n-3$ and has no directed cycles whose vertices $\{v_{l,j}\}$ all satisfy $m+1 \leq j \leq n-2$.

Recall that our vertex $v_i$ of $R_n$ is a sum of a vertex $v_{0,n-2}$ of $\Delta_{0,n-2}$ and a vertex $v_{1,n-2}$ of $\Delta_{1,n-2}$ in a unique way. 
Also, our vertex $w$ of $Q_{n-1}$ is a sum of vertices $v_{l,j}$ of $\Delta_{l,j}$ for $0\leq l\leq 1$ and $0 \leq j \leq n-3$ in a unique way. 
Notice that among those $2(n-1)$ vertices, the $2(n-m-1)$ vertices of $\Delta_{l,j}$ with $m \leq j \leq n-2$ satisfy the assumptions of Lemma~\ref{lemma.minimal.cycles} on the rectangular subarray with rows $0 \leq l \leq 2$ and columns $m  \leq j \leq n-1$ (which necessarily has at least four columns).  
By Lemma~\ref{lemma.minimal.cycles} it follows that there are two possibilities for the directed cycle within $\Gamma_{v_i,w}$ whose vertices $\{v_{l,j}\}$ all satisfy $m \leq j \leq n-2$ and they have one of the two forms described in that lemma (see Figure~\ref{fig: new cycles}). 
\begin{figure}[h]
    \begin{center}
    \includegraphics[width=0.52\textwidth]{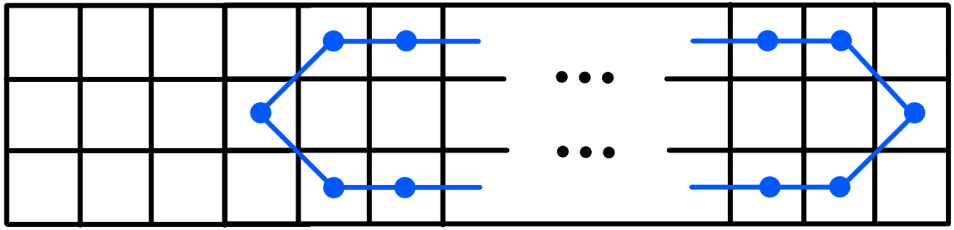}
    \caption{Form of the cycles that are counted in $W_{n,i}$ but not counted in $V_{n,i}$.}
    \label{fig: new cycles}
    \end{center}
\end{figure}

In particular, the pairs $(v_i,w)$ counted in $W_{n,i}-V_{n,i}$ as above only occur when $\pi(v_i) = \onefive$ or $\pi(v_i) = \fivethree$. 
This implies the claim in part \ref{proposition.comparing.V.and.W.part.a} since $i \neq 2$ and $i \neq 13$ implies $\pi(v_i)\neq \pi(v_2) = \onefive$ and $\pi(v_i)\neq \pi(v_{13}) = \fivethree$. 

Let us now focus on the case $i=2$. First we count the possible pairs $(v_2,w)$ that contribute to $W_{n,2}-V_{n,2}$, for each fixed $0 \leq m \leq n-4$ as above.   
The unique directed cycle within $\Gamma_{v_2,w}$ whose vertices $\{v_{l,j}\}$ all satisfy $m \leq j \leq n-2$ has the form in Figure~\ref{fig: new cycles colored}. 
\begin{figure}[h]
    \begin{center}
    \includegraphics[width=0.52\textwidth]{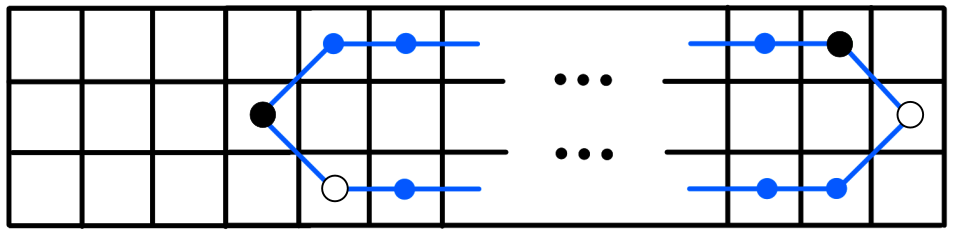}
    \caption{Form of the cycles that are counted in $W_{n,2}$ but not counted in $V_{n,2}$.}
    \label{fig: new cycles colored}
    \end{center}
\end{figure}

We can write $w$ in a unique way as $w=a+b+c$ where $a$, $b$ and $c$ are vertices of  
$\sum \Delta_{l,j}$ where $0\leq l \leq 1$ and respectively $0 \leq j \leq m-1$, $j=m$ and $m+1 \leq j \leq n-3$. 
Notice that the equation $w=a+b+c$ includes an abuse of notation since in the case $m=0$ there is no $a$ and $w=b+c$, but this does not affect the argument below.   
Since we have assumed $i=2$, $b$ and $c$ are determined by  Lemma~\ref{lemma.minimal.cycles}, and in particular $\pi(b)=\twosix$.
Then the possible values of $w$ are in correspondence with the possible values of $a$, which in turn are equal to the possible values of $a+b$.  
Now we notice that the possible values of $a+b$ are precisely the number of vertices of $\sum \Delta_{l,j}$ with $0 \leq l \leq 1$ and $0 \leq j \leq m$, such that $\pi(a+b)=\twosix$.
This number is by definition $V_{m+2,11}$, since $\twosix$ is the eleventh vertex of $Q_2$. 
Now, we let $m$ vary over its possible values $0 \leq m \leq n-4$ and conclude that for all $n \geq 4$ (and even for $n=2$ and $n=3$) we have 
\begin{equation} \label{equation.V.in.terms.of.W} 
V_{n,2}= W_{n,2}- \sum_{0 \leq m \leq n-4}  V_{m+2,11} = W_{n,2}-V_{2,11}-V_{3,11}-\cdots-V_{n-2,11}.
\end{equation}

Let us show by induction that $V_{n,2}=V_{n,13}=V_{n,5}=V_{n,11}$ for all $n \geq 2$ as claimed in part \ref{proposition.comparing.V.and.W.part.b}. %
The cases $n=2$ and $n=3$ hold as explained in Example~\ref{example.comparing.V.and.W}. 
Let us assume that $V_{n-1,2}=V_{n-1,13}=V_{n-1,5}=V_{n-1,11}$ holds for some $n \geq 4$. %
By Equation~\eqref{equation.V.in.terms.of.W} we deduce that 
\begin{equation} \label{equation.comparing.consecutive.V.2}
V_{n,2}-V_{n-1,2}
= W_{n,2}- \sum_{0 \leq m \leq n-4}  V_{m+2,11} 
- \left(W_{n-1,2}- \sum_{0 \leq m \leq n-5}  V_{m+2,11}\right)
= W_{n,2} - W_{n-1,2} - V_{n-2,11}. 
\end{equation} 
Using that $V_{n,5}=W_{n,5}$ by part \ref{proposition.comparing.V.and.W.part.a} and Equation~\eqref{equation.comparing.consecutive.V.2} we get the following expression for $V_{n,2} - V_{n,5}$, 
\begin{equation} \label{equation.comparing.V.2.and.V.5}
V_{n,2} - V_{n,5}
= W_{n,2} - W_{n,5} + V_{n-1,2} - W_{n-1,2} - V_{n-2,11}. 
\end{equation} 
Comparing the second and fifth rows of the matrix $A$ we can read off the following relation  
\begin{equation}
 \label{equation.comparing.rows.of.A.1} 
W_{n,2} - W_{n,5} =  V_{n-1,4} + V_{n-1,8} - V_{n-1,12} - V_{n-1,13}.  
\end{equation}
Similarly, using that $V_{n-1,4} = W_{n-1,4}$ by part \ref{proposition.comparing.V.and.W.part.a} and comparing the fourth and second rows of the matrix $A$ we can read off the following relation 
\begin{equation}
\label{equation.comparing.rows.of.A.2}
V_{n-1,4} = W_{n-1,4} = W_{n-1,2} + V_{n-2,11}.  
\end{equation} 
Replacing Equations~\eqref{equation.comparing.rows.of.A.1}-\eqref{equation.comparing.rows.of.A.2} in Equation~\eqref{equation.comparing.V.2.and.V.5} and using that by induction $V_{n-1,2}=V_{n-1,13}$, we get
\begin{align*}
V_{n,2} - V_{n,5} &= (W_{n,2} - W_{n,5}) + V_{n-1,2} - ( W_{n-1,2} + V_{n-2,11}) \\ 
&= (V_{n-1,4} + V_{n-1,8} - V_{n-1,12} - V_{n-1,13}) + V_{n-1,13} - (V_{n-1,4}) \\
&= V_{n-1,8} - V_{n-1,12}.  
\end{align*} 
Since $V_{n-1,8}=\foursix_{n-1}=\fivefive_{n-1}=V_{n-1,12}$ by Lemma~\ref{lemma.basic.relations.Vs}\ref{lemma.basic.relations.Vs.b}, we conclude that $V_{n,2} = V_{n,5}$. 
By symmetry (or alternatively by Lemma~\ref{lemma.basic.relations.Vs}\ref{lemma.basic.relations.Vs.a}) we have that 
$V_{n,2}=\onefive_{n}=\fivethree_{n}=V_{n,13}$
and 
$V_{n,5}=\fourtwo_{n}=\twosix_{n-1}=V_{n,11}$. 
Therefore, $V_{n,2}=V_{n,13}=V_{n,5}=V_{n,11}$, which completes the argument for part \ref{proposition.comparing.V.and.W.part.b}.  \qedhere
\end{proof}


\subsection{Proof of Theorem~\ref{recurrence.2D} and some consequences}
\label{sec_proof_2D}

\begin{proof}[Proof of Theorem~\ref{recurrence.2D}]
From Proposition~\ref{proposition.comparing.V.and.W} and Lemma~\ref{lemma.basic.relations.Vs}, for all $n \geq 2$ we have
\begin{align*}
&V_{n,1}=V_{n,10}, &V_{n,2}=V_{n,5}=V_{n,11}=V_{n,13},   \\
&V_{n,4}=V_{n,7}, & V_{n,6}=V_{n,8}=V_{n,12}=V_{n,14}.
\end{align*}
Hence, we can write all the $V_{n,i}$ in terms of $V_{n,1}$, $V_{n,2}$, $V_{n,3}$, $V_{n,4}$, $V_{n,6}$ and $V_{n,9}$.  Then, the equation $V_{n}=\sum_{1 \leq i \leq 14} V_{n,i}$ becomes 
\begin{equation} \label{formula.vn}
V_{n}= 2V_{n,1} + 4V_{n,2} + V_{n,3} + 2V_{n,4} + 4V_{n,6} + V_{n,9}.
\end{equation}
For each $n \geq 2$, let $U_n$ be the column vector with six entries equal to $V_{n,1}$, $V_{n,2}$, $V_{n,3}$, $V_{n,4}$, $V_{n,6}$ and $V_{n,9}$, in that order. Notice that the initial value $U_2$ is a column vector with six entries all equal to 1. 
Using the relations between the numbers $V_{n,i}$ encoded in the rows of the matrix $A$ in Definition~\ref{definition.matrix.A} we conclude that $U_{n}=BU_{n-1}$ where $B$ is the following $6 \times 6$ matrix
\[B = \begin{bmatrix}
 2 & 2 & 1 & 1 & 1 & 1 \\ 
 2 & 3 & 1 & 1 & 2 & 1 \\ 
 2 & 2 & 1 & 2 & 1 & 1 \\ 
 2 & 3 & 1 & 2 & 2 & 1 \\ 
 2 & 4 & 1 & 2 & 4 & 1 \\ 
 2 & 2 & 1 & 0 & 1 & 1 \\ 
\end{bmatrix}.  
\]
If we let $W=[2,4,1,2,4,1]$ be the row matrix whose entries are the coefficients of $V_{n,1}$, $V_{n,2}$, $V_{n,3}$, $V_{n,4}$, $V_{n,6}$ and $V_{n,9}$  in $V_n$ in the expression (\ref{formula.vn}), we deduce that $V_n$ is given by 
\[
V_n=WB^{n-2}U_2
\]
for all $n \geq 2$. Since $W$ is equal to the fifth row of $B$ and $U_2$ is equal to the sixth column of $B$, it follows that $V_n$ is equal to the entry in position $(5,6)$ of the matrix $B^{n}$ for all $n \geq 2$. 
By \eqref{eq: def matrix series} we have that 
\begin{align*}
x + \sum_{n\geq 2} V_n x^n &= \frac{(-1)^{11} Q_{5,6}(x)}{Q(x)}= \frac{x+x^2-x^3}{1-13x+31x^2-20x^3+4x^4}
\end{align*}
where $Q(x)=\det(I-xB)$ and $Q_{5,6}(x)$ is the determinant of the submatrix $(Q(x);6,5)$, as desired. 
The recurrence relation for $V_n$ in (\ref{recurrence.2D.in.theorem}) now follows from  Theorem~\ref{thm:GeneratingFunction}~(ii) and the claimed initial values are read off from the entries in position $(5,6)$ of the first five powers of $B$.
\end{proof}


\begin{remark}
From the recurrence in Theorem~\ref{recurrence.2D} 
standard methods provide a closed expression for the number of vertices $V_n$ of $Q_n$ of the form $V_n=a\alpha_1^n+b\alpha_2^n+c\alpha_3^n+d\alpha_4^n$ where $a$, $b$, $c$ and $d$ are constants and 
\begin{align}
\alpha_1&=2, \nonumber \\
\label{alpha2} \alpha_{2} &= 
\frac{1}{3} \left(11 + \sqrt[3]{\frac{1}{2} (1825 - 3 \sqrt{921})} + \sqrt[3]{\frac{1}{2} (1825 + 3 \sqrt{921})}\right) \approx 10.1311, \\
 \alpha_{3} &= \frac{11}{3} - \frac{1}{6} (1-i\sqrt{3}) \sqrt[3]{\frac{1}{2} (1825 - 3 \sqrt{921})} - \frac{1}{6} (1 + i \sqrt{3}) \sqrt[3]{\frac{1}{2} (1825 + 3 \sqrt{921})}, \nonumber \\
\alpha_{4} &= \frac{11}{3} - \frac{1}{6} (1 + i \sqrt{3}) \sqrt[3]{\frac{1}{2} (1825 - 3 \sqrt{921})} - \frac{1}{6} (1-i\sqrt{3}) \sqrt[3]{\frac{1}{2} (1825 + 3 \sqrt{921})},  \nonumber
\end{align}
are the nonzero roots of the characteristic polynomial $x^2(x-2)(x^3 - 11 x^2 + 9 x - 2)$ of the matrix $B$ above. 
Notice that the roots of the polynomial in the denominator of the generating function for $V_n$ in \eqref{eq:gf Vn} are $1/\alpha_1$, $1/\alpha_2$, $1/\alpha_3$ and $1/\alpha_4$.
Here we omit the explicit closed expression $V_n=a\alpha_1^n+b\alpha_2^n+c\alpha_3^n+d\alpha_4^n$ since the exact values of the constants $a$, $b$, $c$ and $d$ are rather lengthy, but the interested reader can compute them from the data in Theorem~\ref{recurrence.2D}. 
\end{remark}

Analogously to the one-dimensional case, we obtain general asymptotics for $V_n$ in the following corollary.

\begin{corollary} \label{cor: asymptotics gf vertices 2 dim}
The number of vertices $V_n$ of the polytope $Q_n$ satisfies

\[
\lim_{n\to \infty} \frac1n \ln V_n = \ln \alpha_2 
\approx \ln 10.1311
\approx 2.3156.
\]
\end{corollary}

\begin{proof}
The result follows from Theorem~\ref{recurrence.2D}, the Perron--Frobenius theorem (Theorem~\ref{thm: main asymptotic theorem}), and the fact that $1/\alpha_2$ in \eqref{alpha2} is the smallest positive root of the denominator in \eqref{eq:gf Vn}.
\end{proof}


\section{Final remarks} 
\label{sec:conclusions}

\subsection{Other models for faces of generalized permutohedra}

Faces of generalized permutohedra $P=\sum_i \Delta_{\lambda_i}$ obtained from Minkowski sums of simplices were previously studied in work of Postnikov \cite{10.1093/imrn/rnn153}, Postnikov--Reiner--Williams \cite{PRW}, Agnarsson \cite{Agnarsson2009,Agnarsson2013}, and Benedetti--Bergeron--Machacek \cite{BBM}. The characterization of faces of $P$ via acyclic graphs in Section~\ref{sec_faces_dag} uses the fact that $P$ is a generalized permutohedron and the machinery of preposets from \cite{PRW}. However, the model in \cite{BBM} also uses an acyclic type object. More precisely, the authors call the polytopes $P$, {\em hypergraphic polytopes} and index them by the hypergraph $H$ with vertices $\cup_i \lambda_i$ and hyperedges $\{\lambda_i\}$. They show in \cite[Thm.\ 2.18]{BBM} that faces of $P$ are in correspondence with {\em acyclic orientations} of $H$ defined in \cite{BB} in the context of {\em Hopf algebras}.

\subsection{Combinatorial proofs of recurrences for vertices}

Lemma~\ref{lem:TowardsNums1} gives an algebraic proof of the recurrence \eqref{eq:recurrence first vals b} for the initial values of the number $b^{(k,s)}_n$ of vertices of the polytope $P_{n,k,s}$ for proportional strides $k=s(r+1)$. It would be of interest to give a combinatorial proof of this recurrence like we do in the proof of Theorem~\ref{thm:LargeStrides}.

\subsection{Connection to counting moves of a chess piece} 
The truncated  sequences $(b_n^{(k,1)})_{n=1}^{k+2}$ for $k=3,4,5$
appear in \cite[{\href{https://oeis.org/A045891}{A045891},\href{https://oeis.org/A087447}{A087447},\href{https://oeis.org/A098156}{A098156}}]{OEIS}, where they enumerate move configurations of a {\em fairy chess piece} (see, e.g.\  \cite[\href{https://oeis.org/A175655}{A175655}]{OEIS} and references therein). We leave as an open problem to give a bijection between such configurations and the objects counting the vertices $b^{(k,1)}_n$.

\subsection{Finding the numerator of generating functions of the number of vertices with determinants} \label{rem:numerator}
The proofs of some results in Section~\ref{sec:enumeration} rely on the using the relationship $k=s(r+1)$ and the induced block decomposition in the matrices. We also did not compute the numerator $P(x)$ from its original determinantal expression \eqref{eq:numeratorGF prop srides}. 
It would be interesting to use this formula to finish the computation. Note that using linear algebra one  has that 
\[
P(x) = -\frac{d}{dx} \det(Q(x)) + \sum_{(i,j) \in S} (-1)^{i+j} \det(Q(x);j,i),
\]
where $S=\{(i,j) \in \mathbb{N}^2 \mid i=s+1,\ldots,k, j=1,\ldots,k-s, i\neq j+s\}$. This reduces the sum from having $k^2$ determinants to having $(k-s)(k-s-1)$ of them.

\subsection{Enumerating other faces in the one- and two-dimensional cases}
In Section~\ref{sec:enumeration} we enumerated the vertices and facets of the polytopes $P_{n,k,s}$ using the characterization of faces of Minkowski sums of simplices from Section~\ref{sec:acyclic.graph} and the transfer-matrix method for vertices and using a direct calculation for facets. It would be of interest to enumerate other faces. See Table~\ref{table:edges one dim} and Table~\ref{table:faces one dim} for data on the number of edges and the total number of faces of the polytopes $P_{n,k,s}$, respectively obtained using an implementation \cite{code}of the graphs $G_{\Pi}$ in Section~\ref{sec_faces_dag} in SageMath \cite{sagemath}.

\begin{table}[h]
$$\begin{array}{lrrrrrrrrrrrrrrr}
\hline
k \backslash n& 1&  2& 3& 4& 5& 6& 7& 8&9&10  \\ \hline
3& 
3& 11& 34& 96& 260& 683& 1757& 4447& 11114&27493\\
4&	6& 21& 64& 180& 480& 1252& 3204& 8088&20208&50056\\
5& 10& 34& 102& 284& 752& 1920& 4822&11935&29223&70964 \\
6& 15& 50& 148& 408& 1072& 2720&6720&16368&39376&93824\\
 \hline
\end{array}
$$
\caption{Initial terms for the number of edges of the polytopes $P_{n,k,1}$.}
\label{table:edges one dim}
\end{table}

\begin{table}[h]
$$\begin{array}{lrrrrrrrrrrrrrrr}
\hline
k \backslash  n& 1&  2& 3& 4& 5& 6& 7& 8&9&10   \\ \hline
3& 
8& 26& 88& 298& 1016& 3466& 11832& 40394& 137912&470858\\
4&	16& 58& 208& 730& 2512& 8650& 29728& 102106&350704&1204426\\
5& 32& 122& 448& 1594& 5536& 18874& 64208&217738&737184&2494042 \\
6& 64& 250& 928& 3322& 11584& 39610&133408&447562&1495504&4984570\\
 \hline
\end{array}
$$
\caption{Initial terms for the total number of faces of the polytopes $P_{n,k,1}$.}
\label{table:faces one dim}
\end{table}

In Section~\ref{sec:2dim} we enumerated the vertices of the polytope $Q_n$ which correspond to linearity regions for the max-pooling method on a $3\times n$ board using windows of size $2\times 2$. It would be of interest to enumerate their facets, like we did for the polytope $P_{n,k,s}$ in Section~\ref{sec_num_facets}, and other faces. The number of facets of $Q_n$ for $n=2,3,4,5$ is $8,21,40,67$, respectively.
The code for the calculations in this article is available at \cite{code}.

\subsection*{Acknowledgments} 

This project was initiated at the Latinx Mathematicians Research Community (LMRC) kick-off workshop held by the American Institute of Mathematics (AIM) in June 2021. 
We are 
especially grateful to the organizers of the LMRC Jes\'us A.\ De Loera and Pamela E.\ Harris—and to AIM and NSF for 
funding the LMRC which led to this collaboration and future related opportunities. 
We thank Federico Ardila for pointing us to the Perron--Frobenius theorem, John Machacek and the anonymous referee for helpful comments. 
GM acknowledges early discussions about the linearity regions of max-pooling layers with Thomas Merkh. 
This work was facilitated by computer experiments using Sage \cite{sagemath} and its geometric  combinatorics features developed by the Sage-Combinat community \cite{Sage-Combinat}. 

This project has been partly supported by UCLA FCDA. 
LE has been partially supported by NSF Grant DMS-1855598 and NSF CAREER Grant DMS-2142656. 
JLG has been supported by the Simons Foundation, Award Number 710443. 
GM has been supported by grants ERC 757983, DFG 464109215, NSF 2145630, NSF 2212520. 
AHM has been partially supported by NSF Grants DMS-1855536 and DMS-22030407. 
PG and JGA thank the University of California, Riverside for the welcoming environment.


\end{document}